\documentclass[11pt]{amsart}%
\usepackage{palatino, mathpazo}
\usepackage{amsfonts}
\usepackage{amsmath}
\usepackage{amssymb,latexsym}
\usepackage{graphicx}
\usepackage[mathscr]{eucal}
\usepackage{amssymb}%
\usepackage{color}
\providecommand{\U}[1]{\protect \rule{.1in}{.1in}}
\newtheorem{theorem}{Theorem}[section]

\newtheorem{corollary}[theorem]{Corollary}

\newtheorem{definition}[theorem]{Definition}

\newtheorem{lemma}[theorem]{Lemma}

\newtheorem{proposition}[theorem]{Proposition}
\theoremstyle{remark}
\newtheorem{remark}[theorem]{Remark}

\numberwithin{equation}{section}
\begin{document}
\title[Proof of a conjecture of Dahmen and Beukers]{Proof of a conjecture of Dahmen and Beukers\\ on counting integral Lam\'{e} equations\\ with finite monodromy}
\author{Zhijie Chen}
\address{Department of Mathematical Sciences, Yau Mathematical Sciences Center,
Tsinghua University, Beijing, 100084, China}
\email{zjchen2016@tsinghua.edu.cn}
\author{Ting-Jung Kuo}
\address{Department of Mathematics, National Taiwan Normal University, Taipei 11677,
Taiwan }
\email{tjkuo1215@ntnu.edu.tw}
\author{Chang-Shou Lin}
\address{Center for Advanced Study in
Theoretical Sciences, National Taiwan University, Taipei 10617, Taiwan }
\email{cslin@math.ntu.edu.tw}

\begin{abstract}
In this paper, we prove Dahmen and Beukers' conjecture that
the number of integral
Lam\'{e} equations with index $n$ modulo scalar equivalence with the monodromy
group dihedral $D_{N}$ of order $2N$ is given by
\[L_{n}(N)=\frac{1}{2}\left(  \frac{n(n+1)\Psi(N)}{24}-\left(  a_{n}%
\phi(N)+b_{n}\phi(\tfrac{N}{2})  \right)  \right)  +\frac{2}%
{3}\varepsilon_{n}(N).\]
Our main tool is the new pre-modular form
$Z_{r,s}^{(n)}(\tau)$ of weight $n(n+1)/2$ introduced by Lin and Wang \cite{LW2} and the associated modular form $M_{n,N}(\tau):=\prod_{(r,s)}Z_{r,s}^{(n)}(\tau)$ of weight $\Psi(N)n(n+1)/{2}$, where the product runs over all $N$-torsion points $(r,s)$ of exact order $N$. We show that this conjecture is equivalent to the precise formula of the vanishing order of $M_{n,N}(\tau)$ at infinity:
\[v_{\infty}(M_{n,N}(\tau))=a_{n}\phi(N)+b_{n}\phi(
N/2).\]
This formula is extremely hard to prove because the explicit expression of $Z_{r,s}^{(n)}(\tau)$
is not known for general $n$. Here we succeed to prove it by using
certain Painlev\'{e} VI equations. Our result also indicates that this conjecture is intimately connected with the problem of counting pole numbers of algebraic
solutions of certain Painlev\'{e} VI equations. The main results of this paper
has been announced in \cite{Lin-CDM}.

\end{abstract}
\maketitle
\tableofcontents

\section{Introduction}

The main purpose of this paper is to prove a conjecture about counting the number of the equivalent
class of integral Lam\'{e} equations with $D_{N}$ (i.e. dihedral of order
$2N$) as its monodromy group, which was proposed by S. Dahmen
and F. Beukers \cite{Dahmen0} and partially confirmed in \cite{Dahmen0,Dahmen,LW2}.
The problem of counting integral Lam\'{e} equations with finite monodromy is a long-standing open problem that was already studied in
\cite{BD,Chi,Lit} before Dahmen-Beukers conjecture \cite{Dahmen0}. We will see that this problem is closely related to finding a precise formula of counting pole numbers of algebraic solutions of certain
Painlev\'{e} VI equations (see Remark \ref{remark1-8}), which will be studied in another work. The connection between the Lam\'{e} equation and Painlev\'{e} VI equation was already studied in \cite{LPU} from the view point of moduli spaces, which can not be applied to prove Dahmen-Beukers conjecture. Our idea is different from \cite{LPU}. Our proof of this conjecture is to apply the new pre-modular form $Z_{r,s}^{(n)}(\tau)$ (as a function of $\tau\in\mathbb{H}=\{\tau | \operatorname{Im}\tau>0\}$) introduced in \cite{LW2} to connect the Lam\'{e} equation with Painlev\'{e} VI equation. This
$Z_{r,s}^{(n)}(\tau)$ is called \emph{pre-modular} in \cite[Definition 0.1]{LW2} because
$Z_{r,s}^{(n)}(\tau)$ is a \emph{modular form of weight $n(n+1)/2$} with respect to the
principal congruence subgroup $\Gamma(N)=\{A\in SL(2,\mathbb{Z})|A\equiv
I_{2}\operatorname{mod}N\}$ for any given $N$-torsion point $(r,s)\in \mathcal{Q}
(N)$, where $\mathcal{Q}(N)$ is the set of $N$-torsion points of exact order $N$
defined by%
\begin{equation}
\mathcal{Q}(N):=\left \{  \left.  \left(  \tfrac{k_{1}}{N},\tfrac{k_{2}}%
{N}\right)  \right \vert \gcd(k_{1},k_{2},N)=1,0\leq k_{1},k_{2}<N\right \}
.\label{N-torsion-1}%
\end{equation}
Consequently, \begin{equation}
M_{n,N}(\tau):=\prod_{(r,s)\in \mathcal{Q}(N)}Z_{r,s}^{(n)}(\tau)
\label{modular-S1L}%
\end{equation}
is a \emph{modular form of weight} $\frac{n(n+1)}{2}\Psi(N)$
with respect to $SL(2,\mathbb{Z})$, where
\begin{equation}
\Psi(N):=\# \mathcal{Q}(N)=N^{2}\prod_{p|N,\text{ }p\text{ prime}}\left(
1-\tfrac{1}{p^{2}}\right)  .\label{Euler-2}%
\end{equation}
This $M_{n,N}(\tau)$ is difficult to study due to the absence of the explicit expressions for general $n\geq 5$.
Our main result of this paper is to calculate the vanishing order $v_{\infty}(M_{n,N}(\tau))$ of this modular form $M_{n,N}(\tau)$ at infinity.

\begin{theorem}\label{thm-vanish-order}
For any $n\in \mathbb{N}$ \textit{and} $N\in
\mathbb{N}_{\geq3}$, there holds \[v_{\infty}(M_{n,N}(\tau))=a_{n}\phi(N)+b_{n}\phi \left(
N/2\right),\] where
\begin{align}
\phi(N) &  :=\# \{k\in \mathbb{Z}|\gcd(k,N)=1,0\leq k<N\} \label{Euler}\\
&  =N\prod_{p|N,\text{ }p\text{ prime}}\left(  1-\tfrac{1}{p}\right)  \nonumber
\end{align}
is the Euler function (set $\phi(N)=0$ if $N\not \in \mathbb{N}$), and
\begin{equation}
a_{2n}=a_{2n+1}=n(n+1)/2,\text{ \ }b_{2n}=b_{2n-1}=n^{2}. \label{iv-23-5}%
\end{equation}
\end{theorem}
We will see in \S 1.2 that Theorem \ref{thm-vanish-order} is actually equivalent to the validity of Dahmen-Beukers conjecture. Our proof of Theorem \ref{thm-vanish-order} relies on two key steps: Step 1 is the formula of $PL_{n}(N)$ obtained by Dahmen \cite{Dahmen} (See (\ref{iv-24-1}) below for the definition of $PL_{n}(N)$); Step 2 is the asymptotics of $Z_{r,s}^{(n)}(\tau)$ as $\tau\to \infty$. We note that Step 1 is algebraic and Step 2 is analytic. Therefore, the complete proof of Theorem \ref{thm-vanish-order} has to apply both the algebraic theory and the analytic theory of the integral Lam\'{e} equation.

Throughout the paper, we use the notations $\omega_{0}=0$, $\omega_{1}=1$,
$\omega_{2}=\tau$, $\omega_{3}=1+\tau$ and $\Lambda_{\tau}=\mathbb{Z+Z}\tau$,
where $\tau \in \mathbb{H}=\{ \tau|\operatorname{Im}\tau>0\}$. Define $E_{\tau
}:=\mathbb{C}/\Lambda_{\tau}$ to be a flat torus and $E_{\tau
}[2]:=\{ \frac{\omega_{k}}{2}|0\leq k\leq3\}+\Lambda_{\tau}$ to be the set
consisting of the lattice points and 2-torsion points in $E_{\tau}$. For
$z\in \mathbb{C}$ we denote $[z]:=z \ (\operatorname{mod} \Lambda_{\tau}) \in E_{\tau}$.
For a point $[z]$ in $E_{\tau}$ we often write $z$ instead of $[z]$ to
simplify notations when no confusion arises.

Let $\wp(z)=\wp(z|\tau)$ be the Weierstrass $\wp$-function with periods
$\Lambda_{\tau}$, defined by%
\[
\wp(z|\tau):=\frac{1}{z^{2}}+\sum_{\omega \in \Lambda_{\tau}\backslash
\{0\}}\left(  \frac{1}{(z-\omega)^{2}}-\frac{1}{\omega^{2}}\right)  ,
\]
and $e_{k}(\tau):=\wp(\frac{\omega_{k}}{2}|\tau)$ for $k\in \{1,2,3\}$. Let
$\zeta(z)=\zeta(z|\tau):=-\int^{z}\wp(\xi|\tau)d\xi$ be the Weierstrass zeta
function with two quasi-periods:%
\begin{equation}
\eta_{1}(\tau)=\zeta(z+1|\tau)-\zeta(z|\tau),\text{ \ }\eta_{2}(\tau
)=\zeta(z+\tau|\tau)-\zeta(z|\tau), \label{quasi}%
\end{equation}
and $\sigma(z)=\sigma(z|\tau)$ be the Weierstrass sigma function defined by
$\frac{\sigma^{\prime}(z)}{\sigma(z)}:=\zeta(z)$. Remark that $\zeta(z)$ is an
odd meromorphic function with simple poles at the lattice points
$\Lambda_{\tau}$, while $\sigma(z)$ is an odd holomorphic function with simple
zeros at the lattice points $\Lambda_{\tau}$. We take \cite{Lang} as our general
reference on elliptic functions.

\subsection{Integral Lam\'{e} equation and Dahmen-Beukers conjecture}

Consider the classical Lam\'{e} equation
(\cite{Halphen,Ince,Poole,Whittaker-Watson}):%
\begin{equation}
y^{\prime \prime}(z)=\left[  n(n+1)\wp(z|\tau)+B\right]  y(z), \label{iv-21-1}%
\end{equation}
where $n>0$
and $B\in \mathbb{C}$ are called \emph{index} and \emph{accessory parameter} respectively.
Equation (\ref{iv-21-1}) was introduced by the French mathematician Gabriel
Lam\'{e} in 1837. Since then, the Lam\'{e} equation has led to an abundance of
applications in physics, especially in the context of completely integrable
models such as generalized Calogero-Moser type systems; see e.g.
\cite{AMM,Cal,Its,Krichever,OP,Perelomov} and references therein. In this
paper, we only consider the \textit{integral} Lam\'{e} equation, i.e.
$n\in \mathbb{N}$ is a \textit{positive integer}. Then the corresponding
Lam\'{e} potential $n(n+1)\wp(z|\tau)$ is also known as the simplest elliptic solutions of stationary KdV hierarchy (see e.g. \cite{GH-Book,Gesztesy-Weikard}).

The Lam\'{e} equation possesses symmetry from $SL(2,\mathbb{Z})$ action. Given any $\gamma=\bigl(\begin{smallmatrix}a & b\\
c & d\end{smallmatrix}\bigr)
\in SL(2,\mathbb{Z})$,
we let $\tau^{\prime}=\gamma \cdot \tau:=\frac{a\tau+b}{c\tau
+d}$ and make a scalar change of the independent variable
$z\mapsto(c\tau+d)z$, then (\ref{iv-21-1}) is transformed to%
\begin{equation}
y^{\prime \prime}(z)=\left[  n(n+1)\wp(z|\tau^{\prime})+(c\tau+d)^{2}B\right]
y(z). \label{iv-22-1}%
\end{equation}
These scalar changes of variable induce a natural \emph{equivalence relation} on the
space of all Lam\'{e} equations; we call such two Lam\'{e} equations
(\ref{iv-21-1}) and (\ref{iv-22-1}) \emph{scalar equivalent}.

By letting $x=\wp(z|\tau)$, we obtain the \emph{algebraic form} of Lam\'{e} equations%
\begin{equation}
p(x)\frac{d^{2}y}{dx^{2}}+\frac{1}{2}p^{\prime}(x)\frac{dy}{dx}-\left(
n(n+1)x+B\right)  y=0\;\text{ on }\;\mathbb{CP}^{1}, \label{iv-21-2}%
\end{equation}
where $p(x):=4x^{3}-g_{2}(\tau)x-g_{3}(\tau)$ and $g_{2}(\tau), g_{3}(\tau)$ are the well-known invariants of the elliptic curve $E_{\tau}$, given by
\[
\wp'(z|\tau)^{2}=4\prod_{k=1}^{3}(\wp(z|\tau)-e_{k}(\tau))=4\wp(z|\tau)^{3}%
-g_{2}(\tau)\wp(z|\tau)-g_{3}(\tau).
\]
Equivalently, two Lam\'{e} equations of the algebraic form (\ref{iv-21-2}) are
\emph{scalar equivalent }if one can be transformed into the other by changing
 variable $x\mapsto ax$ for some $a\in \mathbb{C}\backslash \{0\}$.

Classically, people are interested in (\ref{iv-21-2}) with \emph{algebraic solutions}
only, or equivalently with a \emph{finite monodromy group}; see
\cite{Bal,BD,Beukers-Waall,Chi,Lit,Lit1,Maier,vdW} and references therein. It is
known that for an integral Lam\'{e} equation, if the monodromy group is
finite, then the monodromy group of (\ref{iv-21-2}) is a dihedral group
$D_{N}$ for some $N\in \mathbb{N}_{\geq3}$; see e. g. \cite[Theorem 1.5]{Bal}
or \cite[Corollary 3.4]{Beukers-Waall}. \emph{Let $L_{n}(N)$ denote the number of
Lam\'{e} equations (\ref{iv-21-2}) modulo scalar equivalence with the monodromy
group dihedral $D_{N}$}. It is well known (cf. \cite{Chi,Beukers-Waall,Lit,vdW}%
) that $L_{n}(N)$ is \emph{finite} for given $n\in \mathbb{N}$ and
$N\in \mathbb{N}_{\geq3}$. The issue is how to compute it (cf.
\cite{BD,Beukers-Waall,Chi,Dahmen0,Lit,vdW}). In 2003 Dahmen and Beukers proposed the following conjecture (see
\cite[Conjecture 73]{Dahmen0}):\medskip

\noindent \textbf{Dahmen-Beukers conjecture.} \textit{For every} $n\in
\mathbb{N}$ \textit{and} $N\in \mathbb{N}_{\geq3}$,%
\begin{equation}
L_{n}(N)=\frac{1}{2}\left(  \frac{n(n+1)\Psi(N)}{24}-\left(  a_{n}%
\phi(N)+b_{n}\phi(\tfrac{N}{2})  \right)  \right)  +\frac{2}%
{3}\varepsilon_{n}(N), \label{Dahmen-conjecture-1}%
\end{equation}
\textit{where $(a_n, b_n)$ is given in (\ref{iv-23-5}), $\Psi(N)$ is in (\ref{Euler-2}),
$\phi(N)$ is in (\ref{Euler}) and}
\begin{equation}
\varepsilon_{n}(N)=\left \{
\begin{array}
[c]{l}%
1\text{ \  \textit{if} }N=3\text{ \textit{and} }n\equiv1\operatorname{mod}3,\\
0\text{ \  \textit{otherwise}.}%
\end{array}
\right.  \label{iv-23-1}%
\end{equation}

Denote $PL_{n}(N)$ to be the number of Lam\'{e}
equations (\ref{iv-21-2}) modulo scalar equivalence with the \emph{projective}
monodromy group dihedral $D_{N}$. It is not difficult to see that
when the monodromy group $M\simeq D_{N}$, then the projective monodromy group
$PM\simeq D_{N/2}$ if $N$ is even and $PM\simeq D_{N}$ if $N$ is odd. Thus
(cf. \cite[(4.1)]{Dahmen})%
\begin{equation}
PL_{n}(N)=\left \{
\begin{array}
[c]{l}%
L_{n}(N)+L_{n}(2N)\text{ \ if }N\text{ is odd,}\\
L_{n}(2N)\text{ \ if }N\text{ is even.}%
\end{array}
\right.  \label{iv-24-1}%
\end{equation}
On the other hand, the well-known Klein theorem asserts that
any second order Fuchsian differential equation on $\mathbb{CP}^{1}$ has
finite projective monodromy group if and only if it is a pull-back of a
hypergeometric equation belonging to the basic Schwarz list. Thus every
Lam\'{e} equation (\ref{iv-21-2}) with \emph{projective} monodromy group dihedral
$D_{N}$ is such a pull-back by \emph{Belyi functions }(see
\cite{Chi,Dahmen,Lit,Lit1}). By means of \emph{Grothendieck correpondence}
between \emph{Belyi pairs} and \emph{dessin d'enfants}, Litcanu \cite{Lit}
first showed how $PL_{n}(N)$ can be counted by using the combinatorics of
dessins. Later by applying this algebraic approach, Dahmen \cite{Dahmen} proved

\medskip
\noindent{\bf Theorem A.} \cite{Dahmen}
\begin{equation}
PL_{n}(N)=\left \{
\begin{array}
[c]{l}%
0\text{ \  \  \  \  \  \textit{if} }N\in \{1,2\} \text{,}\\
\frac{n(n+1)}{12}\left(  \Psi(N)-3\phi(N)\right)  +\frac{2}{3}\varepsilon
_{n}(N)\text{ \textit{otherwise},}%
\end{array}
\right.  \label{iv-25-1}%
\end{equation}
\emph{where $\varepsilon_{n}(N)$ is defined in (\ref{iv-23-1}), $\Psi(N)$ is in (\ref{Euler-2}) and
$\phi(N)$ is in (\ref{Euler}).}

\emph{Consequently, the formula (\ref{Dahmen-conjecture-1}) holds for all $n\in \mathbb{N}$ and $4|N$.}

\medskip

Theorem A, which confirms the conjecture for the case $n\in \mathbb{N}$ and $4|N$, will play an important role in our proof of Theorem \ref{thm-vanish-order}.
Besides, the conjecture for the case $n\in \{1,2,3,4\}$ and $N\in
\mathbb{N}_{\geq3}$ was proved in \cite{Dahmen, LW2}. See \S \ref{Dahmen's} for
more details on these results. As far as we know, the general case
$n\geq5$ and $4\nmid N$ still remains open.

\subsection{Dahmen-Beukers conjecture and Pre-modular forms}

The aforementioned
pre-modular form $Z_{r,s}^{(n)}(\tau)$ is
holomorphic in $\tau$ for any
real pair $(r,s)\in \mathbb{R}^{2}\backslash \frac{1}{2}\mathbb{Z}^{2}$.
 For $n\in \{1,2,3\}$,
$Z_{r,s}^{(n)}(\tau)$ was constructed in
\cite{Dahmen0} by using the classical Frobenius-Stickelberger formula (e.g.
\cite[p.458]{Whittaker-Watson}). However, this idea can not work for $n\geq4$.
Based on the earlier works \cite{CLW,LW}, Wang and the third author \cite{LW2}
succeeded to prove the existence of such pre-modular forms.
We will briefly
review this theory in \S \ref{Dahmen's}.1.
Furthermore, we will see in Theorem \ref{simple-zero con} that $Z_{r,s}^{(n)}(\tau)$ has at most \emph{simple} zeros. This simple zero property plays a fundamental role in the connection of Theorem \ref{thm-vanish-order} and Dahmen-Buekers conjecture. Indeed, this connection has already been pointed out by Dahmen \cite[Lemma 65]{Dahmen0}.

As mentioned before, $Z_{r,s}^{(n)}(\tau)$ is a modular form of weight $n(n+1)/2$ with respect to $\Gamma(N)$ for $(r,s)\in \mathcal{Q}(N)$, and $M_{n,N}(\tau)$ defined in (\ref{modular-S1L}) is a modular form of weight $\frac{n(n+1)}{2}\Psi(N)$ with respect to $SL(2,\mathbb{Z})$. Since $Z_{r,s}^{(n)}(\tau)$ satisfies the hypothesis of \cite[Lemma 65]{Dahmen0} which relates Theorem \ref{thm-vanish-order} with Dahmen-Beukers conjecture, then \cite[Lemma 65]{Dahmen0} says that the simple zero property of $Z_{r,s}^{(n)}(\tau)$ implies
\[L_n(N)=\frac12\left(  \frac{n(n+1)\Psi(N)}{24}-v_{\infty}(M_{n,N}(\tau))  \right)  +\frac{2}%
{3}\varepsilon_{n}(N).
\]
We will recall the precise statement of \cite[Lemma 65]{Dahmen0} in Lemma \ref{lem-C} below. This identity clearly shows that Theorem \ref{thm-vanish-order} is equivalent to the validity of
Dahmen-Beukers conjecture, i.e.

\begin{theorem}[=Theorem \ref{thm-vanish-order}]\label{thm-Dahmen-Conjecture}
The formula (\ref{Dahmen-conjecture-1}) holds for all $n\in
\mathbb{N}$ and $N\in \mathbb{N}_{\geq3}$.
\end{theorem}

In order to prove
Theorem \ref{thm-vanish-order}, we need to study the asymptotics of $Z_{r,s}%
^{(n)}(\tau)$ as $\tau \rightarrow \infty$. For our purpose, it suffices to consider $\tau \in
F_{2}$, where $F_{2}$ is a fundamental domain of $\Gamma(2)$ defined by
\[
F_{2}:=\{ \tau \in \mathbb{H}\text{ }|\text{ }0\leq \operatorname{Re}%
\tau<2,\text{ }|\tau-1/2|\geq1/2,\text{ }|\tau-3/2|>1/2\}.
\]
Denote $q=e^{2\pi i\tau}$. Clearly
$q\rightarrow0$ as $F_{2}\ni \tau \rightarrow \infty$. Then we have the following result for $Z_{r,s}^{(n)}(\tau)$.

\begin{theorem}
\label{weak}
Given any $n\geq1$ and $(r,s)\in \mathbb{C}^{2}\backslash \frac
{1}{2}\mathbb{Z}^{2}$. Then as $F_{2}\ni \tau \rightarrow \infty$ the following hold.
\begin{itemize}
\item[(1)]
$\lim_{F_{2}\ni \tau \rightarrow \infty}Z_{r,s}^{(n)}(\tau)$ exists and is not zero as long
as $\operatorname{Re}s\in (0,1/2)  \cup(1/2,1)$.
\item[(2)] There exist $\tilde{a}_n, \tilde{b}_n\in \mathbb{Z}_{\geq 0}$ satisfying
$2\tilde{a}_n+\tilde{b}_n=2a_n+b_n$
such that
\[
Z_{r,0}^{(n)}(\tau)=\alpha_{0}^{(n)}(r)q^{\tilde{a}_{n}}+O(|q|^{\tilde{a}_{n}+1}),
\]
\[
Z_{r,\frac{1}{2}}^{(n)}(\tau)=\beta_{0}^{(n)}(r)q^{\frac{\tilde{b}_{n}}{2}%
}+O(|q|^{\frac{\tilde{b}_{n}+1}{2}}),
\]
where both $\alpha_{0}^{(n)}(r)$ and $\beta_{0}^{(n)}(r)$ are holomorphic in
$\mathbb{C}\backslash \mathbb{Z}$ and have no rational zeros in $(0,1/2)\cup
(1/2,1)$. Here $(a_{n},b_{n})$ is given by (\ref{iv-23-5}).
\item[(3)] For $(r,s)=(\frac14,0)$,
\[Z_{\frac14,0}^{(n)}(\tau)=c_n q^{a_n}+O(|q|^{a_n+1}),\quad c_n\neq 0.\]
In particular, $\tilde{a}_n=a_n$ and so $\tilde{b}_n=b_n$.
\end{itemize}

\end{theorem}

It is easy to see that
Theorem \ref{thm-vanish-order} is a consequence of Theorem \ref{weak}. The proof of Theorem \ref{weak} is not trivial at all due to the absence of explicit expressions of $Z_{r,s}^{(n)}(\tau)$ for general $n\geq 5$. Here we state Theorem \ref{weak} (2) and (3) separately to emphasize that the proof of Theorem \ref{weak} (2) can not yield  $(\tilde{a}_n,\tilde{b}_n)=(a_n, b_n)$, and we have to compute $Z_{\frac14,0}^{(n)}(\tau)$ as stated in (3) to prove $(\tilde{a}_n,\tilde{b}_n)=(a_n, b_n)$. We will apply Painlev\'{e} VI equation to prove Theorem \ref{weak} (1) \& (3), while Theorem A plays a key role in the proof of Theorem \ref{weak} (2).
As an application, we recall the classical modular function $j(\tau)$ of $SL(2,\mathbb{Z}%
)$:
\[
j(\tau):=1728\frac{g_{2}(\tau)^{3}}{g_{2}(\tau)^{3}-27g_{3}(\tau)^{2}%
}=1728\frac{g_{2}(\tau)^{3}}{\Delta(\tau)}.
\]

\begin{corollary}\label{degree}
Assume either $N>3$ or $n\not \equiv 1\operatorname{mod}3$. Then%
\[
M_{n,N}(\tau)=C_{n,N}\Delta(\tau)^{k(n,N)}\ell_{n,N}(j(\tau))^{2},
\]
where $k(n,N)=\frac{n(n+1)}{24}\Psi(N)$, $\ell_{n,N}$ is a monic polynomial
and $C_{n,N}$ is a non-zero constant. Furthermore, $\deg \ell_{n,N}=L_{n}(N)$,
where $L_{n}(N)$ is precisely the number in Dahmen-Beukers conjecture.
\end{corollary}

In general, the polynomial $\ell_{n,N}$ is very difficult to compute even for
small $n$ and $N$. In our joint work with Wang \cite{CKLW}, we studied the
case $n=1$ and computed $\ell_{1,N}$ for $N\leq9$ by applying Painlev\'{e} VI
equation. In general, we conjectured that $\ell_{1,N}(j)$ \emph{is irreducible
in} $\mathbb{Q}[j]$. Once this conjecture can be proved, all the zeros of
$\ell_{1,N}(j)$ should not be algebraic integers provided $N\geq5$, which
implies that all the corresponding $\tau$'s are transcendental.

\subsection{Painlev\'{e} VI equation}
Since it is impossible to write down the explicit expression of $Z_{r,s}^{(n)}(\tau)$
for general $n$, Theorem \ref{weak} can not be proved by direct computations. In this paper, we overcome this difficulty by establishing a
precise connection between $Z_{r,s}^{(n)}(\tau)$ and Painlev\'{e} VI equation.

The well-known Painlev\'{e} VI equation with free parameters
$(\alpha,\beta,\gamma,\delta)$ (denoted it by PVI$(\alpha,\beta,\gamma,\delta)$) reads
as%
\begin{align}
\frac{d^{2}\lambda}{dt^{2}}=  &  \frac{1}{2}\left(  \frac{1}{\lambda}+\frac
{1}{\lambda-1}+\frac{1}{\lambda-t}\right)  \left(  \frac{d\lambda}{dt}\right)
^{2}-\left(  \frac{1}{t}+\frac{1}{t-1}+\frac{1}{\lambda-t}\right)
\frac{d\lambda}{dt}\nonumber \\
&  +\frac{\lambda(\lambda-1)(\lambda-t)}{t^{2}(t-1)^{2}}\left[  \alpha
+\beta \frac{t}{\lambda^{2}}+\gamma \frac{t-1}{(\lambda-1)^{2}}+\delta
\frac{t(t-1)}{(\lambda-t)^{2}}\right]  . \label{46-0}%
\end{align}
Due to its connection with many different disciplines in mathematics and
physics, PVI (\ref{46-0}) has been extensively studied in the past several
decades. See
\cite{Boalch,Dubrovin-Mazzocco,Guzzetti,Guzzetti1,Hit1,Hit2,GP,Lisovyy-Tykhyy,Y.Manin,Okamoto2,Okamoto1}
and references therein.

One of the fundamental properties for PVI (\ref{46-0}) is the so-called
\emph{Painlev\'{e} property} which says that any solution $\lambda(t)$ of
(\ref{46-0}) has neither movable branch points nor movable essential
singularities; in other words, for any $t_{0}\in \mathbb{C}\backslash \{0,1\}$,
either $\lambda(t)$ is holomorphic at $t_{0}$ or $\lambda(t)$ has a pole at $t_{0}%
$. Moreover, $\lambda(t)$ has at most simple poles if $\alpha \not =0$ (see
Theorem \ref{thm-2A} in \S \ref{expression-completely}).

By the Painlev\'{e} property, it is reasonable to lift PVI (\ref{46-0}) to the
covering space $\mathbb{H=}\{ \tau|\operatorname{Im}\tau>0\}$ of
$\mathbb{C}\backslash \{0,1\}$ by the following transformation:%
\begin{equation}
t=\frac{e_{3}(\tau)-e_{1}(\tau)}{e_{2}(\tau)-e_{1}(\tau)},\text{ \ }%
\lambda(t)=\frac{\wp(p(\tau)|\tau)-e_{1}(\tau)}{e_{2}(\tau)-e_{1}(\tau)}.
\label{II-130}%
\end{equation}
Then it is well known (cf.
\cite{Babich-Bordag,Y.Manin,Painleve}) that $p(\tau)$ satisfies the elliptic form of Painlev\'{e} VI equation (denoted it by EPVI$(\alpha_{0},\alpha_{1},\alpha_{2},\alpha_{3})$):
\begin{equation}
\frac{d^{2}p(\tau)}{d\tau^{2}}=\frac{-1}{4\pi^{2}}\sum_{k=0}^{3}\alpha_{k}%
\wp^{\prime}\left(  \left.  p(\tau)+\tfrac{\omega_{k}}{2}\right \vert
\tau \right)  , \label{124-0}%
\end{equation}
with parameters given by
\[
\left(  \alpha_{0},\alpha_{1},\alpha_{2},\alpha_{3}\right)  =\left(
\alpha,-\beta,\gamma,\tfrac{1}{2}-\delta \right)  .
\]
The Painlev\'{e} property implies that $\wp(p(\tau)|\tau)$ is
 single-valued meromorphic in $\mathbb{H}$. This is an advantage of
making the transformation (\ref{II-130}).

In this paper, we only consider the
elliptic form (\ref{124-0}) with parameters%
\begin{equation}
(\alpha_{0},\alpha_{1},\alpha_{2},\alpha_{3})=(\tfrac{1}{2}(n+\tfrac{1}%
{2})^{2},\tfrac{1}{8},\tfrac{1}{8},\tfrac{1}{8}),\text{ }n\in \mathbb{N}\cup \{0\}, \label{parameter-n-1}%
\end{equation}
or equivalently PVI (\ref{46-0}) with parameter%
\begin{align}\label{parameter-n}
(\alpha,\beta,\gamma,\delta)=  (  \tfrac{1}{2}(n+\tfrac{1}{2}
)^{2},\tfrac{-1}{8},\tfrac{1}{8},\tfrac{3}{8}),
\end{align}
which is related to the integral Lam\'{e} equation. For any monodromy data $(r,s)\in\mathbb{C}^2\setminus\frac{1}{2}\mathbb{Z}^2$, we proved in \cite{Chen-Kuo-Lin} that there is \emph{a unique solution} $p_{r,s}^{(n)}(\tau)$ of EPVI$(\tfrac{1}{2}(n+\tfrac{1}%
{2})^{2},\tfrac{1}{8},\tfrac{1}{8},\tfrac{1}{8})$ associated with this $(r,s)$; we will recall this statement in Theorem \ref{thm-II-8}. To study the asymptotics of $Z_{r,s}^{(n)}(\tau)$, we need the following "explicit" expression of
$\wp(p_{r,s}^{(n)}(\tau)|\tau)$.

\begin{theorem}
\label{thm-II-18 copy(1)}Fix any
$n\geq0$ and $(r,s)\in \mathbb{C}^{2}\backslash \frac{1}{2}\mathbb{Z}^{2}$. Then
$p_{r,s}^{(n)}(\tau)$ is expressed by%
\[
\wp(p_{r,s}^{(n)}(\tau)|\tau)=\frac{P_{n}(Z_{r,s}(\tau);r+s\tau,\tau)}%
{Z_{r,s}^{(n-1)}(\tau)Z_{r,s}^{(n+1)}(\tau)},
\]
where $Z_{r,s}^{(-1)}
(\tau)=Z_{r,s}^{(0)}(\tau):=1$,
\[Z_{r,s}(\tau):=\zeta(r+s\tau|\tau)-r\eta_{1}(\tau)-s\eta_{2}(\tau),\]
and $P_{n}(\cdot;r+s\tau,\tau)$ are polynomials with
coefficients being rational functions of $\wp(r+s\tau|\tau)$, $\wp^{\prime
}(r+s\tau|\tau)$ and $e_{k}(\tau)$, $k\in \{1,2,3\}$.

Furthermore, any two of
 $P_{n}(Z_{r,s}(\tau);r+s\tau,\tau)$, $Z_{r,s}^{(n-1)}(\tau)$,
 $Z_{r,s}^{(n+1)}(\tau)$ have no common zeros in $\tau$ provided $r+s\tau
\not \in E_{\tau}[2]$.
\end{theorem}

\begin{remark}
In the seminal work \cite{Hit1}, Hitchin proved his famous formula for EPVI$(\frac{1}{8},\frac{1}{8},\frac{1}{8},\frac{1}{8})$:
\begin{equation}
\wp(p^{(0)}_{r,s}(\tau)|\tau)=\wp(r+s\tau|\tau)+\frac{\wp^{\prime}(r+s\tau|\tau)}%
{2Z_{r,s}(\tau)}. \label{0II-1}%
\end{equation}
Therefore, Theorem \ref{thm-II-18 copy(1)} is a generalization of Hitchin's formula to the general case $n\geq 1$. Theorem \ref{thm-II-18 copy(1)} plays a crucial role in our proof of Theorem \ref{weak} (1) \& (3). We believe that Theorem \ref{thm-II-18 copy(1)} has potential applications to some other problems of Painlev\'{e} VI equations.
\end{remark}

\begin{remark}
\label{remark1-8} We will see in \S \ref{expression-completely} that $p_{r,s}^{(n)}(\tau)$ can be obtained from $p_{r,s}^{(0)}(\tau)$ via the famous Okamoto transformation \cite{Okamoto1}.  Since the Okamoto transformation preserves algebraic
solutions, it is known (cf. \cite{Mazzocco,CKLW}) that $\wp(p_{r,s}^{(n)}%
(\tau)|\tau)$ gives an algebraic solution of PVI$(\tfrac{1}{2}(n+\tfrac{1}%
{2})^{2},\tfrac{-1}{8},\tfrac{1}{8},\tfrac{3}{8})$ if and only if
$(r,s)\in \mathcal{Q}(N)$ for some $N\ge3$. Moreover, after analytic
continuation, $\wp(p_{r,s}^{(n)}(\tau)|\tau)=\wp(p_{1-r,1-s}^{(n)}
(\tau)|\tau)$ gives the same algebraic solution for all $(r,s)\in
\mathcal{Q}(N)$ (i.e. this algebraic solution has $\Psi(N)/2$ branches) if $N$
is odd; $\wp(p_{r,s}^{(n)}(\tau)|\tau)$ gives three different algebraic
solutions for all $(r,s)\in \mathcal{Q}(N)$ (i.e. each algebraic solution has
$\Psi(N)/6$ branches) if $N$ is even. See also \cite[Lemma 3.3]
{Chen-Kuo-Lin2016}. Since Theorem \ref{thm-II-18 copy(1)} shows that zeros of
$Z_{r,s}^{(n-1)}(\tau)$, $Z_{r,s}^{(n+1)}(\tau)$ give poles of $\wp
(p_{r,s}^{(n)}(\tau)|\tau)$, the zero numbers of $M_{n\pm1,N}(\tau)$ and hence Dahmen-Beukers
conjecture, are closely related to counting pole numbers of algebraic
solutions of PVI$(\tfrac{1}{2}(n+\tfrac{1}{2})^{2},\tfrac{-1}{8},\tfrac{1}%
{8},\tfrac{3}{8})$. This precise relation for the case $n=1$ has been
discussed in \cite{Chen-Kuo-Lin2016}, where an explicit formula of the pole
number of algebraic solutions for PVI$(\tfrac{9}{8},\tfrac{-1}{8},\tfrac{1}%
{8},\tfrac{3}{8})$ was obtained via $M_{2,N}(\tau)$. We will
study the general case $n\ge2$ elsewhere.
\end{remark}

\begin{remark}
\label{remarkk}  The proof of Theorem \ref{thm-II-18 copy(1)} will be
given in \S \ref{expression-completely}. Since we
have no explicit expressions of $Z_{r,s}^{(n\pm1)}(\tau)$, the key point is
how to prove that $Z_{r,s}^{(n\pm1)}(\tau)$ appears in the denominator of
$\wp(p_{r,s}^{(n)}(\tau)|\tau)$. Our basic idea is following. It is well-known that solutions of Painlev\'{e} VI equation govern the isomonodromic deformation of some linear ODEs. We already
proved in \cite{Chen-Kuo-Lin} that the monodromy of the associated linear ODE (see GLE (\ref{89-1})
in \S \ref{expression-completely})
corresponding to $p_{r,s}^{(n)}(\tau)$ is generated by (\ref{iiii}) via this $(r,s)$; see
Theorem \ref{thm-II-8} in \S \ref{expression-completely}. If $\tau \to \tau_{0}$ with $\tau_{0}$ being any
pole of $\wp(p_{r,s}^{(n)}(\tau)|\tau)$, we proved in \cite{Chen-Kuo-Lin0}
that GLE (\ref{89-1}) converges to the Lam\'{e} equation $y^{\prime \prime
}=[m(m+1)\wp(z|\tau_{0})+B]y$ with $m\in \{n\pm1\}$ and its monodromy also
generated by (\ref{iiii}) via this $(r,s)$, which finally indicates that either $Z_{r,s}^{(n-1)}%
(\tau_{0})=0$ or $Z_{r,s}^{(n+1)}(\tau_{0})=0$. Hence, the zeros
of $Z_{r,s}^{(n-1)}(\tau)Z_{r,s}^{(n+1)}(\tau)$ coincide with the poles of
$\wp(p_{r,s}^{(n)}(\tau)|\tau)$; see Theorem \ref{q-n=z-n} in \S \ref{expression-completely}.
\end{remark}

The paper is organized as follows. In \S \ref{Dahmen's},
We give a detailed explanation of the deep connection between $Z_{r,s}^{(n)}(\tau)$
and Dahmen-Beukers conjecture. We also review Damhen's and Chou's proofs of this conjecture for $n\leq 4$. In \S 3 we give a complete proof of
Dahmen-Beukers conjecture by applying Theorem \ref{weak}. In \S \ref{expression-completely}, first we apply the Okamoto transformation
to generalize Hitchin's formula via an induction approach. The induction process is technical and will be given in Appendix A.  Then we establish the precise connection
between $Z_{r,s}^{(n)}(\tau)$ and Painlev\'{e} VI equation, and apply it to prove Theorem
\ref{thm-II-18 copy(1)} and Theorem \ref{simple-zero con}. Finally in \S \ref{sec-asymptotics},
we apply the theory of Painlev\'{e} VI equation to study the asymptotics of $Z_{r,s}^{(n)}(\tau)$
and prove Theorem \ref{weak}.

\section{Pre-modular form and Dahmen-Beukers conjecture}

\label{Dahmen's}

\subsection{Connection between $Z_{r,s}^{(n)}(\tau)$ and the conjecture}
In this section, we give a detailed explanation that \emph{counting integral Lam\'{e} equations with monodromy group dihedral $D_{N}$ is equivalent to counting zeros of $Z_{r,s}^{(n)}(\cdot)$ for  $(r,s)\in \mathcal{Q}(N)$}. This is why $Z_{r,s}^{(n)}(\tau)$ can be applied
to study Dahmen-Beukers conjecture.

First we briefly review the theory of $Z_{r,s}^{(n)}(\tau)$ from \cite{CLW,LW2}.
The Lam\'{e}
equation (\ref{iv-21-1}) is associated with a hyperelliptic curve $\bar{Y}_{n}\subset$
Sym$^{n}E_{\tau}:=E_{\tau}^{n}/S_{n}$ (i.e. the $n$-th symmetric product of
$E_{\tau}$) defined by%
\[
Y_{n}=Y_{n}(\tau):=\left \{  \{[a_{i}]\}_{i=1}^{n}\left \vert
\begin{array}
[c]{c}%
\sum_{j\not =i}\left(  \zeta(a_{i}-a_{j})+\zeta(a_{j})-\zeta(a_{i})\right)
=0\\
\lbrack a_{i}]\not =0,[a_{i}]\not =[a_{j}],\text{ }i=1,\cdots,n
\end{array}
\right.  \right \}  ,
\]
and $\bar{Y}_{n}$ is the closure of $Y_{n}$ in Sym$^{n}E_{\tau}$. Historically
$\bar{Y}_{n}$ is also known as the \emph{Lam\'{e} curve}. It was proved in
\cite{CLW,LW2} that the following hold:

(a). $\bar{Y}_{n}(\tau)$ is a hyperelliptic curve with arithmetic genus $n$
and $\bar{Y}_{n}(\tau)=Y_{n}(\tau)\cup \{ \{0,\cdots,0\} \}$.
The affine part $Y_{n}(\tau)\simeq \{(B,C)|C^{2}%
=\ell_{n}(B)\}$, where $\ell_{n}(B)$ is the so-called Lam\'{e} polynomial of degree $2n+1$, and the
branch points of $Y_{n}(\tau)$ are precisely those $\{[a_{i}]\}_{i=1}^{n}\in
Y_{n}(\tau)$ such that $\{[a_{1}],\cdot \cdot \cdot,[a_{n}]\}=\{-[a_{1}%
],\cdot \cdot \cdot,-[a_{n}]\}$.

(b). Let $K(\bar{Y}_{n}(\tau))$ be the field of rational functions on $\bar
{Y}_{n}(\tau)$. Set $\sigma_{n}$ to be the addition map from $\bar{Y}_{n}%
(\tau)$ onto $E_{\tau}$:%
\[
\sigma_{n}(\boldsymbol{a}):=\sum_{i=1}^{n}[a_{i}],\text{ \ }\forall
\boldsymbol{a}=\{[a_{1}],\cdot \cdot \cdot \lbrack a_{n}]\} \in \bar{Y}_{n}%
(\tau).
\]
Then the degree of $\sigma_{n}$ is $\frac{n(n+1)}{2}$. Consequently,
$K(\bar{Y}_{n}(\tau))$ is a finite extension of $K(E_{\tau})$ and%
\[
\left[  K(\bar{Y}_{n}(\tau)):K(E_{\tau})\right]  =\tfrac{n(n+1)}{2}.
\]

(c). Define $\boldsymbol{z}_{n}: \bar{Y}_{n}(\tau)\to\mathbb{C}\cup \{
\infty \}$ by:%
\begin{equation}\label{z--n}
\boldsymbol{z}_{n}(\boldsymbol{a}):=\zeta \left(  \sum \nolimits_{i=1}^{n}%
a_{i}\right)  -\sum \nolimits_{i=1}^{n}\zeta(a_{i}).
\end{equation}
Then $\boldsymbol{z}_{n}\in K(\bar{Y}_{n}(\tau))$ is a primitive generator of the finite field
extension of $K(\bar{Y}_{n}(\tau))$ over $K(E_{\tau})$, and there is a minimal polynomial%
\begin{equation}
W_{n}(X)=W_{n}(X;\sigma_{n},\tau)\in \mathbb{Q}[g_{2}(\tau),g_{3}(\tau
),\wp(\sigma_{n}|\tau),\wp^{\prime}(\sigma_{n}|\tau)][X] \label{k-ll}%
\end{equation}
of the field extension $K(\bar{Y}_{n})$ over $K(E_{\tau})$ which defines the
covering map $\sigma_{n}$ between algebraic curves.
Define
\begin{equation}
Z_{r,s}(\tau):=\zeta(r+s\tau|\tau)-r\eta_{1}(\tau)-s\eta_{2}(\tau).
\label{Hecke}
\end{equation}
Then $W_{n}(X)$ is
characterized by the following properties:

\begin{theorem} \cite[Theorem 3.2]{LW2} \label{thm-5A}

\begin{itemize}
\item[(1)] $W_{n}(X;\sigma_{n},\tau)$ is a monic polynomial of
$X$-degree $\frac{1}{2}n(n+1)$ such that
\[
W_{n}(\boldsymbol{z}_{n}(\boldsymbol{a});\sigma_{n}(\boldsymbol{a}),\tau)=0.
\]
Moreover, $W_{n}(X;\sigma_{n},\tau)$ is weighted homogeneous with
$X$, $\wp(\sigma_{n})$, $\wp^{\prime}(\sigma_{n})$,
$g_{2}$, $g_{3}$ being of weight $1$, $2$,
$3$, $4$, $6$ respectively.

\item[(2)] Fix any $\tau$. For each $a\in E_{\tau
}\backslash E_{\tau}[2]$ being outside the branch loci of $\sigma
_{n}:\bar{Y}_{n}\rightarrow E_{\tau}$, $W_{n}(X;a,\tau)$
has $\frac{1}{2}n(n+1)$ distinct zeros. Furthermore, for any zero
$X_{0}$ of $W_{n}(X;a,\tau)$, there exists
$\boldsymbol{a}\in \bar{Y}_{n}$ such that $a=\sigma
_{n}(\boldsymbol{a})$ and $X_{0}=\boldsymbol{z}_{n}(\boldsymbol{a})$.

\item[(3)] For any given $(r,s)$, the function $Z_{r,s}%
^{(n)}(\tau):=W_{n}(Z_{r,s}(\tau);r+s\tau,\tau)$ is a pre-modular form
of weight $\frac{1}{2}n(n+1)$ with $Z_{r,s}(\tau)$,
$\wp(r+s\tau)$, $\wp^{\prime}(r+s\tau)$, $g_{2}$, $g_{3}$
being of weight $1$, $2$, $3$, $4$,
$6$ respectively.
\end{itemize}
\end{theorem}

Since for any $(m_{1},m_{2})\in \mathbb{Z}^{2}$, $\wp(r+s\tau|\tau)=\wp
(m_{1}-r+(m_{2}-s)\tau|\tau)$, $\wp^{\prime}(r+s\tau|\tau)=-\wp^{\prime}%
(m_{1}-r+(m_{2}-s)\tau|\tau)$ and $Z_{r,s}(\tau)=-Z_{m_{1}-r,m_{2}-s}(\tau)$,
we see from Theorem \ref{thm-5A} that%
\begin{equation}
Z_{r,s}^{(n)}(\tau)=(-1)^{\frac{n(n+1)}{2}}Z_{m_{1}-r,m_{2}-s}^{(n)}(\tau).
\label{iv-20}%
\end{equation}

Let $\boldsymbol{a}\in Y_{n}(\tau)$ and suppose that $\boldsymbol{a}$
satisfies%
\begin{equation}
\sigma_{n}(\boldsymbol{a})=r+s\tau \text{ \ and }\sum_{i=1}^{n}\zeta
(a_{i})=r\eta_{1}(\tau)+s\eta_{2}(\tau) \label{v-10}%
\end{equation}
for some $(r,s)\in \mathbb{R}^{2}\backslash \frac{1}{2}\mathbb{Z}^{2}$. Then%
\[
Z_{r,s}(\tau)=\zeta \left(  \sum \nolimits_{i=1}^{n}a_{i}\right)  -\sum
\nolimits_{i=1}^{n}\zeta(a_{i})=\boldsymbol{z}_{n}(\boldsymbol{a}).
\]
Together with Theorem \ref{thm-5A} we obtain $Z_{r,s}^{(n)}(\tau)=0$. Thus, the zero of
$Z_{r,s}^{(n)}(\tau)$ encodes the identity (\ref{v-10}). We summarize the properties of $Z_{r,s}^{(n)}(\tau)$ in the following result.

\begin{theorem} \cite[Theorem 4.3]{LW2} \label{thm-A}
\begin{itemize}
\item[(1)] For $\bigl(\begin{smallmatrix}a & b\\
c & d\end{smallmatrix}\bigr)
\in SL(2,\mathbb{Z})$, there holds
\[Z_{ar-bs,ds-cr}^{(n)}(\tfrac{a\tau+b}{c\tau+d})=(c\tau+d)^{n(n+1)/2}Z_{r,s}^{(n)}(\tau).\]
In particular,
for $(r,s)\in\mathcal{Q}(N)$, $Z_{r,s}^{(n)}(\tau)$ is a modular form of weight $n(n+1)/2$ with respect to $\Gamma(N)$.
\item[(2)]
For any pair
$(r,s)\in \mathbb{R}^{2}\backslash \frac{1}{2}\mathbb{Z}^{2}$, $Z_{r,s}
^{(n)}(\tau)=0$ for some $\tau \in \mathbb{H}$ if and only if there
is a unique $\boldsymbol{a}=\{[a_{1}],\cdot \cdot \cdot,[a_{n}]\} \in Y_{n}%
(\tau)$ such that (\ref{v-10}) holds.
\end{itemize}
\end{theorem}

For any $\{a_{1},\cdots,a_{n}\} \subset \mathbb{C}\setminus\{0\}$, we consider
the classical \emph{Hermite-Halphen} ansatz
\[
y_{\{a_{i}\}}(z):=e^{z\sum_{i=1}^{n}\zeta(a_{i})}\prod_{i=1}^{n}\frac
{\sigma(z-a_{i})}{\sigma(z)}.
\]
By (\ref{quasi}) and the transformation law of $\sigma(z)$:%
\begin{equation}
\sigma(z+\omega_{k})=-e^{\eta_{k}(z+\frac{\omega_{k}}{2})}\sigma(z),\text{
\ }k=1,2, \label{82-2}%
\end{equation}
we see that $y_{\{a_{i}\}}(z)$ and $y_{\{b_{i}\}}(z)$ are \emph{linearly
dependent} if $\{[a_{1}]$, $\cdots$, $[a_{n}]\}=\{[b_{1}]$, $\cdots$,
$[b_{n}]\}$, so up to a non-zero constant $y_{\{a_{i}\}}(z)$ is uniquely
determined by $\{[a_{1}]$, $\cdots$, $[a_{n}]\}$. We recall the following
classical result.

\begin{proposition} [cf. \cite{CLW,Halphen}] \label{prop-B} $y_{\{a_{i}\}}(z)$
is a solution to the integral Lam\'{e} equation (\ref{iv-21-1}) for some
$B\in \mathbb{C}$ if and only if $\{[a_{1}]$, $\cdots$, $[a_{n}]\} \in
Y_{n}$ and
\begin{equation}
B=B_{\boldsymbol{a}}:=(2n-1)\sum_{i=1}^{n}\wp(a_{i}). \label{B-expression}%
\end{equation}
Moreover, $\{[a_{1}]$, $\cdots$, $[a_{n}]\} \in Y_{n}$ is not a
branch point (equivalently, $\ell_{n}(B)\not =0$) if and only if
$y_{\{a_{i}\}}(z)$ and $y_{\{-a_{i}\}}(z)$ are linearly
independent.\end{proposition}

By (\ref{82-2}) we have%
\[
y_{\{a_{i}\}}(z+1)=e^{-2\pi is}y_{\{a_{i}\}}(z)\;\text{ and }\;y_{\{a_{i}%
\}}(z+\tau)=e^{2\pi ir}y_{\{a_{i}\}}(z),
\]
where $(r,s)\in \mathbb{C}^{2}$ is uniquely determined by (note $\tau \eta_{1}-\eta_{2}=2\pi i$)
\[
r+s\tau=\sum_{i=1}^{n}a_{i}\text{ \ and }r\eta_{1}+s\eta_{2}=\sum_{i=1}%
^{n}\zeta(a_{i}).
\]
Clearly $y_{\{-a_{i}\}}(z)$ is also a solution of (\ref{iv-21-1}) if $y_{\{a_{i}\}}(z)$ is.
Since the
monodromy representation of (\ref{iv-21-1}) is a group homomorphism $\rho:\pi
_{1}(E_{\tau})\rightarrow SL(2,\mathbb{C})$, the monodromy group of
(\ref{iv-21-1}) with respect to $(y_{\{a_{i}\}}(z),y_{\{-a_{i}
\}}(z))$ is generated by%
\begin{equation}
\rho(\ell_{1})=%
\begin{pmatrix}
e^{-2\pi is} & 0\\
0 & e^{2\pi is}%
\end{pmatrix}\;
\text{ and }\;\rho(\ell_{2})=%
\begin{pmatrix}
e^{2\pi ir} & 0\\
0 & e^{-2\pi ir}%
\end{pmatrix}
\label{iiii}%
\end{equation}
provided $\ell_{n}(B)\not =0$,
where $\ell_{n}(B)$ is the Lam\'{e} polynomial recalled  in (a),
and $\ell_{i}$ denote the two fundamental cycles $z\to z+\omega_i$ of
$E_{\tau}$.

Therefore, Theorem \ref{thm-A} says that $Z_{r,s}^{(n)}(\tau)=0$ for
some $(r,s)\in \mathbb{R}^{2}\backslash \frac{1}{2}\mathbb{Z}^{2}$ if and only
if there is a non-branch point $\boldsymbol{a}\in Y_{n}(\tau)$ such that
the monodromy of the Lam\'{e} equation (\ref{iv-21-1}) with $B=B_{\boldsymbol{a}}$
 is generated by (\ref{iiii}) via this $(r,s)$. In particular, if $(r,s)=(\frac{k_{1}}{N},\frac{k_{2}}{N}%
)\in \mathcal{Q}(N)$ is a $N$-torsion point of exact order $N$ and $\tau$ is
any zero of $Z_{r,s}^{(n)}(\cdot)$, then the monodromy group of the Lam\'{e} equation (\ref{iv-21-1}) with some $B$ is \emph{finite} with order $N$, or equivalently the corresponding algebraic form
(\ref{iv-21-2})  has \emph{dihedral $D_{N}$}
as its monodromy group; see \cite{Dahmen0} for a proof.
In conclusion, {\it counting integral Lam\'{e} equations (modulo scalar equivalence) with monodromy group dihedral $D_{N}$ is equivalent to counting zeros of the modular form $M_{n,N}(\cdot)=\prod_{(r,s)\in \mathcal{Q}(N)}Z_{r,s}^{(n)}(\cdot)$ (modulo $SL(2,\mathbb{Z})$ action)}.

\subsection{The previous partial answer \cite{Dahmen0,Dahmen,LW2}}

In 2007
Dahmen \cite{Dahmen} studied this conjecture by considering the \emph{projective}
monodromy group of integral Lam\'{e} equations. Recall \S 1.1 that the number
of integral Lam\'{e} equations modulo scalar equivalence with (resp.
projective) monodromy group dihedral $D_{N}$ is denoted by $L_{n}(N)$ (resp.
$PL_{n}(N)$). Then (cf. \cite[(4.1)]{Dahmen})%
\begin{equation}
PL_{n}(N)=\left \{
\begin{array}
[c]{l}%
L_{n}(N)+L_{n}(2N)\text{ \ if }N\text{ is odd,}\\
L_{n}(2N)\text{ \ if }N\text{ is even.}%
\end{array}
\right.  \label{iv-24}%
\end{equation}
As recalled in Theorem A, Dahmen obtained the explicit
formula
\begin{equation}
PL_{n}(N)=\left \{
\begin{array}
[c]{l}%
0\text{ \  \ if }N\in \{1,2\} \text{,}\\
\frac{n(n+1)}{12}\left(  \Psi(N)-3\phi(N)\right)  +\frac{2}{3}\varepsilon
_{n}(N)\text{ \ otherwise,}%
\end{array}
\right.  \label{iv-25}%
\end{equation}
where $\phi(N)$, $\Psi(N)$ and $\varepsilon_{n}(N)$ are defined in
(\ref{Euler}), (\ref{Euler-2}) and (\ref{iv-23-1}). A immediate consequence of
(\ref{iv-24})-(\ref{iv-25}) is that formula (\ref{Dahmen-conjecture-1}), i.e.
Dahmen-Beukers conjecture, holds for any $n\in \mathbb{N}$ and
$N\in \mathbb{N}_{\geq3}$ satisfying $4|N$.

On the other hand, the following result is the key lemma in \cite{Dahmen0} to prove the conjecture for $n\in \{1,2,3\}$ and $N\in\mathbb{N}_{\geq 3}$.

\begin{lemma} (\cite[Lemma 65]{Dahmen0}) \label{lem-C}
Let $n\geq 1$ and $N\geq 3$.
Suppose for any $N$-torsion point $(r,s)\in\mathcal{Q}(N)$, there exists a modular form $f_{r,s}^{(n)}(\tau)$ of weight $n(n+1)/2$ with respect to $\Gamma(N)$ such that Theorem \ref{thm-A} (1)-(2) holds for $f_{r,s}^{(n)}(\tau)$ and $(r,s)\in\mathcal{Q}(N)$. Consider the modular form
\begin{equation}
M_{n,N}(\tau):=\prod_{(r,s)\in \mathcal{Q}(N)}f_{r,s}^{(n)}(\tau)
\label{modular-SL}
\end{equation} and define
\begin{equation}
U_{n}(N):=\frac{1}{2}\left(  \frac{n(n+1)\Psi(N)}{24}-v_{\infty}\left(
M_{n,N}(\tau)\right)  \right)  +\frac{2}{3}\varepsilon_{n}(N).
\label{U-n-value}%
\end{equation}
Then
$L_{n}(N)\leq U_{n}(N)$. Moreover, $L_{n}(N)=U_{n}(N)$ if and
only if $f_{r,s}^{(n)}(\tau)$ has only simple zeros in $\mathbb{H}$
for all $(r,s)\in \mathcal{Q}({N})$.\end{lemma}

Theorem \ref{thm-A} shows that the new pre-modular form $Z_{r,s}^{(n)}(\tau)$ introduced by Wang and the third author \cite{LW2} satisfies all the assumptions of Lemma \ref{lem-C}. Therefore, the conclusion of Lemma \ref{lem-C} holds for all $n\geq1$ by letting $f_{r,s}^{(n)}(\tau)=Z_{r,s}^{(n)}(\tau)$. Indeed, by applying Lemma \ref{lem-C} to $Z_{r,s}^{(n)}(\tau)$, the conjecture was proved in \cite{Dahmen0,Dahmen,LW2} for $n\in\{1,2,3,4\}$. To explain how to apply Lemma \ref{lem-C}, we reproduce their proofs below.

For $n\in \{1,2,3,4\}$, the explicit expressions of $Z_{r,s}^{(n)}(\tau)$ are
known as shown in \cite{Dahmen0,LW2}. Denote
$Z=Z_{r,s}(\tau)$, $\wp=\wp(r+s\tau|\tau)$ and $\wp^{\prime
}=\wp^{\prime}(r+s\tau|\tau)$ for convenience. Then
\begin{equation}
Z_{r,s}^{(1)}(\tau)=Z_{r,s}(\tau), \quad Z_{r,s}^{(2)}(\tau)=Z^{3}-3\wp Z-\wp^{\prime}, \label{z-n-1}%
\end{equation}%
\begin{align}\label{z-n-4}
Z_{r,s}^{(3)}(\tau)=  &  Z^{6}-15\wp Z^{4}-20\wp^{\prime}Z^{3}+\left(
\tfrac{27}{4}g_{2}-45\wp^{2}\right)  Z^{2}\\
&  -12\wp \wp^{\prime}Z-\tfrac{5}{4}(\wp^{\prime})^{2},\nonumber
\end{align}
{\allowdisplaybreaks
\begin{align*}
Z_{r,s}^{(4)}(\tau)=  &  Z^{10}-45\wp Z^{8}-120\wp^{\prime}Z^{7}+(\tfrac
{399}{4}g_{2}-630\wp^{2})Z^{6}-504\wp \wp^{\prime}Z^{5}\\
&  -\tfrac{15}{4}(280\wp^{3}-49g_{2}\wp-115g_{3})Z^{4}+15(11g_{2}-24\wp
^{2})\wp^{\prime}Z^{3}\\
&  -\tfrac{9}{4}(140\wp^{4}-245g_{2}\wp^{2}+190g_{3}\wp+21g_{2}^{2}%
)Z^{2}\\
&  -(40\wp^{3}-163g_{2}\wp+125g_{3})\wp^{\prime}Z+\tfrac{3}{4}(25g_{2}%
-3\wp^{2})(\wp^{\prime})^{2}.
\end{align*}
}%
Hence the vanishing order $v_{\infty}(M_{n,N}(\tau))$ at infinity for $n\in\{1,2,3,4\}$ can be calculated
explicitly:%
\begin{equation}
v_{\infty}(M_{n,N}(\tau))=a_{n}\phi(N)+b_{n}\phi \left( {N}/{2}\right),
\label{1-18}%
\end{equation}
where $a_{n}$ and $b_{n}$ are given by (\ref{iv-23-5}). Remark that the computation of (\ref{1-18}) is not simple because the expression of $Z_{r,s}^{(n)}(\tau)$ for $n=3,4$ is already very complicated. We refer \cite{Dahmen0,LW2} for the actual computations of (\ref{1-18}).

Now by applying Lemma \ref{lem-C} and (\ref{1-18}), the following result was proved in \cite{Dahmen0,Dahmen,LW2}.

\begin{theorem}\cite{Dahmen0,Dahmen,LW2}
Dahmen-Beukers conjecture holds for $n\in \{1,2,3,4\}$.
\end{theorem}

\begin{proof} For $n\leq 4$,
Lemma \ref{lem-C} and (\ref{1-18}) yield
\begin{align*}
L_{n}(N)  &  \leq U_{n}(N)\\
&  =\frac{1}{2}\left(  \frac{n(n+1)\Psi(N)}{24}-\left(  a_{n}\phi(N)+b_{n}%
\phi (\tfrac{N}{2})  \right)  \right)  +\frac{2}{3}\varepsilon
_{n}(N).
\end{align*}
This, together with (\ref{iv-24})-(\ref{iv-25}), easily implies that
for odd $N$,%
\[
PL_{n}(N)=L_{n}(N)+L_{n}(2N)\leq U_{n}(N)+U_{n}(2N)=PL_{n}(N),
\]
so $L_{n}(N)=U_{n}(N)$ and $L_{n}(2N)=U_{n}(2N)$ if $N$ is odd. If $4|N$ then
$L_{n}(N)=U_{n}(N)$ follows directly from (\ref{iv-24})-(\ref{iv-25}).
This proves Dahmen-Beukers conjecture for $n\in \{1,2,3,4\}$.\end{proof}

A consequence of
this proof and Lemma \ref{lem-C} is that for $n\in \{1,2,3,4\}$, $Z_{r,s}^{(n)}(\tau)$
has at most simple zeros in $\mathbb{H}$ if $(r,s)\in \mathcal{Q}(N)$.

\section{Proof of Dahmen-Beukers conjecture}

In this section, we acknowledge the validity of Theorem \ref{weak} (1) \& (3)
and apply it to give proofs of Theorem \ref{weak} (2), Theorem \ref{thm-vanish-order} and Dahmen-Beukers conjecture.
 The first step of our proof is
\begin{theorem}
[=Theorem \ref{simple-zn}]\label{simple-zero con}For any $n\geq1$ and
$(r,s)\in \mathbb{R}^{2}\backslash \frac{1}{2}\mathbb{Z}^{2}$, $Z_{r,s}%
^{(n)}(\tau)$ has at most simple zeros in $\mathbb{H}$.
\end{theorem}

Theorem \ref{simple-zero con} for $n=1$ was proved in our previous joint work
with Wang \cite{CKLW}. For any $n\geq 1$, Theorem \ref{simple-zero con} is actually a consequence of Theorem \ref{thm-II-18 copy(1)}, which will be proved in \S \ref{expression-completely}.
Thanks to Theorem \ref{simple-zero con}, we can apply Lemma \ref{lem-C} to obtain

\begin{corollary}
\label{coro-con}Let $n\in \mathbb{N}$ \textit{and} $N\in \mathbb{N}_{\geq3}$.
Then $L_{n}(N)=U_{n}(N)$, where $U_n(N)$ is defined in (\ref{U-n-value}).
\end{corollary}

Corollary \ref{coro-con} shows that Dahmen-Beukers conjecture (i.e. Theorem \ref{thm-Dahmen-Conjecture}) is equivalent to Theorem \ref{thm-vanish-order} as we discussed in \S 1. For the rest of this section, we will give a complete proof of
Theorem \ref{thm-vanish-order} by assuming Theorem \ref{weak} (1) \& (3).

\begin{remark}
\label{remark-Dahmen}As mentioned before, in both Dahmen's and Chou's proof
 for $n\in \{1,2,3,4\}$, since the explicit expression of
$Z_{r,s}^{(n)}(\tau)$ is known, they could prove the identity (\ref{1-18})
first and then obtained the simple zero property of $Z_{r,s}^{(n)}(\tau)$ for
a \textbf{rational pair} $(r,s)$ as a byproduct. However, this approach can
not work for general $n\geq5$ because it is impossible to write down the
explicit expression of $Z_{r,s}^{(n)}(\tau)$ for general $n$. In this
paper, we exploit an \textbf{opposite idea}, namely we prove the simple zero
property of $Z_{r,s}^{(n)}(\tau)$ (i.e. Theorem \ref{simple-zero con}) first,
and then apply it to prove the identity (\ref{1-18})!
\end{remark}

First we have the following result, which proves Theorem \ref{weak}-(2).

\begin{theorem}
\label{modular-zero}Fix any $n\in \mathbb{N}$ and recall $a_n, b_n$ in (\ref{iv-23-5}). Then there exist $\tilde{a}%
_{n},\tilde{b}_{n}\in \mathbb{N\cup \{}0\mathbb{\}}$ independent of $N$ and satisfying%
\begin{equation}
2\tilde{a}_{n}+\tilde{b}_{n}=2a_{n}+b_{n}, \label{iv-39-1}
\end{equation}
such that the following hold.
\begin{itemize}
\item[(1)]
\begin{equation}
Z_{r,0}^{(n)}(\tau)=\alpha_{0}^{(n)}(r)q^{\tilde{a}_{n}}+\sum_{k=1}^{\infty
}\alpha_{k}^{(n)}(r)q^{\tilde{a}_{n}+k},\quad \alpha_{0}^{(n)}(r)\not\equiv 0, \label{z-nn}%
\end{equation}%
\begin{equation}\label{z-0nn}
Z_{r,\frac{1}{2}}^{(n)}(\tau)=\beta_{0}^{(n)}(r)q^{\frac{\tilde{b}_{n}}{2}%
}+\sum_{k=1}^{\infty}\beta_{k}^{(n)}(r)q^{\frac{\tilde{b}_{n}+k}{2}},\quad \beta_{0}^{(n)}(r)\not\equiv 0,
\end{equation}
where both $\alpha_{0}^{(n)}(r)$ and $\beta_{0}^{(n)}(r)$ are holomorphic in $r\in\mathbb{C}\setminus\mathbb{Z}$ and have \emph{no}
rational zeros in $(0,1/2)\cup(1/2,1)$.
\item[(2)]
For any $N\in \mathbb{N}_{\geq3}$,
\begin{equation}
v_{\infty}(  M_{n,N}(\tau))=\tilde{a}_{n}\phi(N)+\tilde{b}_{n}%
\phi (N/2). \label{iv-39}%
\end{equation}
\end{itemize}
\end{theorem}

\begin{proof} {\bf Step 1.} We prove (\ref{z-nn})-(\ref{z-0nn}).

Recall the following $q=e^{2\pi i\tau}$
expansions (see e.g. \cite[Appendix A]{LW2}):%
\begin{equation}
g_{2}(\tau)=\frac{4}{3}\pi^{4}+320\pi^{4}\sum_{n=1}^{\infty}\sigma_{3}%
(n)q^{n}, \label{iv-34}%
\end{equation}%
\[
g_{3}(\tau)=\frac{8}{27}\pi^{6}-\frac{448}{3}\pi^{6}\sum_{n=1}^{\infty}%
\sigma_{5}(n)q^{n},
\]
where $\sigma_{k}(n):=\sum_{d|n}d^{k}$, and
\begin{equation}
\wp^{\prime}(r|\tau)=-2\pi^{3}\cot(\pi r)-2\pi^{3}\cot^{3}(\pi r)+16\pi
^{3}\sum_{n,m=1}^{\infty}n^{2}\sin(2\pi nr)q^{nm}, \label{iv-35-2}%
\end{equation}%
\[
\wp(r|\tau)=\pi^{2}\cot^{2}(\pi r)+\frac{2}{3}\pi^{2}+8\pi^{2}\sum
_{n,m=1}^{\infty}n(1-\cos(2\pi nr))q^{nm},
\]%
\begin{equation}
Z_{r,0}(\tau)=\pi \cot(\pi r)+4\pi \sum_{n,m=1}^{\infty}\sin(2n\pi r)q^{nm}.
\label{iv-35}%
\end{equation}
Note from Theorem \ref{thm-5A} that
\[
Z_{r,s}^{(n)}(\tau)\in \mathbb{Q}\left[  g_{2}(\tau),g_{3}(\tau),\wp
(r+s\tau|\tau),\wp^{\prime}(r+s\tau|\tau)\right]  [Z_{r,s}(\tau)]\text{.}%
\]
Applying (\ref{iv-34})-(\ref{iv-35}), there exist $\tilde{a}_{n}\in
\mathbb{N}\cup \{0\}$ and holomorphic functions $\alpha_{k}^{(n)}(r)$ in
$r\in \mathbb{C}\setminus\mathbb{Z}$ which have poles at most at $r\in \mathbb{Z}$ and
$\alpha_{0}^{(n)}(r)\not \equiv 0$, such that%
\begin{equation}
Z_{r,0}^{(n)}(\tau)=\alpha_{0}^{(n)}(r)q^{\tilde{a}_{n}}+\sum_{k=1}^{\infty
}\alpha_{k}^{(n)}(r)q^{\tilde{a}_{n}+k}. \label{iv-32}%
\end{equation}
Similarly, by applying (see \cite[Appendix A]{LW2})
\[
\wp^{\prime}\left(  \left.  r+\tfrac{\tau}{2}\right \vert \tau \right)
=16\pi^{3}\sum_{n,m=1}^{\infty}n^{2}\sin(2\pi nr)q^{n(m-\frac{1}{2})},
\]%
\[
\wp \left(  \left.  r+\tfrac{\tau}{2}\right \vert \tau \right)  =-\frac{\pi^{2}%
}{3}+8\pi^{2}\sum_{n,m=1}^{\infty}q^{nm}-8\pi^{2}\sum_{n,m=1}^{\infty}%
n\cos(2\pi nr)q^{n(m-\frac{1}{2})},
\]%
\begin{equation}
Z_{r,\frac{1}{2}}(\tau)=4\pi \sum_{n,m=1}^{\infty}\sin(2\pi nr)q^{n(m-\frac
{1}{2})}, \label{iv-35-1}%
\end{equation}
there exist $\tilde{b}_{n}\in \mathbb{N}\cup \{0\}$ and holomorphic
functions $\beta_{k}^{(n)}(r)$ in $r\in \mathbb{C}$ satisfying
$\beta_{0}^{(n)}(r)\not \equiv 0$, such that%
\begin{equation}
Z_{r,\frac{1}{2}}^{(n)}(\tau)=\beta_{0}^{(n)}(r)q^{\frac{\tilde{b}_{n}}{2}%
}+\sum_{k=1}^{\infty}\beta_{k}^{(n)}(r)q^{\frac{\tilde{b}_{n}+k}{2}}.
\label{iv-33}%
\end{equation}

{\bf Step 2.} We prove $2\tilde{a}_n+\tilde{b}_n=2a_n+b_n$.

 By Theorem \ref{weak}-(1) which will be proved by applying Painlev\'{e} VI equation in \S \ref{sec-asymptotics}, we have $v_{\infty}(Z_{r,s}^{(n)}(\tau))=0$ for any $s\in
(0,1/2)\cup(1/2,1)$, so we see from (\ref{modular-SL}) that
\begin{equation}
v_{\infty}\left(  M_{n,N}(\tau)\right)  =\sum_{(r,0)\in \mathcal{Q}%
(N)}v_{\infty}(Z_{r,0}^{(n)}(\tau))+\sum_{(r,\frac{1}{2})\in \mathcal{Q}%
(N)}v_{\infty}(Z_{r,\frac12}^{(n)}(\tau)). \label{iv-30}%
\end{equation}
Furthermore, Theorem A says that (\ref{Dahmen-conjecture-1}) holds for any
$N$ satisfying $4|N$. This together with
(\ref{U-n-value}) and Corollary \ref{coro-con} implies
\begin{align}
v_{\infty}\left(  M_{n,N}(\tau)\right)  &=a_{n}\phi(N)+b_{n}\phi \left(
\tfrac{N}{2}\right)\nonumber\\
&=(2a_n+b_n)\phi(\tfrac{N}{2}),  \text{ for any }N\text{ satisfying }4|N, \label{iv-39-3}%
\end{align}
where we used $\phi(N)=2\phi(N/2)$ for $4|N$ to obtain the second equality.

On the other hand,
since zeros of a meromorphic function are isolated, it follows that
$\alpha_{0}^{(n)}(r)$ has only finitely many zeros in $[0,1]$. Denote all its
rational zeros in $(0,1)$ by%
\[
O_{1}=\left \{  \left.  \frac{l_{1,k}}{m_{1,k}}\in(0,1)\right \vert \gcd
(l_{1,k},m_{1,k})=1,k\leq j_{1}\right \}  .
\]
Again denote all rational zeros in $(0,1)$ of $\beta_{0}^{(n)}(r)$ by%
\[
O_{2}=\left \{  \left.  \frac{l_{2,k}}{m_{2,k}}\in(0,1)\right \vert \gcd
(l_{2,k},m_{2,k})=1,k\leq j_{2}\right \}  .
\]
Now for any $N$ such that $4|N$ and $N>\max \{m_{i,k}|i=1,2,k\leq j_{i}\}$, we
have $r\not \in O_{1}$, i.e. $\alpha_{0}^{(n)}(r)\not =0$ and so $v_{\infty
}(Z_{r,0}^{(n)}(\tau))=\tilde{a}_{n}$ for any $(r,0)=(\frac{k_{1}}{N}%
,0)\in \mathcal{Q}(N)$; also $r\not \in O_{2}$, i.e. $\beta_{0}^{(n)}%
(r)\not =0$ and so $v_{\infty}(Z_{r,\frac{1}{2}}^{(n)}(\tau))=\tilde{b}_{n}/2$
for any $(r,\frac{1}{2})=(\frac{k_{1}}{N},\frac{1}{2})\in \mathcal{Q}(N)$.
Therefore,
\begin{align}
v_{\infty}(M_{n,N}(\tau))  &=\sum_{(r,0)\in \mathcal{Q}(N)}%
\tilde{a}_{n}+\sum_{(r,\frac{1}{2})\in \mathcal{Q}(N)}\tfrac{\tilde{b}_{n}}%
{2}=\tilde{a}_{n}\phi(N)+\tilde{b}_{n}\phi \left(  \tfrac{N}{2}\right)\nonumber\\
&=(2\tilde{a}_n+\tilde{b}_n)\phi(\tfrac{N}{2}),
\label{iv-31}%
\end{align}
where we used $\phi(N)=2\phi(N/2)$ for $4|N$ again. Clearly (\ref{iv-39-3})-(\ref{iv-31}) imply
\begin{equation}
2\tilde{a}_{n}+\tilde{b}_{n}=2a_{n}+b_{n}. \label{iv-36}%
\end{equation}

{\bf Step 3}. We prove (\ref{iv-39}), and both $\alpha_{0}^{(n)}(r)$ and $\beta_{0}^{(n)}(r)$ have no
rational zeros in $(0,1/2)\cup(1/2,1)$.

Note (\ref{iv-39-3}) and (\ref{iv-36}) already prove (\ref{iv-39}) for $4|N$. So it suffices to consider $4\nmid N$.
Fix any odd $N\geq3$. Then $\phi(2N)=\phi(N)$ and $\phi(N/2)=0$. It
follows from (\ref{iv-32}) and (\ref{iv-33}) that $v_{\infty}(Z_{r,0}%
^{(n)}(\tau))\geq \tilde{a}_{n}$ for any $(r,0)\in \mathcal{Q}(N)$ and
$v_{\infty}(Z_{r,\frac{1}{2}}^{(n)}(\tau))\geq \tilde{b}_{n}/2$ for any
$(r,\frac{1}{2})\in \mathcal{Q}(N)$, i.e.%
\begin{equation}
v_{\infty}\left(  M_{n,N}(\tau)\right)  \geq \tilde{a}_{n}\phi(N)+\tilde{b}%
_{n}\phi \left(  \tfrac{N}{2}\right)  . \label{iv-37}%
\end{equation}
Similarly,%
\begin{equation}
v_{\infty}\left(  M_{n,2N}(\tau)\right)  \geq \tilde{a}_{n}\phi(2N)+\tilde
{b}_{n}\phi(N). \label{iv-38}%
\end{equation}
On the other hand, by denoting the RHS of (\ref{Dahmen-conjecture-1})
by $\overline{L}_{n}(N)$, i.e.%
\begin{equation}
\overline{L}_{n}(N):=\frac{1}{2}\left(  \frac{n(n+1)\Psi(N)}{24}-\left(
a_{n}\phi(N)+b_{n}\phi(\tfrac{N}{2})  \right)  \right)  +\frac
{2}{3}\varepsilon_{n}(N), \label{iv-26}%
\end{equation}
it is easy to derive from (\ref{iv-24})-(\ref{iv-25}) and (\ref{iv-26}) that%
\begin{equation*}
L_{n}(N)+L_{n}(2N)=\overline{L}_{n}(N)+\overline{L}_{n}(2N)\;\text{ if }\;N\text{
is odd.}
\end{equation*}
This, together with Corollary \ref{coro-con} which says $L_n(N)=U_{n}(N)$, implies
\begin{equation*}
U_{n}(N)+U_{n}(2N)=\overline{L}_{n}(N)+\overline{L}_{n}(2N)\;\text{ if }\;N\text{
is odd,}
\end{equation*}
and so
\begin{align*}
&  v_{\infty}\left(  M_{n,N}(\tau)\right)  +v_{\infty}\left(  M_{n,2N}%
(\tau)\right) \\
=  &  a_{n}\phi(N)+b_{n}\phi \left(  \tfrac{N}{2}\right)  +a_{n}\phi
(2N)+b_{n}\phi(N)\\
=  &  (2a_{n}+b_{n})\phi(N), \;\text{ if }\;N\text{
is odd.}
\end{align*}
Therefore, we see from (\ref{iv-36}) that both (\ref{iv-37}) and (\ref{iv-38})
must be identities, namely (\ref{iv-39}) holds for all $N\geq3$.
Furthermore,
\[v_{\infty}(Z_{r,0}^{(n)}(\tau))=\tilde{a}_n,\quad v_{\infty}(Z_{r,\frac12}^{(n)}(\tau))=\tilde{b}_n/2\]
for any $r\in (0,1/2)\cup(1/2,1)\cap \mathbb{Q}$, namely both $\alpha_{0}^{(n)}(r)$ and $\beta_{0}^{(n)}(r)$ have no
rational zeros in $(0,1/2)\cup(1/2,1)$. This completes the proof.
\end{proof}

\begin{corollary}
\label{a-n=b-n}Assume Theorem \ref{weak}-(3), then (\ref{z-nn}) implies $\tilde{a}_n=a_n$ and so $\tilde{b}_n=b_n$ by Theorem \ref{modular-zero}.
\end{corollary}

Now we are in a position to prove Dahmen-Beukers conjecture.

\begin{proof}
[Proof of Theorems \ref{thm-vanish-order} and \ref{thm-Dahmen-Conjecture}] Obviously Theorem \ref{thm-vanish-order}
follows directly from Theorem \ref{modular-zero}
and Corollary \ref{a-n=b-n}, and Theorem \ref{thm-Dahmen-Conjecture}, i.e. Dahmen-Beukers conjecture, follows from Theorem \ref{thm-vanish-order} and Corollary \ref{coro-con}.
\end{proof}

We conclude this section by proving Corollary \ref{degree}.

\begin{proof}[Proof of Corollary \ref{degree}]
 Consider the case either
$N>3$ or $n\not \equiv 1\operatorname{mod}3$, i.e. $\varepsilon_{n}(N)=0$.
Recall that $M_{n,N}(\tau)$ is a modular form of weight $\frac{n(n+1)}{2}%
\Psi(N)$.
Then $k(n,N):=\frac{n(n+1)}{24}\Psi(N)\in \mathbb{N}$ and $\frac{M_{n,N}(\tau
)}{\Delta(\tau)^{k(n,N)}}$ is invariant under $SL(2,\mathbb{Z})$. Observe from
(\ref{iv-20}) that any zero of $\frac{M_{n,N}(\tau)}{\Delta(\tau)^{k(n,N)}}$
must be doubled. Since $\frac{M_{n,N}(\tau)}{\Delta(\tau)^{k(n,N)}}$ has no
poles in $\mathbb{H}$, we conclude that
\begin{equation}
\frac{M_{n,N}(\tau)}{\Delta(\tau)^{k(n,N)}}=C_{n,N}\ell_{n,N}(j(\tau
))^{2}\label{j-polynomial}%
\end{equation}
for some monic polynomial $\ell_{n,N}$ of $j$ and non-zero constant $C_{n,N}$.
Recall the $q$-expansions (cf. \cite[p.193]{Husemoller}):%
\[
\Delta(\tau)=(2\pi)^{12}q\prod_{n=1}^{+\infty}(1-q^{n})^{24},
\]%
\[
j(\tau)=\tfrac{1}{q}+744+196884q+21493760q^{2}+\cdots.
\]
By comparing the leading term of $q$-expansions in (\ref{j-polynomial}), we
easily obtain%
\[
\deg \ell_{n,N}=\tfrac{1}{2}\left(  k(n,N)-v_{\infty}\left(  M_{n,N}%
(\tau)\right)  \right)  =L_{n}(N).
\]
The proof is complete.\end{proof}

\section{Generalization of Hitchin's formula}

\label{expression-completely}

From now on, we apply Painlev\'{e} VI equation to prove
Theorem \ref{weak} (1) \& (3) and Theorem \ref{simple-zero con}. In this section,
we generalize Hitchin's formula
to the general case $n\in\mathbb{N}$ and prove Theorem \ref{thm-II-18 copy(1)} and
Theorem \ref{simple-zero con}.

\subsection{The Hamiltonian system and asymptotics at poles}
It is well known (cf. \cite{GP}) that PVI (\ref{46-0}) is equivalent to the following
Hamiltonian system%
\begin{equation}
\frac{d\lambda(t)}{dt}=\frac{\partial K}{\partial \mu},\text{ \ }\frac{d\mu
(t)}{dt}=-\frac{\partial K}{\partial \lambda},\label{aa}%
\end{equation}
where $K=K(\lambda,\mu,t)$ is given by%
\begin{equation}
K=\frac{1}{t(t-1)}\left \{
\begin{array}
[c]{l}%
\lambda(\lambda-1)(\lambda-t)\mu^{2}+\theta_{0}(\theta_{0}+\theta_{4}%
)(\lambda-t)\\
-\left[
\begin{array}
[c]{l}%
\theta_{1}(\lambda-1)(\lambda-t)+\theta_{2}\lambda(\lambda-t)\\
+(\theta_{3}-1)\lambda(\lambda-1)
\end{array}
\right]  \mu
\end{array}
\right \}  ,\label{98}%
\end{equation}
and the relation of parameters is given by%
\begin{equation}
\left(  \alpha,\beta,\gamma,\delta \right)  =\left(  \tfrac{1}{2}\theta_{4}%
^{2},\,-\tfrac{1}{2}\theta_{1}^{2},\, \tfrac{1}{2}\theta_{2}^{2},\, \tfrac{1}%
{2}\left(  1-\theta_{3}^{2}\right)  \right)  ,\label{46-2}%
\end{equation}%
\begin{equation}
2\theta_{0}+\theta_{1}+\theta_{2}+\theta_{3}+\theta_{4}=1.\label{46-3}%
\end{equation}
Here we recall the following classical
result concerning the asymptotics of $\lambda(t)$ at poles.

\begin{proposition} \cite[Proposition 1.4.1]{GP} \label{thm-2A} Assume
$\theta_{4}\not =0$. Then for any $t_{0}%
\in \mathbb{C}\backslash \{0,1\}$, there exist two $1$%
-parameter families of solutions $\lambda(t)$ of PVI
(\ref{46-0}) such that
\begin{equation}
\lambda(t)=\frac{\xi}{t-t_{0}}+h+O(t-t_{0})\text{ as }t\rightarrow t_{0},
\label{II-132}%
\end{equation}
where $h\in \mathbb{C}$ can be taken arbitrary and
\begin{equation}
\xi=\xi(\theta_{4},t_{0})\in \left \{  \pm \tfrac{t_{0}(t_{0}-1)}{\theta_{4}%
}\right \}  . \label{II-133}%
\end{equation}
Furthermore, these two $1$-parameter families of solutions
give all solutions of PVI (\ref{46-0}) which has a pole at $t_{0}
$.\end{proposition}

\begin{definition}
When $\theta_{4}>0$, we call that a pole
$t_{0}\in \mathbb{C}\backslash \{0,1\}$ of $\lambda(t)$ is a positive
pole (resp. a negative pole) if the residue is $\frac{t_{0}(t_{0}%
-1)}{\theta_{4}}$ (resp. $-\frac{t_{0}(t_{0}-1)}{\theta_{4}}$).
\end{definition}

Here we have the following simple but interesting observation, which
gives a criterion to determine the type of poles. This result is
important when we want to connect $Z_{r,s}^{(n)}(\tau)$ with PVI.

\begin{proposition}
\label{Prop-II-1}Assume $\theta_{4}>0$. Let
$t_{0}\in \mathbb{C}\backslash \{0,1\}$ and $\lambda(t)$ be a solution
of PVI (\ref{46-0}) which has a pole at $t_{0}$. Then $\lambda \mu$ is
holomorphic at $t_{0}$ and more precisely,
\begin{equation}
\lim \limits_{t\rightarrow t_{0}}\lambda \mu=\left \{
\begin{array}
[c]{l}%
-\theta_{0}\text{\  \  \ if }t_{0}\text{ is a negative pole,}\\
-(\theta_{0}+\theta_{4})\text{\  \  \ if }t_{0}\text{ is a positive pole,}%
\end{array}
\right.  \label{II-136}%
\end{equation}
where $\mu(t)$ is defined by the first equation of the Hamiltonian system
(\ref{aa}).
\end{proposition}

\begin{proof}
By Proposition \ref{thm-2A} we have%
\begin{equation}
\lambda(t)=\frac{\xi}{t-t_{0}}\left(  1+O(t-t_{0})\right)  \text{ as
}t\rightarrow t_{0}, \label{II-137}%
\end{equation}%
\begin{equation}
\lambda^{\prime}(t)=-\frac{\xi}{(t-t_{0})^{2}}\left(  1+O(t-t_{0})\right)
\text{ as }t\rightarrow t_{0}. \label{II-138}%
\end{equation}
Since the first equation of the Hamiltonian system (\ref{aa})
reads{\allowdisplaybreaks%
\begin{align}
\lambda^{\prime}(t)=\frac{\partial K}{\partial \mu}=  &  \frac{1}%
{t(t-1)}\big[2\lambda(\lambda-1)(\lambda-t)\mu-\theta_{1}(\lambda
-1)(\lambda-t)\label{II-139}\\
&  -\theta_{2}\lambda(\lambda-t)-(\theta_{3}-1)\lambda(\lambda
-1)\big],\nonumber
\end{align}
}we see that $\mu(t)$ is meromorphic in a neighborhood of $t_{0}$ and so does
$\lambda \mu$. Denote%
\begin{equation}
\lim \limits_{t\rightarrow t_{0}}\lambda(t)\mu(t)=L\in \mathbb{C}\cup \{
\infty \}. \label{II-140}%
\end{equation}
If $L=\infty$, then by substituting (\ref{II-137})-(\ref{II-138}) into
(\ref{II-139}), we easily obtain a contradiction. Thus $L$ is finite, namely
$\lambda \mu$ is holomorphic at $t_{0}$. Again by substituting (\ref{II-137}%
)-(\ref{II-138}) and (\ref{II-140}) into (\ref{II-139}) and comparing the
coefficients of the term $(t-t_{0})^{-2}$, we easily obtain%
\[
2L=\theta_{1}+\theta_{2}+\theta_{3}-1-\frac{t_{0}(t_{0}-1)}{\xi}.
\]
This, together with (\ref{II-133}) and (\ref{46-3}), proves (\ref{II-136}).
\end{proof}

\subsection{The Okamoto transformations}
From now on we consider the elliptic form (\ref{124-0}) with parameters%
\[
(\alpha_{0},\alpha_1, \alpha_2, \alpha_3)
=(\tfrac{1}{2}(n+\tfrac{1}{2})^{2},\tfrac{1}{8},\tfrac{1}{8},\tfrac{1}{8}),
\quad n\in\mathbb{Z}_{\geq 0},
\]
or equivalently PVI (\ref{46-0}) with parameters%
\[
(\alpha,\beta,\gamma,\delta)=  (  \tfrac{1}{2}(n+\tfrac{1}{2}%
)^{2},\tfrac{-1}{8},\tfrac{1}{8},\tfrac{3}{8}).
\]
It is known that solutions of PVI$(\frac{1}{2}(n+\frac12)^2,\frac{-1}{8},\frac{1}{8},\frac
{3}{8})$ could be
obtained from solutions of PVI$(\frac{1}{8},\frac{-1}{8},\frac{1}{8},\frac
{3}{8})$ (i.e. $n=0$) via the Okamoto transformation
\cite{Okamoto1}. First we recall the explicit form of the Okamoto
transformations. By (\ref{46-2})-(\ref{46-3}), it is convenient to think
of the parameter space of PVI (\ref{46-0}) (equivalently the Hamiltonian
system (\ref{aa})-(\ref{98})) as an affine space%
\[
\mathcal{K}=\left \{  \theta=(\theta_{0},\theta_{1},\theta_{2},\theta
_{3},\theta_{4})\in \mathbb{C}^{5}\text{ }:\text{ }2\theta_{0}+\theta
_{1}+\theta_{2}+\theta_{3}+\theta_{4}=1\right \}  .
\]
An Okamoto transformation $\kappa$ maps solutions $\lambda(t)$ of PVI
(\ref{46-0}) (resp. solutions $(\lambda(t),\mu(t))$ of the Hamiltonian
system (\ref{aa})) with parameter $\theta$ to solutions $\kappa(\lambda)(t)$
of PVI (\ref{46-0}) (resp. solutions $(\kappa(\lambda)(t),\kappa(\mu)(t))$
of (\ref{aa})) with new parameter $\kappa(\theta)\in \mathcal{K}$. The list of
the Okamoto transformations $\kappa_{j}(0\leq j\leq4)$ is given in the Table 1
(cf. \cite{Tsuda-Okamoto-Sakai}).\begin{table}[tbh]
\caption{Okamoto transformations}%
\centering
\par%
\begin{tabular}
[c]{|l|c|c|c|c|c|l|c|c|}\hline
& $\theta_{0}$ & $\theta_{1}$ & $\theta_{2}$ & $\theta_{3}$ & $\theta_{4}$ &
$t$ & $\lambda$ & $\mu$\\ \hline
$\kappa_{0}$ & $-\theta_{0}$ & $\theta_{1}+\theta_{0}$ & $\theta_{2}%
+\theta_{0}$ & $\theta_{3}+\theta_{0}$ & $\theta_{4}+\theta_{0}$ & $t$ &
$\lambda+\frac{\theta_{0}}{\mu}$ & $\mu$\\ \hline
$\kappa_{1}$ & $\theta_{0}+\theta_{1}$ & $-\theta_{1}$ & $\theta_{2}$ &
$\theta_{3}$ & $\theta_{4}$ & $t$ & $\lambda$ & $\mu-\frac{\theta_{1}}%
{\lambda}$\\ \hline
$\kappa_{2}$ & $\theta_{0}+\theta_{2}$ & $\theta_{1}$ & $-\theta_{2}$ &
$\theta_{3}$ & $\theta_{4}$ & $t$ & $\lambda$ & $\mu-\frac{\theta_{2}}%
{\lambda-1}$\\ \hline
$\kappa_{3}$ & $\theta_{0}+\theta_{3}$ & $\theta_{1}$ & $\theta_{2}$ &
$-\theta_{3}$ & $\theta_{4}$ & $t$ & $\lambda$ & $\mu-\frac{\theta_{3}%
}{\lambda-t}$\\ \hline
$\kappa_{4}$ & $\theta_{0}+\theta_{4}$ & $\theta_{1}$ & $\theta_{2}$ &
$\theta_{3}$ & $-\theta_{4}$ & $t$ & $\lambda$ & $\mu$\\ \hline
\end{tabular}
\end{table}

These five transformations $\kappa_{j}$ $(0\leq j\leq4)$, which satisfy
$\kappa_{j}\circ \kappa_{j}=Id$, generate the affine Weyl group of type
$D_{4}^{(1)}$:%
\[
W(D_{4}^{(1)})=\left \langle \kappa_{0},\kappa_{1},\kappa_{2},\kappa_{3}%
,\kappa_{4}\right \rangle .
\]
Define%
\begin{equation}
\kappa_{5}:=\kappa_{0}(\kappa_{3}\kappa_{2}\kappa_{1}\kappa_{0})^{2}\kappa
_{4}, \label{II-120}%
\end{equation}%
A straightforward computation shows%
\begin{equation}
\kappa_{5}(\theta)=(\theta_{0}-1,\theta_{1},\theta_{2},\theta_{3},\theta
_{4}+2). \label{II-121}%
\end{equation}
Note that for PVI$(\frac{1}{2}(n+\frac12)^2,\frac{-1}{8},\frac{1}{8},\frac
{3}{8})$, the corresponding $\theta^n:=\theta\in \mathcal{K}$ is given by
\begin{equation}
\theta^{n}:=\left(  -\tfrac{n+1}{2},\text{ }\tfrac{1}{2},\text{ }\tfrac{1}%
{2},\text{ }\tfrac{1}{2},\text{ }n+\tfrac{1}{2}\right),\quad\text{i.e. }\;
\theta_4=n+\tfrac12>0. \label{I1I-175}%
\end{equation}
Letting $\kappa^{0,1}:=\kappa_{0}\circ \kappa_{3}%
\circ \kappa_{2}\circ \kappa_{1}$, we have $\kappa^{0,1}(\theta^{0})=\theta^{1}$. Define
\[\kappa^{0,n}:=\begin{cases}
(\kappa_{5})^m \circ \kappa^{0,1},\quad \text{if $n=2m+1$ odd},\\
(\kappa_{5})^m \quad \text{if $n=2m$ even},
\end{cases}\]
then it follows from $\kappa_5(\theta^n)=\theta^{n+2}$ that $\kappa^{0,n}(\theta^0)=\theta^n$ for any $n$. Thus there exist two \emph{rational functions} $R_n
(\cdot,\cdot,\cdot)$ and $\tilde{R}_n(\cdot,\cdot,\cdot)$ of three
independent variables with coefficients in $\mathbb{Q}$ such that for any
solution $(\lambda^{(0)}(t),\mu^{(0)}(t))$ of the Hamiltonian
system (\ref{aa}) with parameter $\theta^{0}$, $(\lambda^{(n)
}(t),\mu^{(n)}(t))$ given by%
\begin{equation}
\lambda^{(n)}(t):=\kappa(\lambda^{(0)})(t)=R_{n
}(\lambda^{(0)}(t),\mu^{(0)}(t),t), \label{II-128}%
\end{equation}%
\begin{equation}
\mu^{(n)}(t):=\kappa(\mu^{(0)})(t)=\tilde{R}_{n
}(\lambda^{(n)}(t),\mu^{(0)}(t),t), \label{II-128-0}%
\end{equation}
is a solution of the Hamiltonian system (\ref{aa}) with parameter
$\theta^{n}$, or equivalently, $\lambda^{(n)}(t)$ is a
solution of PVI$(\frac{1}{2}(n+\frac12)^2,\frac{-1}{8},\frac{1}{8},\frac
{3}{8})$.

Now we recall Hitchin's formula.
\medskip

\noindent \textbf{Theorem B. (Hitchin \cite{Hit1})} \emph{For any
pair $(r,s)\in \mathbb{C}^{2}\backslash \frac{1}{2}\mathbb{Z}^{2}$,
define $p^{(0)}_{r,s}(\tau)$ by
\begin{equation}
\wp(p^{(0)}_{r,s}(\tau)|\tau):=\wp(r+s\tau|\tau)+\frac{\wp^{\prime}(r+s\tau|\tau)}%
{2(\zeta(r+s\tau|\tau)-r\eta_{1}(\tau)-s\eta_{2}(\tau))}. \label{II-1}%
\end{equation}
Then $p^{(0)}_{r,s}(\tau)$ is a solution of EPVI$(\frac{1}{8},\frac{1}{8},\frac{1}{8},\frac{1}{8})$; or
equivalently, \[\lambda^{(0)}_{r,s}(t):=\frac{\wp(p^{(0)}_{r,s}(\tau)|\tau)-e_1(\tau)}
{e_2(\tau)-e_1(\tau)}\]
is a solution of PVI$(\frac{1}{8},\frac{-1}{8},\frac{1}{8},\frac{3}{8})$.}

\medskip

\noindent \textbf{Notation}: \emph{Let $\lambda^{(n)}(t)$ be a solution of PVI$(\frac{1}{2}(n+\frac12)^2,\frac{-1}{8},\frac{1}{8},\frac
{3}{8})$ and $p^{(n)}(\tau)$ be the corresponding solution of EPVI$(\frac{1}{2}(n+\frac12)^2,\frac{1}{8},\frac{1}{8},\frac
{1}{8})$. We denote them by $\lambda^{(n)}_{r,s}(t)$  and $p^{(n)}_{r,s}(\tau)$ respectively if them come
from the solution $\lambda_{r,s}^{(0)}(t)$
of PVI$(\frac{1}{8},$ $\frac{-1}{8},\frac{1}{8},\frac{3}{8})$ and
the corresponding $p_{r,s}^{(0)}(\tau)$ of
EPVI$(\frac{1}{8},\frac{1}{8},\frac{1}{8},\frac{1}{8})$ stated in Theorem B
via (\ref{II-128}), i.e.
\begin{equation}
\frac{\wp(p_{r,s}^{(n)}(\tau)|\tau)-e_{1}(\tau)}{e_{2}(\tau)-e_{1}%
(\tau)}=\lambda^{(n)}_{r,s}(t)=R_n\left(\lambda_{r,s}^{(0)}(t),\mu_{r,s}^{(0)}(t),t\right),
\label{comred}%
\end{equation}%
where $\mu_{r,s}^{(0)}(t)$ is the corresponding $\mu^{(0)}(t)$
of $\lambda_{r,s}^{(0)}(t)$.
We also use notation $\mu_{r,s}^{(n)}(t)$ via (\ref{II-128-0}).}

\begin{remark}
\label{identify} Clearly for any $m_{1},m_{2}\in \mathbb{Z}$, $\pm
p_{r,s}^{(n)}(\tau)+m_{1}+m_{2}\tau$ is also a solution of EPVI$(\frac{1}{2}(n+\frac12)^2,\frac{1}{8},\frac{1}{8},\frac
{1}{8})$. Since they all give the same $\lambda_{r,s}^{(n)}(t)$ via
(\ref{II-130}), in this paper we always identify all these solutions $\pm
p_{r,s}^{(n)}(\tau)+m_{1}+m_{2}\tau$ with the same one
$p_{r,s}^{(n)}(\tau)$.
\end{remark}

\subsection{Generalization of Hitchin's formula}
In this section, we exploit the
Okamoto transformation to study the explicit expression of $\wp(p_{r,s}^{(n)}(\tau)|\tau)$,
which can be seen as a generalization of Hitchin's formula (\ref{II-1}).
First we recall the Okamoto transformation $\kappa_{5}$ defined in (\ref{II-120}).

\begin{lemma}
\label{lemII-2-5}Let $(\lambda(t),\mu(t))$ be a solution of the Hamiltonian
system (\ref{aa}) with parameter $\theta \in \mathcal{K}$. Then $(\tilde
{\lambda}(t),\tilde{\mu}(t)):=(\kappa_{5}(\lambda)(t),\kappa_{5}(\mu)(t))$,
which is a solution of the Hamiltonian system (\ref{aa}) with new parameter
\[
\tilde{\theta}=(\tilde{\theta}_{0},\tilde{\theta}_{1},\tilde{\theta}%
_{2},\tilde{\theta}_{3},\tilde{\theta}_{4}):=\kappa_{5}(\theta)=(\theta
_{0}-1,\theta_{1},\theta_{2},\theta_{3},\theta_{4}+2),
\]
are expressed as{\allowdisplaybreaks%
\begin{align*}
\tilde{\lambda}(t)  &  =\hat{\lambda}(t)+\frac{1-\theta_{0}}{\tilde{\mu}%
(t)},\\
\tilde{\mu}(t)  &  =\hat{\mu}(t)+\frac{\theta_{0}+\theta_{1}-1}{\hat{\lambda
}(t)}+\frac{\theta_{0}+\theta_{2}-1}{\hat{\lambda}(t)-1}+\frac{\theta
_{0}+\theta_{3}-1}{\hat{\lambda}(t)-t},
\end{align*}
}where{\allowdisplaybreaks%
\begin{align*}
\hat{\lambda}(t)  &  =\bar{\lambda}(t)+\frac{1+\theta_{4}}{\hat{\mu}%
(t)},\text{ \  \  \ }\bar{\lambda}(t)=\lambda(t)+\frac{\theta_{0}+\theta_{4}%
}{\mu(t)},\\
\hat{\mu}(t)  &  =\mu(t)-\frac{\theta_{0}+\theta_{1}+\theta_{4}}{\text{\ }%
\bar{\lambda}(t)}-\frac{\theta_{0}+\theta_{2}+\theta_{4}}{\text{\ }%
\bar{\lambda}(t)-1}-\frac{\theta_{0}+\theta_{3}+\theta_{4}}{\text{\ }%
\bar{\lambda}(t)-t}.
\end{align*}
}
\end{lemma}

\begin{proof}
The proof is just a straightforward computation via Table 1.
\end{proof}

Since the Okamoto transformation is invertible, it is known from $\kappa^{0,n}(\theta^0)=\theta^n$ that $(\lambda
_{r,s}^{(n)}(t),\mu_{r,s}^{(n)}(t))$ can
be transformed into $(\lambda_{r,s}^{(m)}(t),\mu_{r,s}^{(m)}(t))$ via Okamoto transformations for any $n\not =m$.
In the following, we often omit the subscripts $r,s$ for convenience. The
following result gives the explicit expression of $(\lambda^{(n)}(t),\mu
^{(n)}(t))$ in terms of $(\lambda^{(n-1)}(t),\mu^{(n-1)}(t))$.

\begin{lemma}
\label{thm-II-16}Under the above notations, for $n\geq1$ there holds:%
\begin{align}
\mu^{(n)}=  &  \mu^{(n-1)}-\frac{n}{2}\bigg(\frac{1}{\lambda^{(n-1)}%
+\frac{n-1}{2\mu^{(n-1)}}}\label{II-500}\\
&  +\frac{1}{\lambda^{(n-1)}+\frac{n-1}{2\mu^{(n-1)}}-1}+\frac{1}%
{\lambda^{(n-1)}+\frac{n-1}{2\mu^{(n-1)}}-t}\bigg),\nonumber
\end{align}%
\begin{equation}
\lambda^{(n)}=\lambda^{(n-1)}+\frac{n-1}{2\mu^{(n-1)}}+\frac{n+1}{2\mu^{(n)}}.
\label{II-501}%
\end{equation}

\end{lemma}

\begin{proof}
We prove these two formulae via induction.\medskip

\textbf{Step 1.} We consider $n=1$.

Recalling Table 1 and $\kappa^{0,1}:=\kappa_{0}\circ \kappa_{3}%
\circ \kappa_{2}\circ \kappa_{1}$, we have $\kappa^{0,1}(\theta^{0})=\theta^{1}$.
Using the expressions of $\kappa_{j}$ given in Table 1, a straightforward
calculation gives%
\begin{equation}
\mu^{(1)}=\mu^{(0)}-\frac{1}{2}\left(  \frac{1}{\lambda^{(0)}}+\frac
{1}{\lambda^{(0)}-1}+\frac{1}{\lambda^{(0)}-t}\right)  , \label{II-502}%
\end{equation}%
\begin{equation}
\lambda^{(1)}=\lambda^{(0)}+\frac{1}{\mu^{(1)}}. \label{503}%
\end{equation}
This proves (\ref{II-500})-(\ref{II-501}) for $n=1$.\medskip

\textbf{Step 2.} Assume that (\ref{II-500})-(\ref{II-501}) hold for $n=m-1$,
where $m\geq2$, we claim that (\ref{II-500})-(\ref{II-501}) hold for $n=m$.

Since $\kappa_{5}(\theta^{m-2})=\theta^{m}$, we apply Lemma \ref{lemII-2-5} to
obtain
\begin{equation}
\lambda^{(m)}=\hat{\lambda}+\frac{m+1}{2\mu^{(m)}},\text{ }\mu^{(m)}=\hat{\mu
}-\frac{m}{2}\Big(\frac{1}{\hat{\lambda}}+\frac{1}{\hat{\lambda}-1}+\frac
{1}{\hat{\lambda}-t}\Big), \label{II-504}%
\end{equation}
where {\allowdisplaybreaks%
\begin{align}
\bar{\lambda}  &  =\lambda^{(m-2)}+\frac{m-2}{2\mu^{(m-2)}},\label{II-506}\\
\hat{\lambda}  &  =\bar{\lambda}(t)+\frac{2m-1}{2\hat{\mu}(t)}, \label{II-505}%
\\
\hat{\mu}  &  =\mu^{(m-2)}-\frac{m-1}{2}\left(  \frac{1}{\text{\ }\bar
{\lambda}(t)}-\frac{1}{\text{\ }\bar{\lambda}(t)-1}-\frac{1}{\text{\ }%
\bar{\lambda}(t)-t}\right)  . \label{II-508}%
\end{align}
}Since (\ref{II-500})-(\ref{II-501}) hold for $n=m-1$, it is easy to see from
(\ref{II-506}) and (\ref{II-508}) that $\hat{\mu}=\mu^{(m-1)}$. Substituting
$\hat{\mu}=\mu^{(m-1)}$ and (\ref{II-506}) into (\ref{II-505}), we have%
\[
\hat{\lambda}=\lambda^{(m-2)}+\frac{m-2}{2\mu^{(m-2)}}+\frac{2m-1}%
{2\mu^{(m-1)}}=\lambda^{(m-1)}+\frac{m-1}{2\mu^{(m-1)}}.
\]
Consequently, (\ref{II-504}) implies that (\ref{II-500})-(\ref{II-501}) hold
for $n=m$. This completes the proof.
\end{proof}

To state our main results of this section, we give some general settings. Fix
any $(r,s)\in \mathbb{C}^{2}\backslash \frac{1}{2}\mathbb{Z}^{2}$ and let
\begin{equation}
a(\tau)=r+s\tau,\text{ \ }Z_{r,s}(\tau)=\zeta(a(\tau)|\tau)-r\eta_{1}%
(\tau)-s\eta_{2}(\tau). \label{II-46-7}%
\end{equation}
Note from Theorem B that
\begin{equation}
\lambda_{r,s}^{(0)}(t)=\frac{\wp(p_{r,s}^{(0)}(\tau)|\tau)-e_{1}(\tau)}{e_{2}(\tau)-e_{1}(\tau)}%
=\frac{\wp(a(\tau)|\tau)+\frac{\wp^{\prime}\left(  a(\tau)|\tau \right)
}{2Z_{r,s}(\tau)}-e_{1}(\tau)}{e_{2}(\tau)-e_{1}(\tau)}. \label{II-126}%
\end{equation}
Then
\cite[(4.21)]{Chen-Kuo-Lin} shows that the corresponding $\mu_{r,s}^{(0)}(t)$ is given by
\begin{equation}
\mu_{r,s}^{(0)}(t)=\frac{e_{2}(\tau)-e_{1}(\tau)}{2(\wp(p_{r,s}^{(0)}(\tau)|\tau)-\wp(a(\tau)|\tau
))}=\frac{(e_{2}(\tau)-e_{1}(\tau))Z_{r,s}(\tau)}{\wp^{\prime}(a(\tau)|\tau)}.
\label{II-127}%
\end{equation}
We rewrite them as follows:
\begin{equation}
\lambda^{(0)}(t)=\lambda_{r,s}^{(0)}(t)=q_{0}(Z_{r,s}(\tau)),\text{ \ }%
\mu^{(0)}(t)=\mu_{r,s}^{(0)}(t)=p_{0}(Z_{r,s}(\tau)), \label{II-46}%
\end{equation}
where%
\begin{equation}
q_{0}(X):=\frac{R_{0}(X)}{Q_{0}(X)},\text{ \  \ }p_{0}(X):=\frac{Q_{0}%
(X)}{G_{0}(X)}, \label{II-513}%
\end{equation}%
\begin{equation}
Q_{0}(X)=X,\text{ \ }G_{0}(X)=\frac{\wp^{\prime}(a(\tau)|\tau)}{e_{2}%
(\tau)-e_{1}(\tau)}, \label{II-511}%
\end{equation}%
\begin{equation}
R_{0}(X)=\frac{\wp(a(\tau)|\tau)-e_{1}(\tau)}{e_{2}(\tau)-e_{1}(\tau)}%
X+\frac{1}{2}\frac{\wp^{\prime}(a(\tau)|\tau)}{e_{2}(\tau)-e_{1}(\tau)}.
\label{II-512}%
\end{equation}
In the following, we fix any $\tau \in \mathbb{H}$ such that $a(\tau
)\not \in E_{\tau}[2]$ and let $t=t(\tau)$. Clearly $R_{0}(X)$, $Q_{0}(X)$ and
$G_{0}(X)$ satisfy

\begin{itemize}
\item[(1-$0$)] $Q_{0}(X)$ is a polynomial of degree $1$; $R_{0}(X)$ is a
polynomial of degree $1$; $G_{0}(X)$ is a non-zero constant.

\item[(2-$0$)] any two of $\{Q_{0}(X),G_{0}(X),R_{0}(X)\}$ has no common zeros.

\item[(3-$0$)] $Q_{0}(X)|(R_{0}(X)-\frac{1}{2}G_{0}(X))$.

\item[(4-$0$)] $\deg(R_{0}(X)-Q_{0}(X))=\deg(R_{0}(X)-tQ_{0}(X))=\deg
R_{0}(X)=1$.
\end{itemize}

\noindent Remark that $a(\tau)\not \in E_{\tau}[2]$ is the \emph{key}
assumption, because these properties can \emph{not} hold if $a(\tau)\in
E_{\tau}[2]$.

For convenience, we define
\begin{align}
H(x;y)  &  :=x(x-y)(x-ty),\label{II-510}\\
H^{\prime}(x;y)  &  :=\frac{\partial H}{\partial x}%
=x(x-y)+x(x-ty)+(x-y)(x-ty).\nonumber
\end{align}
Clearly we have%
\begin{equation}
H(bx;by)=b^{3}H(x;y),\text{ \  \ }H^{\prime}(bx;by)=b^{2}H^{\prime}(x;y).
\label{II-510-1}%
\end{equation}
Now starting from $Q_{0}(X)$, $G_{0}(X)$ and $R_{0}(X)$ defined in
(\ref{II-511})-(\ref{II-512}) and setting $Q_{-2}(X)=Q_{-1}(X)=1$, we can
define $Q_{n}(X)$, $G_{n}(X)$, $R_{n}(X)$ for $n\geq1$ by induction as
follows:%
\begin{equation}
\varphi_{n-1}(X):=R_{n-1}(X)Q_{n-2}(X)+\tfrac{n-1}{2}G_{n-1}(X), \label{II-515}%
\end{equation}%
\begin{equation}
G_{n}(X):=\frac{H(\varphi_{n-1}(X);Q_{n-3}(X)Q_{n-2}(X)Q_{n-1}(X))}%
{G_{n-1}(X)Q_{n-3}(X)^{3}}, \label{II-516}%
\end{equation}
{\allowdisplaybreaks%
\begin{align}
Q_{n}(X):=  &  \frac{H(\varphi_{n-1}(X);Q_{n-3}(X)Q_{n-2}(X)Q_{n-1}%
(X))}{G_{n-1}(X)^{2}Q_{n-3}(X)^{2}}\nonumber \\
&  -\frac{n}{2}\frac{H^{\prime}(\varphi_{n-1}(X);Q_{n-3}(X)Q_{n-2}%
(X)Q_{n-1}(X))}{G_{n-1}(X)Q_{n-3}(X)^{2}}\nonumber \\
=  &  \frac{Q_{n-3}(X)G_{n}(X)}{G_{n-1}(X)}-\frac{n}{2}\frac{H^{\prime
}(\varphi_{n-1}(X);Q_{n-3}(X)Q_{n-2}(X)Q_{n-1}(X))}{G_{n-1}(X)Q_{n-3}(X)^{2}},
\label{II-517}%
\end{align}
}{\allowdisplaybreaks%
\begin{equation}
R_{n}(X):=\frac{\varphi_{n-1}(X)Q_{n}(X)+\frac{n+1}{2}Q_{n-3}(X)G_{n}%
(X)}{Q_{n-3}(X)Q_{n-1}(X)}. \label{II-518}%
\end{equation}
We will prove that }$Q_{n}(X)$, $G_{n}(X)$, $R_{n}(X)$ are all
\emph{polynomials}; see Theorem \ref{thm-II-17} below.

Recalling $(p_{0}(X),q_{0}(X))$ defined in (\ref{II-46})-(\ref{II-513}), we
exploit (\ref{II-500})-(\ref{II-501}) to define rational function pairs
$(p_{n}(X),q_{n}(X))$ for all $n\geq1$ by induction:%
\begin{align}
p_{n}(X):=  &  p_{n-1}(X)-\frac{n}{2}\bigg(\frac{1}{q_{n-1}(X)+\frac
{n-1}{2p_{n-1}(X)}}\label{II-500-1}\\
&  +\frac{1}{q_{n-1}(X)+\frac{n-1}{2p_{n-1}(X)}-1}+\frac{1}{q_{n-1}%
(X)+\frac{n-1}{2p_{n-1}(X)}-t}\bigg),\nonumber
\end{align}%
\begin{equation}
q_{n}(X):=q_{n-1}(X)+\frac{n-1}{2p_{n-1}(X)}+\frac{n+1}{2p_{n}(X)}.
\label{II-501-1}%
\end{equation}
Here is the first main result of this section.

\begin{theorem}
\label{thm-II-17}Under the assumption $a(\tau)\not \in E_{\tau}[2]$ and the
above notations, for any $n\geq1$ we have{\allowdisplaybreaks%
\begin{align}
p_{n}(X)  &  =\frac{Q_{n-2}(X)Q_{n-1}(X)Q_{n}(X)}{G_{n}(X)},\label{II-514}\\
q_{n}(X)  &  =\frac{R_{n}(X)}{Q_{n-2}(X)Q_{n}(X)}, \label{II-514-1}%
\end{align}
}where $Q_{-2}(X)=Q_{-1}(X)=1$, $Q_{0}(X)$, $R_{0}(X)$, $G_{0}(X)$ are given
by (\ref{II-511})-(\ref{II-512}), and $Q_{n}(X)$, $G_{n}(X)$, $R_{n}(X)$ are
given by induction in (\ref{II-515})-(\ref{II-518}). Furthermore, $Q_{n}(X)$,
$G_{n}(X)$, $R_{n}(X)$ satisfy

\begin{itemize}
\item[(1-$n$)] $Q_{n}(X)$ is a polynomial of degree $\frac{(n+1)(n+2)}{2}$;
$G_{n}(X)$ is a polynomial of degree $\frac{3n(n+1)}{2}$; $R_{n}(X)$ is a
polynomial of degree $n(n+1)+1$. Their coefficients are all rational functions
of $e_{1}(\tau)$, $e_{2}(\tau)$, $e_{3}(\tau)$, $\wp(a(\tau)|\tau)$ and
$\wp^{\prime}(a(\tau)|\tau)$ with coefficients in $\mathbb{Q}$.

\item[(2-$n$)] any two of $\{Q_{n-2}(X),Q_{n-1}(X),Q_{n}(X),G_{n}(X)\}$ have
no common zeros; any two of $\{Q_{n-2}(X),Q_{n}(X),R_{n}(X)\}$ have no common
zeros; $G_{n-1}(X)$ and $G_{n}(X)$ have no common zeros.

\item[(3-$n$)] $Q_{n-2}(X)|\varphi_{n}(X)$ and $Q_{n}(X)|(R_{n}(X)Q_{n-1}%
(X)-\frac{n+1}{2}G_{n}(X))$.

\item[(4-$n$)] $\deg(R_{n}-Q_{n-2}Q_{n})=\deg(R_{n}-tQ_{n-2}Q_{n})=\deg
R_{n}=n(n+1)+1.$
\end{itemize}
\end{theorem}

\begin{remark}
The proof of Theorem \ref{thm-II-17} is technical and will be given in Appendix \ref{appendix-A}. Although the expressions (\ref{II-515})-(\ref{II-518}) of
$Q_{n}(X)$, $G_{n}(X)$, $R_{n}(X)$ defined by induction look complicated, they
turn out to be very useful. For example, they will be applied to
calculate the asymptotics of $Z_{r,s}^{(n)}(\tau)$ in \S \ref{sec-asymptotics}.
\end{remark}

As an application of Theorem \ref{thm-II-17}, we have the following important
result, which is a generalization of Hitchin's formula (\ref{II-1}) to
solutions of PVI$(\tfrac{1}{2}(n+\tfrac{1}{2})^{2},\tfrac{-1}{8},\tfrac{1}%
{8},\tfrac{3}{8})$. Since for any fixed $(r,s)\in \mathbb{C}^{2}\backslash
\frac{1}{2}\mathbb{Z}^{2}$, Property (1-$n$) shows that the coefficients
depend on $\tau$ as $\tau \in \mathbb{H}$ deforms, we also denote $Q_{n}(X)$,
$G_{n}(X)$, $R_{n}(X)$ by $Q_{n}(X|\tau)$, $G_{n}(X|\tau)$ and $R_{n}(X|\tau)$ respectively.

\begin{theorem}
\label{thm-II-18}Fix any $(r,s)\in \mathbb{C}^{2}\backslash \frac{1}%
{2}\mathbb{Z}^{2}$ and let $a(\tau)=r+s\tau$, $Z_{r,s}(\tau)=\zeta
(a(\tau)|\tau)-r\eta_{1}(\tau)-s\eta_{2}(\tau)$ as before. Then for any
$n\geq0$ the following hold:

\begin{itemize}
\item[(i)] $\lambda_{r,s}^{(n)}(t)$, $\mu_{r,s}^{(n)}(t)$ and $p_{r,s}%
^{(n)}(\tau)$ are expressed by%
\begin{equation}
\lambda_{r,s}^{(n)}(t)=q_{n}(Z_{r,s}(\tau))=\frac{R_{n}(Z_{r,s}(\tau)|\tau
)}{Q_{n-2}(Z_{r,s}(\tau)|\tau)Q_{n}(Z_{r,s}(\tau)|\tau)}, \label{II-535}%
\end{equation}%
\begin{align}
\mu_{r,s}^{(n)}(t)  &  =p_{n}(Z_{r,s}(\tau))\label{II-536}\\
&  =\frac{Q_{n-2}(Z_{r,s}(\tau)|\tau)Q_{n-1}(Z_{r,s}(\tau)|\tau)Q_{n}%
(Z_{r,s}(\tau)|\tau)}{G_{n}(Z_{r,s}(\tau)|\tau)},\nonumber
\end{align}%
\begin{equation}
\wp(p_{r,s}^{(n)}(\tau)|\tau)=\frac{(e_{2}(\tau)-e_{1}(\tau))R_{n}%
(Z_{r,s}(\tau)|\tau)}{Q_{n-2}(Z_{r,s}(\tau)|\tau)Q_{n}(Z_{r,s}(\tau)|\tau
)}+e_{1}(\tau). \label{II-537}%
\end{equation}

\item[(ii)] Let $t_{0}=t(\tau_{0})$ such that $a(\tau_{0})\not \in E_{\tau
_{0}}[2]$, then $t_{0}$ is a positive pole of $\lambda_{r,s}^{(n)}(t)$ if and
only if $Q_{n-2}(Z_{r,s}(\tau_{0})|\tau_{0})=0$; $t_{0}$ is a negative pole of
$\lambda_{r,s}^{(n)}(t)$ if and only if $Q_{n}(Z_{r,s}(\tau_{0})|\tau_{0})=0$.
In particular, both $\lambda_{r,s}^{(0)}(t)$ and $\lambda_{r,s}^{(1)}(t)$ have
no positive poles in $\mathbb{C}\backslash \{0,1\}$ provided $(r,s)\in
\mathbb{R}^{2}\backslash \frac{1}{2}\mathbb{Z}^{2}$.
\end{itemize}
\end{theorem}

\begin{remark}
For Hitchin's formula (\ref{II-1}), the result that any zero $\tau_{0}$ of
$Z_{r,s}(\tau)$ gives a negative pole $t(\tau_{0})$ of $\lambda_{r,s}%
^{(0)}(t)$ was proved by Hitchin in \cite[Proposition 9]{Hit1}. Theorem
\ref{thm-II-18}-(ii), which extends this result to the general case $n\geq1$,
will be applied to establish the relation $Q_{n}(Z_{r,s}(\tau
)|\tau)=\mathfrak{q}_{n}(\tau)Z_{r,s}^{(n+1)}(\tau)$ for some non-zero
constant $\mathfrak{q}_{n}(\tau)$ in \S \ref{expression-completely}.4.
\end{remark}

\begin{proof}
[Proof of Theorem \ref{thm-II-18}](i) For the formulae (\ref{II-535}%
)-(\ref{II-536}), the case $n=0$ follows from (\ref{II-46}), and the general
case $n\geq1$ follows directly from Lemma \ref{thm-II-16} and (\ref{II-500-1}%
)-(\ref{II-501-1}). Clearly (\ref{II-535}) implies (\ref{II-537}). This proves (i).

(ii) Let $t_{0}=t(\tau_{0})$ such that $a(\tau_{0})\not \in E_{\tau_{0}}[2]$.
Then $a(\tau)\not \in E_{\tau}[2]$ for $\tau$ in a small neighborhood $U$ of
$\tau_{0}$. Consequently, $Q_{n-2}(X|\tau)$, $Q_{n}(X|\tau)$ and $R_{n}%
(X|\tau)$ are all well-defined polynomials for each fixed $\tau \in U$ and
holomorphically depend on $\tau \in U$ (because the coefficients of these
polynomials are rational functions of $e_{1}(\tau)$, $e_{2}(\tau)$,
$e_{3}(\tau)$, $\wp(a(\tau)|\tau)$ and $\wp^{\prime}(a(\tau)|\tau)$ and take
finite values provided $a(\tau)\not \in E_{\tau}[2]$). Therefore, by property
(2-$n$) in Theorem \ref{thm-II-17}, we conclude that $t_{0}$ is a pole of
$\lambda_{r,s}^{(n)}(t)$ if and only if
\[
\text{either }Q_{n-2}(Z_{r,s}(\tau_{0})|\tau_{0})=0\text{ \  \ or \  \ }%
Q_{n}(Z_{r,s}(\tau_{0})|\tau_{0})=0.
\]
If $Q_{n-2}(Z_{r,s}(\tau_{0})|\tau_{0})=0$, then properties (2-$n$)-(3-$n$) in
Theorem \ref{thm-II-17} imply%
\[
\lim_{t\rightarrow t_{0}}\lambda_{r,s}^{(n)}\mu_{r,s}^{(n)}=\frac
{R_{n}(Z_{r,s}(\tau_{0})|\tau_{0})Q_{n-1}(Z_{r,s}(\tau_{0})|\tau_{0})}%
{G_{n}(Z_{r,s}(\tau_{0})|\tau_{0})}=\frac{-n}{2}=-(\theta_0+\theta_4),
\]
where we use (\ref{I1I-175}) to obtain the last equality.
This, together with Proposition \ref{Prop-II-1}, shows that
$t_{0}$ is a positive pole. If $Q_{n}(Z_{r,s}(\tau_{0})|\tau_{0})=0$, then
properties (2-$n$)-(3-$n$) in Theorem \ref{thm-II-17} imply%
\[
\lim_{t\rightarrow t_{0}}\lambda_{r,s}^{(n)}\mu_{r,s}^{(n)}=\frac
{R_{n}(Z_{r,s}(\tau_{0})|\tau_{0})Q_{n-1}(Z_{r,s}(\tau_{0})|\tau_{0})}%
{G_{n}(Z_{r,s}(\tau_{0})|\tau_{0})}=\frac{n+1}{2}=-\theta_0.
\]
By Proposition \ref{Prop-II-1} again, we see that $t_{0}$ is a negative pole.
Since $Q_{-2}=Q_{-1}\equiv1$, we see that the assertion (ii) holds. This completes
the proof.
\end{proof}

\begin{corollary}
\label{thm-II-18-1}For any $n\geq0$ and fixed real pair $(r,s)\in
\mathbb{R}^{2}\backslash \frac{1}{2}\mathbb{Z}^{2}$, $Q_{n}(Z_{r,s}(\tau
)|\tau)$ has only simple zeros as a holomorphic function of $\tau \in
\mathbb{H}$.
\end{corollary}

\begin{proof}
For any $(r,s)\in \mathbb{R}^{2}\backslash \frac{1}{2}\mathbb{Z}^{2}$,
$a(\tau)\not \in E_{\tau}[2]$ for all $\tau \in \mathbb{H}$, which implies that
$Q_{n}(X|\tau)$ is a well-defined polynomial for each fixed $\tau$ and
holomorphically depend on $\tau \in \mathbb{H}$. Thus $Q_{n}(Z_{r,s}(\tau
)|\tau)$ is holomorphic as a function of $\tau \in \mathbb{H}$. Let $\tau_{0}$
be any a zero of $Q_{n}(Z_{r,s}(\tau)|\tau)$, then $t(\tau_{0})$ is a pole of
$\lambda_{r,s}^{(n)}(t)$. By Theorem \ref{thm-2A} we know that $t(\tau_{0})$ is a
simple pole of $\lambda_{r,s}^{(n)}(t)$. Since $t^{\prime}(\tau_{0})\not =0$
and Property (2-$n$) in Theorem \ref{thm-II-17} shows $R_{n}(Z_{r,s}(\tau
_{0})|\tau_{0})\not =0$, we conclude that $\tau_{0}$ is a simple zero of
$Q_{n}(Z_{r,s}(\tau)|\tau)$.
\end{proof}

Now we give two examples. For convenience we write $\wp=\wp(a(\tau)|\tau)$,
$Z=Z_{r,s}(\tau)$ and $e_{k}=e_{k}(\tau)$.

\begin{theorem}
\label{thm-expression-1}Let $n=1$ in Theorem \ref{thm-II-18}. Then%
\begin{align}\label{expression-1-0}
\lambda_{r,s}^{(1)}(t)
=\frac{(\wp-e_{1})Z^{3}+\frac{3\wp^{\prime}}{2}Z^{2}+\frac
{6\wp^{2}+6e_{1}\wp-g_{2}}{2}Z+\frac{\wp+2e_{1}}{2}\wp^{\prime}}{(e_{2}-e_{1})(Z^{3}-3\wp Z-\wp^{\prime})},
\end{align}%
\[
\mu_{r,s}^{(1)}(t)=\frac{2(e_{2}-e_{1})Z(Z^{3}-3\wp Z-\wp^{\prime})}%
{2\wp^{\prime}Z^{3}+(12\wp^{2}-g_{2})Z^{2}+6\wp\wp^{\prime}Z+(\wp^{\prime})^{2}},
\]
\begin{equation}
\wp(p_{r,s}^{(1)}(\tau)|\tau)=\wp+\frac{3\wp^{\prime}Z^{2}+\left(
12\wp^{2}-g_{2}\right)  Z+3\wp\wp^{\prime}}{2(Z^{3}-3\wp
Z-\wp^{\prime})}. \label{expression-1}%
\end{equation}
That is, when $(r,s)\in \mathbb{C}^{2}\backslash \frac{1}{2}\mathbb{Z}^{2}$,
(\ref{expression-1}) gives solutions of EPVI$(\frac{9}{8},\frac{1}{8},\frac{1}{8},\frac{1}{8})$; or equivalently,
(\ref{expression-1-0}) gives solutions of
PVI$(\frac{9}{8},\frac{-1}{8},\frac{1}{8},\frac{3}{8})$.
\end{theorem}

\begin{proof}
Recalling (\ref{II-511}), (\ref{II-512}) and (\ref{II-515})-(\ref{II-518}), a
direct computation gives%
\[
G_{1}(X)=\frac{\wp^{\prime}X^{3}+(6\wp^{2}-\frac{g_{2}}{2})X^{2}%
+3\wp\wp^{\prime}X+\frac{(\wp^{\prime})^{2}}{2}}{4(e_{2}-e_{1})^{2}},
\]%
\[
Q_{1}(X)=\frac{X^{3}-3\wp X-\wp^{\prime}}{4(e_{2}-e_{1})},
\]
\[
R_{1}(X)=\frac{(\wp-e_{1})X^{3}+\frac{3\wp^{\prime}}{2}X^{2}+\frac
{6\wp^{2}+6e_{1}\wp-g_{2}}{2}X+\frac{\wp+2e_{1}}{2}\wp^{\prime}}{4(e_{2}-e_{1})^{2}}.
\]
Then this theorem follows readily from (\ref{II-535})-(\ref{II-537}).
\end{proof}

Formula (\ref{expression-1}), which is a generalization of Hitchin's formula
(\ref{II-1}) to PVI$(\frac{9}{8},\frac{-1}{8},\frac{1}{8},\frac{3}{8})$, was
first obtained by Takemura \cite{Takemura} via a different method.
The following formula
(\ref{expression-2}) for PVI$(\frac{25}{8},\frac{-1}{8},\frac{1}{8},\frac
{3}{8})$ can not be found in the literature and is new.

\begin{theorem}
\label{thm-expression-2}Let $n=2$ in Theorem \ref{thm-II-18}. Then%
\begin{equation}
\wp(p_{r,s}^{(2)}(\tau)|\tau)=\wp+\frac{\Xi^{(2)}(Z)}{8Z
Z_{r,s}^{(3)}(\tau)}, \label{expression-2}%
\end{equation}
where $Z_{r,s}^{(3)}(\tau)$ is given in (\ref{z-n-4}) and
{\allowdisplaybreaks%
\begin{align*}
\Xi^{(2)}(Z)=  &  28\wp^{\prime}Z^{6}+\left(  288\wp^{2}-24g_{2}\right)  Z^{5}+300\wp\wp^{\prime}Z
^{4}\\
&  +\left(  640\wp^{3}-88g_{2}\wp-52g_{3}\right)  Z^{3}\\
&  +(180\wp^{2}-3g_{2})\wp^{\prime}Z^{2}+24\wp(\wp^{\prime})^{2}Z+(\wp^{\prime})^{3}.
\end{align*}
}That is, when $(r,s)\in \mathbb{C}^{2}\backslash \frac{1}{2}\mathbb{Z}^{2}$,
(\ref{expression-2}) gives solutions of EPVI$(\frac{25}{8},\frac{1}{8},\frac{1}{8},\frac{1}{8})$.
\end{theorem}
Theorem \ref{thm-expression-2} can be proved similarly as Theorem \ref{thm-expression-1};
we just need to calculate $G_{2}$, $Q_{2}$, $R_{2}$ by using Theorem \ref{thm-II-17}. We omit the details here.

Now we turn to the special case $(r,s)=(\frac{1}{4},0)$ and have the following
simple observation.

\begin{theorem}
\label{r=1/4}For any $n\geq0$, there holds%
\begin{equation}
\lambda_{\frac{1}{4},0}^{(n)}(t)=\frac{(-1)^{n}}{2n+1}t^{\frac{1}{2}},\text{
\ }\mu_{\frac{1}{4},0}^{(n)}(t)=\frac{1}{4\lambda_{\frac{1}{4},0}^{(n)}%
(t)}=(-1)^{n}\frac{2n+1}{4}t^{-\frac{1}{2}}. \label{vi-0}%
\end{equation}
In particular, $t=1$ is not a branch point of $\lambda_{\frac{1}{4},0}%
^{(n)}(t)$.
\end{theorem}

\begin{remark}
A simple computation shows that $\frac{(-1)^{n}}{2n+1}t^{\frac{1}{2}}$ is really
 a solution of
 PVI$(\frac12(n+\frac12)^2,\frac{-1}{8},\frac{1}{8},\frac{3}{8})$ for
\emph{any} $n\geq0$. Hitchin \cite{Hit2} already knew that $t^{\frac{1}{2}}$
is a solution of PVI$(\frac{1}{8},\frac{-1}{8},\frac{1}{2k^{2}},\frac{1}%
{2}-\frac{1}{2k^{2}})$ for any $k\in \mathbb{N}$ and hence a solution of
PVI$(\frac{1}{8},\frac{-1}{8},\frac{1}{8},\frac{3}{8})$. Here we need to prove
$t^{\frac{1}{2}}=\lambda_{\frac{1}{4},0}^{(0)}(t)$. Formula (\ref{vi-0}) will
play a key role in our proof of Theorem \ref{weak}-(3) in \S \ref{sec-asymptotics}.
\end{remark}

\begin{proof}
[Proof of Theorem \ref{r=1/4}]In this proof we always consider $(r,s)=(\frac
{1}{4},0)$ and write $\lambda^{(n)}(t)=\lambda_{\frac{1}{4},0}^{(n)}(t)$,
$\mu^{(n)}(t)=\mu_{\frac{1}{4},0}^{(n)}(t)$ for convenience. First we prove
(\ref{vi-0}) for $n=0$. Recalling the Hamiltonian system (\ref{aa})-(\ref{98})
and (\ref{I1I-175}), we know that $(\lambda^{(0)},\mu^{(0)})$ satisfies%
\begin{equation}
\frac{d\lambda^{(0)}}{dt}=\frac{(2\lambda^{(0)}\mu^{(0)}-\frac{1}{2}%
)(\lambda^{(0)}-1)(\lambda^{(0)}-t)+\frac{1}{2}(t-1)\lambda^{(0)}}{t(t-1)},
\label{vi-2}%
\end{equation}%
\begin{equation}
\frac{d\mu^{(0)}}{dt}=-\frac{[3(\lambda^{(0)})^{2}-2(t+1)\lambda^{(0)}%
+t](\mu^{(0)})^{2}-(\lambda^{(0)}-t)\mu^{(0)}}{t(t-1)}. \label{vi-3}%
\end{equation}

Recall the addition formula%
\[
\zeta(u+v)+\zeta(u-v)-2\zeta(u)=\frac{\wp^{\prime}(u)}{\wp(u)-\wp(v)}.
\]
Letting $u=\frac{1}{4}$ and $v=\frac{\omega_{1}}{2}=\frac{1}{2}$ leads to%
\begin{align*}
\frac{\wp^{\prime}(\frac{1}{4})}{\wp(\frac{1}{4})-e_{1}}  &  =\zeta \left(
3/4\right)  +\zeta \left(  -1/4\right)  -2\zeta \left(  1/4\right) \\
&  =\eta_{1}-4\zeta \left(  1/4\right)  =-4Z_{\frac{1}{4},0}(\tau).
\end{align*}
This identity, together with $a(\tau)=r+s\tau=\frac{1}{4}$ and (\ref{II-46}%
)-(\ref{II-512}), easily implies $4R_{0}(Z_{\frac{1}{4},0}(\tau)|\tau
)=G_{0}(Z_{\frac{1}{4},0}(\tau)|\tau)$ and so%
\begin{equation}
\lambda^{(0)}(t)\mu^{(0)}(t)\equiv \tfrac{1}{4}. \label{vi-1}%
\end{equation}
Inserting (\ref{vi-1}) into (\ref{vi-2}) gives $\frac{d\lambda^{(0)}}%
{dt}=\frac{\lambda^{(0)}}{2t}$ and then $\frac{d\mu^{(0)}}{dt}=\frac
{-1}{8t\lambda^{(0)}}$. Substituting this and (\ref{vi-1}) into (\ref{vi-3})
we easily obtain
\[
(\lambda^{(0)}(t))^{2}\equiv t.
\]
This proves (\ref{vi-0}) for $n=0$ (In \S \ref{sec-asymptotics} we will prove
$\lambda^{(0)}(t)\rightarrow1$ as $t=t(\tau)\rightarrow1$ by letting $F_{2}%
\ni \tau \rightarrow \infty$, where $F_{2}$ is a fundamental domain of
$\Gamma(2)$ defined in (\ref{funde}); see (\ref{iv-53})). Finally, formula
(\ref{vi-0}) for all $n\geq1$ can be proved via Lemma \ref{thm-II-16} by
induction. Since the proof is trivial, we omit the details here. This
completes the proof.
\end{proof}

\subsection{Connection between PVI and $Z_{r,s}^{(n)}(\tau)$}

In this section, we establish the precise connection of PVI
(\ref{46-0}) and the pre-modular form $Z_{r,s}^{(n)}(\tau)$. Consequently, we give the proof
of Theorem \ref{thm-II-18 copy(1)} and Theorem \ref{simple-zero con}.

Let $n\geq1$.
Recall Theorem \ref{thm-II-17} that $Q_{n}(X)$ is a polynomial of degree
$\frac{(n+1)(n+2)}{2}$ with coefficients being rational functions of
$e_{k}(\tau)$'s, $\wp(a(\tau)|\tau)$ and $\wp^{\prime}(a(\tau)|\tau)$,
provided that $a(\tau)\not \in E_{\tau}[2]$. In this section we denote it by
$Q_{n}(X)=Q_{n}(X;a(\tau),\tau)$. Then we have

\begin{theorem}
\label{q-n=z-n}Let $\sigma_{n+1}=a(\tau)\not \in \Lambda_{\tau}$ in
(\ref{k-ll}). Then for any $n\geq0$ there holds
\begin{equation}
Q_{n}(X;a(\tau),\tau)=\mathfrak{q}_{n}(\tau)W_{n+1}(X;a(\tau),\tau),
\label{caxi}%
\end{equation}
where $\mathfrak{q}%
_{n}(\tau)\neq 0$ denotes the coefficient of the leading
term $X^{\frac{(n+1)(n+2)}{2}}$ of the polynomial $Q_{n}%
(X;a(\tau),\tau)$, and its expression will be given in Lemma \ref{coefficient-Qn}.

\end{theorem}

To prove Theorem \ref{q-n=z-n},
 we need to apply the monodromy theory of the associated linear
ODE for $p_{r,s}^{(n)}(\tau)$. For EPVI$(\frac{1}{2}(n+\frac12)^2,\frac{1}{8},\frac{1}{8},\frac
{1}{8})$, one choice of its associated linear ODE is the generalized
Lam\'{e} equation
(denoted it by GLE$(n,p,A,\tau)$)
\begin{equation}
y^{\prime \prime}=\left[
\begin{array}
[c]{l}%
n(n+1)\wp(z)+\frac{3}{4}(\wp(z+p)+\wp(z-p))\\
+A(\zeta(z+p)-\zeta(z-p))+B
\end{array}
\right]  y=:I(z)y, \label{89-1}%
\end{equation}
where $\pm p\not \in E_{\tau}[2]$ are always assumed to be
\emph{apparent singularities} (i.e. non-logarithmic), which is equivalent
to (see \cite[Lemma 2.1]{Chen-Kuo-Lin0})
\begin{equation}
B=A^{2}-\zeta(2p)A-\tfrac{3}{4}\wp(2p)-n(n+1)\wp (
p)  . \label{101}%
\end{equation}

Fix any base point $q_{0}\in E_{\tau}\backslash(\{  \pm \lbrack
p]\}  \cup E_{\tau}[2])$. The monodromy representation of GLE
(\ref{89-1}) is a homomorphism $\rho:\pi_{1}(  E_{\tau}\backslash
( \{  \pm \lbrack p] \}  \cup E_{\tau}[2]),q_{0})  \rightarrow
SL(2,\mathbb{C})$. Since $n\in \mathbb{Z}_{\geq 0}$ and the local exponents
of (\ref{89-1}) at $0$ are $-n$ and $n+1$, the
local monodromy matrix at $0$ is $I_{2}$. Thus the
monodromy representation is reduced to $\rho:\pi
_{1}(  E_{\tau}\backslash \{  \pm \lbrack p]\}  ,q_{0})
\rightarrow SL(2,\mathbb{C})$. Let $\gamma_{\pm}\in \pi_{1}(  E_{\tau
}\backslash(\{ \pm \lbrack p]\} \cup E_{\tau}[2]),q_{0})  $ be a simple
loop encircling $\pm p$ counterclockwise respectively, and $\ell_{j}\in \pi
_{1}(  E_{\tau}\backslash(\{ \pm \lbrack p]\} \cup E_{\tau}[2]),q_{0}
)  $, $j=1,2$, be two fundamental cycles of $E_{\tau}$ connecting
$q_{0}$ with $q_{0}+\omega_{j}$ such that $\ell_{j}$ does not intersect with
$L+\Lambda_{\tau}$ (here $L$ is the straight segment connecting $\pm p$) and
satisfies%
\[
\gamma_{+}\gamma_{-}=\ell_{1}\ell_{2}\ell_{1}^{-1}\ell_{2}^{-1}\text{ in }%
\pi_{1}(  E_{\tau}\backslash\{  \pm \lbrack p] \}
,q_{0})  .
\]
Since the local exponents of (\ref{89-1}) at $\pm p$ are $
\{-\frac{1}{2}, \frac{3}{2}\}$ and $\pm p\not \in E_{\tau}[2]$ are apparent singularities, we
always have
$\rho(\gamma_{\pm})=-I_{2}$.
Denote by $N_{j}=\rho(\ell_j)$ the monodromy matrix along the loop $\ell_{j}$ of GLE
(\ref{89-1}) with respect to any linearly independent solutions. Then the monodromy group of GLE (\ref{89-1}) is
generated by $\{-I_{2},N_{1},N_{2}\}$, i.e. is always \emph{abelian and so reducible}.
It is known (cf. \cite{CKL1}) that expect finitely many $A$'s for given $(\tau, p)$,
$N_1$ and $N_2$ can be diagonalized simultaneously, and more precisely, there exists
$(r,s)\in\mathbb{C}^2\setminus\frac12\mathbb{Z}^2$ such that
\[N_{1}=%
\begin{pmatrix}
e^{-2\pi is} & 0\\
0 & e^{2\pi is}%
\end{pmatrix},\quad N_{2}=%
\begin{pmatrix}
e^{2\pi ir} & 0\\
0 & e^{-2\pi ir}%
\end{pmatrix}.\]

Let
$U$ be an open subset of $\mathbb{H}$ such that $p(\tau)\not \in E_{\tau}[2]$
for any $\tau \in U$. Then we proved in \cite{Chen-Kuo-Lin0} that
$p(\tau)$ \emph{is a solution of EPVI$(\frac{1}{2}(n+\frac12)^2,\frac{1}{8},\frac{1}{8},\frac
{1}{8})$ if and only
if there exist }$A(\tau)$\emph{ (and the corresponding $B(\tau)$ via (\ref{101}))
 such that the associated GLE$(n,p(\tau),A(\tau),\tau)$ is monodromy preserving
as $\tau \in U$ deforms}. We refer the reader to \cite{Chen-Kuo-Lin0}
for the more general statement, where we proved that $(p(\tau),A(\tau))$ sloves a new Hamiltonian system which is equivalent to EPVI.
Here we need the following result \cite{Chen-Kuo-Lin}.

\begin{theorem}(\cite[Theorem 5.3]{Chen-Kuo-Lin}) \label{thm-II-8}
For $n\in\mathbb{Z}_{\geq 0}$, let $p^{(n)}(\tau)$ be a
solution to EPVI$(\frac{1}{2}(n+\frac12)^2,\frac{1}{8},\frac{1}{8},\frac
{1}{8})$. Then the following hold:

\begin{itemize}
\item[(1)]  For any $\tau$ satisfying $p^{(n)}
(\tau)\not \in E_{\tau}[2]$, the monodromy group of the
associated GLE$(n$, $ p^{(n)}(\tau), A^{(n)}(\tau), \tau)$ is generated by
\begin{equation}
\rho(\gamma_{\pm})=-I_{2},\text{ }N_{1}=%
\begin{pmatrix}
e^{-2\pi is} & 0\\
0 & e^{2\pi is}%
\end{pmatrix}
\text{, }N_{2}=%
\begin{pmatrix}
e^{2\pi ir} & 0\\
0 & e^{-2\pi ir}%
\end{pmatrix}
 \label{II-101}%
\end{equation}
if and
only if $(r,s)  \in \mathbb{C}^{2}\backslash \frac
{1}{2}\mathbb{Z}^{2}$ and $p^{(n)}(\tau)=p_{r,s}^{(n)
}(\tau)$ in the sense of Remark \ref{identify}.

\item[(2)] $\wp(p^{(n)}_{r_{1},s_{1}}(\tau)|\tau)\equiv \wp
(p_{r_{2},s_{2}}^{(n)}(\tau)|\tau)\Longleftrightarrow(r_{1},s_{1})  \equiv \pm(r_{2},s_{2})  \operatorname{mod}$
$\mathbb{Z}^{2}$.
\end{itemize}
\end{theorem}

Now we are in a position to prove Theorem \ref{q-n=z-n}.

\begin{proof}[Proof of Theorem \ref{q-n=z-n}]
It is known in \cite{LW2} that $W_{1}(X)=X$, so (\ref{II-511}) gives
$Q_{0}(X)=W_{1}(X)$. Therefore, we only need to prove this theorem for any
fixed $n\geq1$. Fix any $\tau_{0}\in \mathbb{H}$ and any $a_{0}\not \in
E_{\tau_{0}}[2]$ being outside the branch loci of $\sigma_{n+1}:\bar{Y}%
_{n+1}(\tau_{0})\rightarrow E_{\tau_{0}}$. Let $X_{0}$ be any zero of
$W_{n+1}(X;a_{0},\tau_{0})$. Our goal is to prove $Q_{n}(X_{0};a_{0},\tau
_{0})=0$.

By Theorem \ref{thm-5A}-(2), there is $\boldsymbol{a}=\{[a_{1}],\cdots,[a_{n+1}]\}
\in \bar{Y}_{n+1}(\tau_{0})$ such that $a_{0}=\sigma_{n+1}(\boldsymbol{a})$ and
$X_{0}=\boldsymbol{z}_{n+1}(\boldsymbol{a})$. Clearly $\boldsymbol{a}$ is not
a branch point of $\bar{Y}_{n+1}(\tau_{0})$ because $a_{0}\notin E_{\tau_0}[2]$.
Recalling (\ref{z--n}), we define
$(r,s)$ by%
\begin{align}
r+s\tau_{0}  &  =\sum_{i=1}^{n+1}a_{i}=a_{0}\notin E_{\tau_0}[2],\label{kkll-1}\\
r\eta_{1}(\tau_{0})+s\eta_{2}(\tau_{0})  &  =\zeta(a_{0}|\tau_{0})-X_{0}%
=\sum_{i=1}^{n+1}\zeta(a_{i}|\tau_{0}).\nonumber
\end{align}
Since $\tau_{0}\eta_{1}(\tau_{0})-\eta_{2}(\tau_{0})=2\pi i$, $(r,s)$ is
uniquely determined and $(r,s)\not \in \frac{1}{2}\mathbb{Z}^{2}$. Clearly
(\ref{kkll-1}) implies%
\begin{equation}
a_{0}=r+s\tau_{0}\text{ \  \ and \  \ }X_{0}=Z_{r,s}(\tau_{0}). \label{k-il}%
\end{equation}
By Proposition \ref{prop-B} and (\ref{iiii}), $y_{\{a_{i}\}}(z)$ and $y_{\{-a_{i}\}}(z)$
are linearly independent solutions to the integral Lam\'{e} equation%
\begin{equation}
y^{\prime \prime}=[  (n+1)(n+2)\wp(z|\tau_{0})+B_{0}]  y, \label{III-67}%
\end{equation}
where $B_{0}=(2n+1)\sum_{i=1}^{n+1}\wp(a_{i}|\tau_{0})$, and the monodromy group of
(\ref{III-67}) with respect to $(y_{\{a_{i}\}}(z),y_{\{-a_{i}
\}}(z))$ is generated by%
\begin{equation}
\rho(\ell_{1})=%
\begin{pmatrix}
e^{-2\pi is} & 0\\
0 & e^{2\pi is}%
\end{pmatrix}\;
\text{ and }\;\rho(\ell_{2})=%
\begin{pmatrix}
e^{2\pi ir} & 0\\
0 & e^{-2\pi ir}%
\end{pmatrix}.
\label{III-66-3}%
\end{equation}

On the other hand, by \cite[Lemma 3.1]{Chen-Kuo-Lin0} we see that
given any $\tilde{h}\in\mathbb{C}$,
 EPVI$(\tfrac{1}{2}(n+\tfrac{1}{2})^{2},\frac{1}{8},\frac{1}{8},\frac{1}{8})$
has a solution $p(\tau)$ such that $p(\tau_{0})=0$ and%
\begin{equation}
p(\tau)=c_{0}(\tau-\tau_{0})^{\frac{1}{2}}(1+\tilde{h}(\tau-\tau_{0}%
)+O(\tau-\tau_{0})^{2})\text{ as }\tau \rightarrow \tau_{0}, \label{515-5}%
\end{equation}
where $c_{0}^{2}=i\frac{2n+1}{2\pi}$ (note $\pm p(\tau)$ corresponds to the
same $\lambda(t)$, so $c_{0}$ is determined up to a sign $\pm$). Consequently,
the corresponding%
\[
\lambda(t)=\frac{\wp(p(\tau)|\tau)-e_{1}(\tau)}{e_{2}(\tau)-e_{1}(\tau)}
\]
is a solution of PVI$(\tfrac{1}{2}(n+\tfrac{1}{2})^{2},\frac{-1}{8},\frac
{1}{8},\frac{3}{8})$ and has a \emph{negative pole} at $t(\tau_{0})$; see
\cite[Lemma 3.1]{Chen-Kuo-Lin0}. Furthermore, it follows from \cite[Theorem 1.5]{Chen-Kuo-Lin0} that the
potential $I(z;\tau)$ of the associated GLE$(n,p(\tau),A(\tau),\tau)$
\begin{align}
y^{\prime \prime}  &  =\left[
\begin{array}
[c]{l}%
n(n+1) \wp(z|\tau) +\frac{3}{4}( \wp( z+p(\tau)|\tau) +\wp( z-p(\tau)|\tau)
)\\
+A(\tau)\left(  \zeta( z+p(\tau)|\tau) -\zeta( z-p(\tau)|\tau) \right)
+B(\tau)
\end{array}
\right]  y\nonumber \\
&  =:I(z;\tau)y \label{60-4}%
\end{align}
which is monodromy preserving as $\tau$ deforms, converges to $(n+1)(n+2)\wp(z|\tau_{0}%
)+\tilde{B}$ uniformly for $z$ bounded away from lattice points
$\Lambda_{\tau_{0}}$ as $\tau \rightarrow \tau_{0}$, where
\[
\tilde{B}:=\lim_{\tau \rightarrow \tau_{0}}B(\tau)=2\pi ic_{0}^{2}\left(
4\pi i\tilde{h}-\eta_{1}(\tau_{0})\right).
\]
Now letting
\[
\tilde{h}=\frac{(2n+1)\eta_{1}(\tau_{0})-B_{0}}{4\pi i(2n+1)}
\]
in (\ref{515-5}), we obtain $\tilde{B}=B_0$, namely the
potential $I(z;\tau)$ of GLE (\ref{60-4}) converges to exactly the potential $(n+1)(n+2)\wp(z|\tau_{0})+B_{0}$
of the Lam\'{e} equation (\ref{III-67})
as $\tau \rightarrow \tau_{0}$. Then we proved in \cite[Theorem 6.2]{Chen-Kuo-Lin} that the monodromy of GLE (\ref{60-4}) is also given by
\[N_1=\begin{pmatrix}
e^{-2\pi is} & 0\\
0 & e^{2\pi is}%
\end{pmatrix},\quad N_2=\begin{pmatrix}
e^{2\pi ir} & 0\\
0 & e^{-2\pi ir}%
\end{pmatrix},\]
the same as the Lam\'{e} equation (\ref{III-67}).
From here and Theorem \ref{thm-II-8}, we obtain $\wp(p(\tau)|\tau)=\wp
(p_{r,s}^{(n)}(\tau)|\tau)$ and so $\lambda(t)=\lambda_{r,s}^{(n)}(t)$. Recalling that
$t(\tau_{0})$ is a negative pole of $\lambda(t)$, we conclude from Theorem
\ref{thm-II-18} that $Q_{n}(Z_{r,s}(\tau_{0});r+s\tau_{0},\tau_{0})=0$. Then
by (\ref{k-il}) we finally obtain $Q_{n}(X_{0};a_{0},\tau_{0})=0$.

Therefore, we have proved that%
\[
W_{n+1}(X;a_{0},\tau_{0})=0\text{ \ }\Rightarrow \text{ \ }Q_{n}(X;a_{0}%
,\tau_{0})=0.
\]
Since $\deg_{X}Q_{n}=\deg_{X}W_{n+1}={(n+1)(n+2)}/{2}$ and Theorem \ref{thm-5A}-(2)
says that $W_{n+1}(X;a_{0},\tau_{0})$ has ${(n+1)(n+2)}/{2}$ distinct
zeros, we conclude that
\begin{equation}
Q_{n}(X;a_{0},\tau_{0})=\mathfrak{q}_{n}(\tau_{0})W_{n+1}(X;a_{0},\tau_{0})
\label{ivi}%
\end{equation}
for any $\tau_{0}$ and $a_{0}\not \in E_{\tau_{0}}[2]$ being outside the
branch loci of $\sigma_{n+1}:\bar{Y}_{n+1}(\tau_{0})\rightarrow E_{\tau_{0}}$.
By $W_{n+1}(X;\sigma_{n+1},\tau)\in \mathbb{Q}[\wp(\sigma_{n+1}|\tau)$,
$\wp^{\prime}(\sigma_{n+1}|\tau)$, $g_{2}(\tau)$, $g_{3}(\tau)][X]$ and the
fact that coefficients of $Q_{n}(X;a(\tau),\tau)$ are rational functions of
$e_{1}(\tau)$, $e_{2}(\tau)$, $e_{3}(\tau)$, $\wp(a(\tau)|\tau)$ and
$\wp^{\prime}(a(\tau)|\tau)$ provided $a(\tau)\not \in E_{\tau}[2]$, it
follows from (\ref{ivi}) and analytic continuation that (\ref{caxi}) holds for
any $\tau$ and $\sigma_{n+1}=a(\tau)\not \in \Lambda_{\tau}$. This completes
the proof.
\end{proof}

A direct consequence of Theorem \ref{thm-5A}-(3), Theorem \ref{q-n=z-n} and Corollary
\ref{thm-II-18-1} is (note from Lemma \ref{coefficient-Qn} that
$\mathfrak{q}_{n}(\tau)\not =0$ for any $\tau \in \mathbb{H}$)

\begin{theorem}
\label{simple-zn}For any $n\geq1$
and $(r,s)\in \mathbb{C}^{2}\backslash \frac{1}{2}\mathbb{Z}^{2}$, there holds
\[
Z_{r,s}^{(n)}(\tau)=\frac{Q_{n-1}(Z_{r,s}(\tau);r+s\tau,\tau)}{\mathfrak{q}%
_{n-1}(\tau)}\text{ \ for any }\tau \in \mathbb{H}.
\]
In particular, for any $(r,s)\in \mathbb{R}^{2}\backslash \frac{1}{2}%
\mathbb{Z}^{2}$, $Z_{r,s}^{(n)}(\tau)$ has only simple zeros as a holomorphic
function of $\tau \in \mathbb{H}$.
\end{theorem}

In general, it is too difficult to write down the explicit formula of
 $Z_{r,s}^{(n)}(\tau)$ and so it seems impossible to prove the
simple zero property directly. Theorem \ref{simple-zn}, which answers an open
question raised in \cite{Dahmen0,LW2}, is a beautiful application of
Painlev\'{e} VI equation to pre-modular forms.

Now we are in the position to prove Theorem \ref{thm-II-18 copy(1)}.

\begin{proof}
[Proof of Theorem \ref{thm-II-18 copy(1)}]Clearly Theorem
\ref{thm-II-18 copy(1)} follows readily from Theorems \ref{thm-II-17},
\ref{thm-II-18} and \ref{simple-zn}.
\end{proof}

\section{Asymptotics of pre-modular forms}

\label{sec-asymptotics}

In this section, we study the asymptotics of $Z_{r,s}^{(n)}(\tau)$ and give the proof of
Theorem \ref{weak} (1) \& (3).

\subsection{The case $\operatorname{Re}s\in(0,\frac{1}{2})$}

First we study the asymptotic behaviors of solution $\lambda
_{r,s}^{(n)}(t)$ in terms of free parameters $(r,s)$ at the branch point
$t=1$. By Theorem \ref{thm-II-8}-(2), we can always assume $\operatorname{Re}%
s\in \lbrack0,\frac{1}{2}]$ for any solution
$\lambda_{r,s}^{(n)}(t)$. In this section we only consider the case
$\operatorname{Re}s\in(0,\frac{1}{2})$. Since $t=t(\tau)=\frac{e_{3}%
(\tau)-e_{1}(\tau)}{e_{2}(\tau)-e_{1}(\tau)}$ maps the fundamental domain%
\begin{equation}
F_{2}:=\{ \tau \in \mathbb{H}\text{ }|\text{ }0\leq \operatorname{Re}%
\tau<2,\text{ }|\tau-1/2|\geq1/2,\text{ }|\tau-3/2|>1/2\} \label{funde}%
\end{equation}
of $\Gamma(2)$ onto $\mathbb{C}\backslash \{0,1\}$, without loss of generality
we may only consider $\tau \in F_{2}$. Then $t\rightarrow1$ if $\tau
\rightarrow \infty$. The first result of this section is

\begin{theorem}
\label{solution-asymptotic}Fix $n\in \mathbb{N\cup \{}0\mathbb{\}}$ and
$(r,s)\in \mathbb{C}^{2}\backslash \frac{1}{2}\mathbb{Z}^{2}$ such that
$\operatorname{Re}s\in(0,\frac{1}{2})$. Then as $F_{2}\ni \tau \rightarrow
\infty$ there holds
\begin{equation}
\lambda_{r,s}^{(n)}(t)=1+C^{(n)}(s)e^{2\pi ir}\left(  \frac{1-t}{16}\right)
^{2s}+o\left(  |1-t|^{2s}\right)  , \label{iv6}%
\end{equation}
where $t=t(\tau)\rightarrow1$ and%
\begin{equation}
C^{(n)}(s)=\left \{
\begin{array}
[c]{l}%
\frac{8s}{2s-1}\prod_{k=0}^{m-1}\frac{(s+k)(s+k+\frac{1}{2})}%
{(s-k-1)(s-k-\frac{3}{2})}\text{ if }n=2m,\\
\frac{8s^{2}}{(2s-1)(s-1)}\prod_{k=0}^{m-1}\frac{(s+k+\frac{1}{2}%
)(s+k+1)}{(s-k-\frac{3}{2})(s-k-2)}\text{ if }n=2m+1.
\end{array}
\right.  \label{iv7}%
\end{equation}
Here we set $\prod_{k=0}^{-1}\ast=1$.
\end{theorem}

\begin{remark}
If we use the notation $(\alpha)_{m}:=\alpha(\alpha+1)\cdots(\alpha+m-1)$,
then (\ref{iv7}) can be written as%
\[
C^{(n)}(s)=\left \{
\begin{array}
[c]{l}%
\frac{8s}{2s-1}\cdot \frac{(s)_{m}(s+\frac{1}{2})_{m}}{(1-s)_{m}(\frac{3}%
{2}-s)_{m}}\text{ if }n=2m,\\
\frac{8s^{2}}{(2s-1)(s-1)}\cdot \frac{(s+\frac{1}{2})_{m}(s+1)_{m}}{(\frac
{3}{2}-s)_{m}(2-s)_{m}}\text{ if }n=2m+1.
\end{array}
\right.
\]
Theorem \ref{solution-asymptotic} for the simplest case $n=0$ was known, see
e. g. \cite[Appendix A]{CKLW}, where the remaining case $s\in \{0,\frac{1}%
{2}\}$ was also discussed. We can compare Theorem \ref{solution-asymptotic}
with Guzzetti's works \cite{Guzzetti,Guzzetti1} for asymptotics of generic
solutions $\lambda(t)$ to Painlev\'{e} VI equation with generic parameters:%
\[
\lambda(t)=1+C(1-t)^{2s}+o\left(  |1-t|^{2s}\right)
\]
as $t\rightarrow1$ along some special paths. See \cite{Guzzetti,Guzzetti1} for
his precise statements. It does not seem that he wrote down the explicit
formula for the coefficient $C$ in terms of the monodromy data $(r,s)$ for our
concerning case PVI$(\frac12(n+\frac12)^2,\frac{-1}{8},\frac18,\frac38)$.
\end{remark}

We prove Theorem \ref{solution-asymptotic} by applying Theorems
\ref{thm-II-17} and \ref{thm-II-18}. To this goal, we have to study the
asymptotic behaviors of $R_{n}(X)-Q_{n-2}(X)Q_{n}(X)$ and $Q_{n}(X)$ with
$X=Z_{r,s}(\tau)$.
To do this, we recall the following $q=e^{2\pi i\tau}$ expansion for
$\wp(z|\tau)$ (see e. g. \cite[p.46]{Lang}): For $|q|<|e^{2\pi iz}|<|q|^{-1}$,
\begin{align}
&  \wp(z|\tau)\label{iv-8}\\
=  &  -\frac{\pi^{2}}{3}-4\pi^{2}\left[  \frac{e^{2\pi iz}}{(1-e^{2\pi
iz})^{2}}+\sum_{n=1}^{\infty}\frac{nq^{n}}{1-q^{n}}\left(  e^{2\pi
inz}+e^{-2\pi inz}-2\right)  \right]  .\nonumber
\end{align}
Now we put
\begin{equation}
z=a(\tau)=r+s\tau \text{ \ and denote }x=e^{2\pi i(r+s\tau)}. \label{iv-9}%
\end{equation}
Since $\operatorname{Re}s\in(0,\frac{1}{2})$ and $\tau \in F_{2}$, we have
$q\rightarrow0$ as $\tau \rightarrow \infty$ and
\[
|q|^{-1}>|x|=e^{-2\pi(\operatorname{Im}r+\operatorname{Im}s\operatorname{Re}%
\tau)}|q|^{\operatorname{Re}s}>|q|^{\frac{1}{2}}\text{\  \ and }|q|^{\frac
{1}{2}}=o(|x|)
\]
for $|\tau|$ large. Hence (\ref{iv-8})-(\ref{iv-9}) give%
\[
\wp(a(\tau)|\tau)=-\tfrac{\pi^{2}}{3}-4\pi^{2}x+o(|x|),
\]%
\[
\wp^{\prime}(a(\tau)|\tau)=-8\pi^{3}ix+o(|x|).
\]
Letting $z=\frac{1}{2},\frac{\tau}{2},\frac{1+\tau}{2}$ in (\ref{iv-8})
respectively, we easily obtain%
\begin{equation}
e_{1}(\tau)=\tfrac{2\pi^{2}}{3}+16\pi^{2}q+O(|q|^{2}), \label{iv-50}%
\end{equation}%
\[
e_{2}(\tau)=-\tfrac{\pi^{2}}{3}-8\pi^{2}q^{\frac{1}{2}}-8\pi^{2}q+O(|q|^{\frac
{3}{2}}),
\]%
\begin{equation}
e_{3}(\tau)=-\tfrac{\pi^{2}}{3}+8\pi^{2}q^{\frac{1}{2}}-8\pi^{2}q+O(|q|^{\frac
{3}{2}}), \label{iv-51}%
\end{equation}
and so%
\begin{equation}
t-1=\frac{e_{3}(\tau)-e_{2}(\tau)}{e_{2}(\tau)-e_{1}(\tau)}=-16q^{\frac{1}{2}%
}+O(|q|)=o(|x|), \label{iv-10}%
\end{equation}%
\begin{equation}
e_{2}(\tau)-e_{1}(\tau)=-\pi^{2}+o(|x|). \label{iv-10-1}%
\end{equation}
In fact, it is easy to obtain the well-known $q$-expansion of $t$:%
\begin{align}
t=  &  1-16q^{\frac{1}{2}}+128q-704q^{\frac{3}{2}}+3072q^{2}-11488q^{\frac
{5}{2}}\label{t-expansion}\\
&  +38400q^{3}-117632q^{\frac{7}{2}}+335872q^{4}+\cdots.\nonumber
\end{align}
On the other hand, it was proved in \cite[(5.3)]{CKLW} that%
\[
Z_{r,s}(\tau)=\pi i(2s-1)-2\pi ix+o(|x|).
\]
These, together with (\ref{II-511})-(\ref{II-512}), easily imply%
\begin{equation}
G_{0}(Z_{r,s}(\tau)|\tau)=8\pi ix+o(|x|), \label{iv-17}%
\end{equation}%
\begin{equation}
Q_{0}(Z_{r,s}(\tau)|\tau)=\pi i(2s-1)-2\pi ix+o(|x|), \label{iv-18}%
\end{equation}%
\begin{align}
& R_{0}(Z_{r,s}(\tau)|\tau)-tQ_{0}(Z_{r,s}(\tau)|\tau) \nonumber\\
=&R_{0}(Z_{r,s}%
(\tau)|\tau)-Q_{0}(Z_{r,s}(\tau)|\tau)+o(|x|)
=8\pi isx+o(|x|). \label{iv-19}%
\end{align}
Starting from these formulae, we want to prove

\begin{lemma}
\label{Q-asymp}Given any $n\geq0$, there holds%
\[
Q_{n}(Z_{r,s}(\tau)|\tau)=\check{Q}_{n}(s)+o(1)\text{ \ as }F_{2}\ni
\tau \rightarrow \infty,
\]
where $\check{Q}_{n}(s)$ is a polynomial of $s$ defined as follows:

\begin{itemize}
\item[(1)] if $n=2m$ with $m\in \mathbb{N}\cup \{0\}$,%
\begin{align}
\check{Q}_{n}(s):=  &  \left[  \pi i(2s-1)\right]  ^{m+1}\prod_{k=0}%
^{m-1}\left[  (s+k)\left(  s+k+\tfrac{1}{2}\right)  \right. \nonumber \\
&  \left.  (s-k-1)\left(  s-k-\tfrac{3}{2}\right)  \right]  ^{m-k};
\label{Q-expansion}%
\end{align}

\item[(2)] if $n=2m+1$ with $m\in \mathbb{N}\cup \{0\}$,%
\begin{align}
\check{Q}_{n}(s):=  &  \left[  \pi i(2s-1)\right]  ^{m+1}s^{m+1}%
(s-1)^{m+1}\prod_{k=0}^{m-1}\left[  \left(  s+k+\tfrac{1}{2}\right)  \right.
\nonumber \\
&  \left.  (s+k+1)\left(  s-k-\tfrac{3}{2}\right)  (s-k-2)\right]  ^{m-k}.
\label{Q-expansion2}%
\end{align}

\end{itemize}
\end{lemma}

\begin{proof}
In this proof we write $Q_{n}=Q_{n}(Z_{r,s}(\tau)|\tau)$, $G_{n}=G_{n}%
(Z_{r,s}(\tau)|\tau)$ and $R_{n}=R_{n}(Z_{r,s}(\tau)|\tau)$ for convenience.
We want to prove%
\begin{equation}
Q_{n}=\check{Q}_{n}+o(1),\text{ \  \ }G_{n}=\check{G}_{n}x+o(|x|),
\label{iv-11}%
\end{equation}%
\begin{equation}
R_{n}-tQ_{n-2}Q_{n}=R_{n}-Q_{n-2}Q_{n}+o(|x|)=\check{R}_{n}x+o(|x|)
\label{iv-12}%
\end{equation}
by induction, where $\check{G}_{n}$ and $\check{R}_{n}$ are polynomials of $s$
defined as follows: if $n=2m$ with $m\in \mathbb{N\cup \{}0\mathbb{\}}$,%
\begin{align}
\check{G}_{n}:=  &  8\pi i\left[  \pi i(2s-1)\right]  ^{3m}\prod_{k=0}%
^{m-1}\left[  (s+k)^{3(m-k)}\left(  s+k+\tfrac{1}{2}\right)  ^{3(m-k)-1}\right.
\nonumber \\
&  \left.  (s-k-1)^{3(m-k)-2}\left(  s-k-\tfrac{3}{2}\right)  ^{3(m-k-1)}%
\right]  , \label{G-expansion}%
\end{align}%
\begin{align}
\check{R}_{n}:=  &  8\pi i\left[  \pi i(2s-1)\right]  ^{2m}s^{2m+1}\left(
s+\tfrac{1}{2}\right)  ^{2m}\nonumber \\
&  \prod_{k=1}^{m-1}\left[  (s+k)\left(  s+k+\tfrac{1}{2}\right)  (s-k)\left(
s-k-\tfrac{1}{2}\right)  \right]  ^{2(m-k)}; \label{R-expansion}%
\end{align}
if $n=2m+1$ with $m\in \mathbb{N}\cup \{0\}$,%
\begin{align}
\check{G}_{n}:=  &  8\pi i\left[  \pi i(2s-1)\right]  ^{3m+1}s^{3m+2}%
(s-1)^{3m}\prod_{k=0}^{m-1}\left[  \left(  s+k+\tfrac{1}{2}\right)
^{3(m-k)}\right. \nonumber \\
&  \left.  (s+k+1)^{3(m-k)-1}\left(  s-k-\tfrac{3}{2}\right)  ^{3(m-k)-2}%
(s-k-2)^{3(m-k-1)}\right]  , \label{G-expansion2}%
\end{align}%
\begin{align}
\check{R}_{n}:=  &  8\pi i\left[  \pi i(2s-1)\right]  ^{2m}s^{2m+3}\left(
s+\tfrac{1}{2}\right)  ^{2m}(s+1)^{2m}(s-1)^{2m}\nonumber \\
\prod_{k=1}^{m-1}  &  \left[  \left(  s+k+\tfrac{1}{2}\right)  (s+k+1)\left(
s-k-\tfrac{1}{2}\right)  (s-k-1)\right]  ^{2(m-k)}. \label{R-expan2}%
\end{align}
Here we set $\prod_{k=i}^{j}\ast=1$ for $j<i$
as before. Remark by the definitions of $\check{Q}_{n}$, $\check{G}_{n}$ and
$\check{R}_{n}$, we have $\check{Q}_{n}\check{G}_{n}\check{R}_{n}\not =0$
because of $\operatorname{Re}s\in(0,\frac{1}{2})$. Furthermore, a direct
computation shows that (set $\check{Q}_{-2}=\check{Q}_{-1}=1$)%
\begin{equation}
\check{R}_{n}\check{Q}_{n-1}=s\check{G}_{n}, \label{iv-13}%
\end{equation}%
\begin{equation}
\check{G}_{n}\check{Q}_{n-3}^{2}=\left(  s+\tfrac{n-1}{2}\right)  ^{2}\check
{G}_{n-1}\check{Q}_{n-2}\check{Q}_{n-1}, \label{iv-14}%
\end{equation}%
\begin{equation}
\check{Q}_{n}\check{Q}_{n-3}=\left(  s+\tfrac{n-1}{2}\right)  \left(
s-\tfrac{n+1}{2}\right)  \check{Q}_{n-2}\check{Q}_{n-1}, \label{iv-15}%
\end{equation}%
\begin{equation}
\check{R}_{n}\check{Q}_{n-3}^{2}=s\left(  s+\tfrac{n-1}{2}\right)  ^{2}%
\check{G}_{n-1}\check{Q}_{n-2}, \label{iv-15-1}%
\end{equation}
hold for any $n\geq1$. We remark that these four identities are the key points
of this proof. By (\ref{iv-17})-(\ref{iv-19}) we see that (\ref{iv-11}%
)-(\ref{iv-12}) hold for $n=0$.\medskip

\textbf{Step 1}. We prove (\ref{iv-11})-(\ref{iv-12}) for $n=1$.

Recalling (\ref{II-515}) that $\varphi_{0}=R_{0}$, we deduce from
(\ref{iv-18})-(\ref{iv-19}) that%
\[
H(\varphi_{0};Q_{0})=-64\pi^{3}i(2s-1)s^{2}x^{2}+o(|x|^{2}),
\]%
\[
H^{\prime}(\varphi_{0};Q_{0})=-16\pi^{2}(2s-1)sx+o(|x|).
\]
Together with (\ref{iv-17}) and (\ref{II-516})-(\ref{II-518}), we easily
obtain%
\[
G_{1}=\frac{H(\varphi_{0};Q_{0})}{G_{0}}=-8\pi^{2}(2s-1)s^{2}x+o(|x|)=\check
{G}_{1}x+o(|x|),
\]%
\[
Q_{1}=\frac{G_{1}-\frac{1}{2}H^{\prime}(\varphi_{0};Q_{0})}{G_{0}}=\pi
i(2s-1)s(s-1)+o(1)=\check{Q}_{1}+o(1),
\]%
\[
R_{1}-Q_{1}=\frac{(R_{0}-Q_{0})Q_{1}+G_{1}}{Q_{0}}=8\pi is^{3}x+o(|x|)=\check
{R}_{1}x+o(|x|),
\]%
\[
R_{1}-tQ_{1}=R_{1}-Q_{1}+o(|x|)=\check{R}_{1}x+o(|x|).
\]
Therefore, (\ref{iv-11})-(\ref{iv-12}) hold for $n=1$.\medskip

\textbf{Step 2}. Assume that (\ref{iv-11})-(\ref{iv-12}) hold for any
$\ell \leq n-1$ where $n\geq2$, we prove (\ref{iv-11})-(\ref{iv-12}) for $n$.

By our assumption that (\ref{iv-11})-(\ref{iv-12}) hold for $n-1$ and
(\ref{iv-13}), we have%
\begin{align*}
(R_{n-1}-Q_{n-3}Q_{n-1})Q_{n-2}  &  =\left(  \check{R}_{n-1}x+o(|x|)\right)
\left(  \check{Q}_{n-2}+o(1)\right) \\
&  =s\check{G}_{n-1}x+o(|x|).
\end{align*}
Recalling (\ref{II-515}) that $\varphi_{n-1}=R_{n-1}Q_{n-2}+\frac{n-1}%
{2}G_{n-1}$, we obtain{\allowdisplaybreaks%
\begin{align*}
&  \varphi_{n-1}-tQ_{n-3}Q_{n-2}Q_{n-1}\\
=  &  \varphi_{n-1}-Q_{n-3}Q_{n-2}Q_{n-1}+o(|x|)\\
=  &  (R_{n-1}-Q_{n-3}Q_{n-1})Q_{n-2}+\tfrac{n-1}{2}G_{n-1}+o(|x|)\\
=  &  \left(  s+\tfrac{n-1}{2}\right)  \check{G}_{n-1}x+o(|x|),
\end{align*}
}and so%
\[
\varphi_{n-1}=Q_{n-3}Q_{n-2}Q_{n-1}+o(1)=\check{Q}_{n-3}\check{Q}_{n-2}%
\check{Q}_{n-1}+o(1).
\]
Applying (\ref{II-516}) and (\ref{iv-14}) leads to{\allowdisplaybreaks%
\begin{align*}
G_{n}  &  =\frac{H(\varphi_{n-1};Q_{n-3}Q_{n-2}Q_{n-1})}{G_{n-1}Q_{n-3}^{3}}\\
&  =\frac{[\check{Q}_{n-3}\check{Q}_{n-2}\check{Q}_{n-1}+o(1)][(s+\frac
{n-1}{2})\check{G}_{n-1}x+o(|x|)]^{2}}{[\check{G}_{n-1}x+o(|x|)][\check
{Q}_{n-3}+o(1)]^{3}}\\
&  =\frac{(s+\frac{n-1}{2})^{2}\check{G}_{n-1}\check{Q}_{n-2}\check{Q}_{n-1}%
}{\check{Q}_{n-3}^{2}}x+o(|x|)=\check{G}_{n}x+o(|x|).
\end{align*}
}Furthermore,%
\begin{align*}
H^{\prime}  &  (\varphi_{n-1};Q_{n-3}Q_{n-2}Q_{n-1})=2\varphi_{n-1}%
(\varphi_{n-1}-Q_{n-3}Q_{n-2}Q_{n-1})+o(|x|)\\
&  =2\left(  s+\tfrac{n-1}{2}\right)  \check{G}_{n-1}\check{Q}_{n-3}\check
{Q}_{n-2}\check{Q}_{n-1}x+o(|x|).
\end{align*}
Then we derive from (\ref{II-517}) and (\ref{iv-15}) that{\allowdisplaybreaks%
\begin{align*}
Q_{n}=  &  \frac{Q_{n-3}G_{n}}{G_{n-1}}-\frac{n}{2}\frac{H^{\prime}%
(\varphi_{n-1};Q_{n-3}Q_{n-2}Q_{n-1})}{G_{n-1}Q_{n-3}^{2}}\\
=  &  \frac{(\check{Q}_{n-3}+o(1))\left(  \frac{(s+\frac{n-1}{2})^{2}\check
{G}_{n-1}\check{Q}_{n-2}\check{Q}_{n-1}}{\check{Q}_{n-3}^{2}}x+o(|x|)\right)
}{\check{G}_{n-1}x+o(|x|)}\\
&  -n\frac{\left(  s+\frac{n-1}{2}\right)  \check{G}_{n-1}\check{Q}%
_{n-3}\check{Q}_{n-2}\check{Q}_{n-1}x+o(|x|)}{(\check{G}_{n-1}x+o(|x|))(\check
{Q}_{n-3}+o(1))^{2}}\\
=  &  \left(  s+\tfrac{n-1}{2}\right)  \left(  s-\tfrac{n+1}{2}\right)
\frac{\check{Q}_{n-2}\check{Q}_{n-1}}{\check{Q}_{n-3}}+o(1)=\check{Q}%
_{n}+o(1).
\end{align*}
}So (\ref{iv-11}) holds for $n$. Finally,{\allowdisplaybreaks%
\begin{align*}
&  \left(  \varphi_{n-1}-Q_{n-3}Q_{n-2}Q_{n-1}\right)  Q_{n}+\tfrac{n+1}%
{2}Q_{n-3}G_{n}\\
=  &  \left[  \left(  s+\tfrac{n-1}{2}\right)  \check{G}_{n-1}x+o(|x|)\right]
\\
&  \times \left[  \left(  s+\tfrac{n-1}{2}\right)  \left(  s-\tfrac{n+1}%
{2}\right)  \frac{\check{Q}_{n-2}\check{Q}_{n-1}}{\check{Q}_{n-3}}+o(1)\right]
\\
&  +\tfrac{n+1}{2}\left[  \check{Q}_{n-3}+o(1)\right]  \left[  \frac
{(s+\frac{n-1}{2})^{2}\check{G}_{n-1}\check{Q}_{n-2}\check{Q}_{n-1}}{\check
{Q}_{n-3}^{2}}x+o(|x|)\right] \\
=  &  s\left(  s+\tfrac{n-1}{2}\right)  ^{2}\frac{\check{G}_{n-1}\check
{Q}_{n-2}\check{Q}_{n-1}}{\check{Q}_{n-3}}x+o(|x|).
\end{align*}
}This, together with (\ref{II-518}) and (\ref{iv-15-1}),
gives{\allowdisplaybreaks%
\begin{align*}
R_{n}-tQ_{n-2}Q_{n}  &  =R_{n}-Q_{n-2}Q_{n}+o(|x|)\\
&  =\frac{\left(  \varphi_{n-1}-Q_{n-3}Q_{n-2}Q_{n-1}\right)  Q_{n}+\frac
{n+1}{2}Q_{n-3}G_{n}}{Q_{n-3}Q_{n-1}}+o(|x|)\\
&  =s\left(  s+\tfrac{n-1}{2}\right)  ^{2}\frac{\check{G}_{n-1}\check{Q}_{n-2}%
}{\check{Q}_{n-3}^{2}}x+o(|x|)\\
&  =\check{R}_{n}x+o(|x|),
\end{align*}
}namely (\ref{iv-12}) holds for $n$.
The proof is complete.
\end{proof}

As an application of Lemma \ref{Q-asymp}, we prove Theorem
\ref{solution-asymptotic}.

\begin{proof}
[Proof of Theorem \ref{solution-asymptotic}]Recall from (\ref{II-535}) in
Theorem \ref{thm-II-18} that%
\[
\lambda_{r,s}^{(n)}(t)-1=\frac{R_{n}(Z_{r,s}(\tau)|\tau)-Q_{n-2}(Z_{r,s}%
(\tau)|\tau)Q_{n}(Z_{r,s}(\tau)|\tau)}{Q_{n-2}(Z_{r,s}(\tau)|\tau
)Q_{n}(Z_{r,s}(\tau)|\tau)}.
\]
Furthermore, (\ref{iv-9}) and (\ref{iv-10}) imply%
\[
x=e^{2\pi ir}q^{s}=e^{2\pi ir}\left(  \frac{1-t}{16}\right)  ^{2s}%
+o(|1-t|^{2s})\text{ as }\tau \rightarrow \infty.
\]
Then Theorem \ref{solution-asymptotic} follows readily from (\ref{iv-11}%
)-(\ref{iv-12}) and the definitions (\ref{Q-expansion}), (\ref{Q-expansion2}),
(\ref{R-expansion}), (\ref{R-expan2}) of $\check{Q}_{n}$, $\check{R}_{n}$.
\end{proof}

To obtain the asymptotics of $Z_{r,s}^{(n)}(\tau)$, we need to calculate the
coefficient $\mathfrak{q}_n(\tau)$ of the leading term
$X^{\frac{(n+1)(n+2)}{2}}$ of $Q_{n}(X)$.

\begin{lemma}
\label{coefficient-Qn}Denote the coefficient of the leading term
$X^{\frac{(n+1)(n+2)}{2}}$ of $Q_{n}(X)$ by $\mathfrak{q}_{n}=\mathfrak{q}%
_{n}(\tau)$. Then for all $n\geq0$, there holds%
\begin{equation}
\mathfrak{q}_{n}(\tau)=\left \{
\begin{array}
[c]{l}%
2^{-\frac{n(n+2)}{2}}(e_{2}(\tau)-e_{1}(\tau))^{-\frac{n(n+2)}{4}}\text{ \ if
}n\text{ even,}\\
2^{-\frac{(n+1)^{2}}{2}}(e_{2}(\tau)-e_{1}(\tau))^{-\frac{(n+1)^{2}}{4}}\text{
\ if }n\text{ odd.}%
\end{array}
\right.  \label{iv-1}%
\end{equation}

\end{lemma}

\begin{proof}
Denote the coefficient of the leading term $X^{\frac{3n(n+1)}{2}}$ of
$G_{n}(X)$ by $\mathfrak{g}_{n}$, and the coefficient of the leading term
$X^{n(n+1)+1}$ of $R_{n}(X)$ by $\mathfrak{r}_{n}$. We prove (\ref{iv-1}) and
the following two formulae%
\begin{equation}
\mathfrak{g}_{n}=\left \{
\begin{array}
[c]{l}%
2^{-\frac{3n^{2}}{2}}(e_{2}(\tau)-e_{1}(\tau))^{-\frac{3n^{2}}{4}-1}%
\wp^{\prime}(a(\tau)|\tau)\text{ \ if }n\text{ even,}\\
2^{-\frac{3n^{2}+1}{2}}(e_{2}(\tau)-e_{1}(\tau))^{-\frac{3n^{2}+1}{4}-1}%
\wp^{\prime}(a(\tau)|\tau)\text{ if }n\text{ odd,}%
\end{array}
\right.  \label{iv-2}%
\end{equation}%
\begin{equation}
\mathfrak{r}_{n}=\left \{
\begin{array}
[c]{l}%
2^{-n^{2}}(e_{2}(\tau)-e_{1}(\tau))^{-\frac{n^{2}}{2}-1}\left(  \wp
(a(\tau)|\tau)-e_{1}(\tau)\right)  \text{ \ if }n\text{ even,}\\
2^{-n^{2}-1}(e_{2}(\tau)-e_{1}(\tau))^{-\frac{n^{2}+1}{2}-1}\left(  \wp
(a(\tau)|\tau)-e_{1}(\tau)\right)  \text{ if }n\text{ odd}%
\end{array}
\right.  \label{iv-3}%
\end{equation}
by induction. Recall $Q_{-1}(X)\equiv1$. Then (\ref{iv-1}) holds for $Q_{-1}%
$.\medskip

\textbf{Step 1}. We consider $n=0,1$. This follows directly from
(\ref{II-511})-(\ref{II-512}) and Theorem \ref{thm-expression-1}.\medskip

\textbf{Step 2.} Assume that (\ref{iv-1})-(\ref{iv-3}) hold for all $0\leq
m\leq n-1$ where $n\geq2$. We prove that (\ref{iv-1})-(\ref{iv-3}) also hold
for $n$. Without loss of generality, we may assume that $n$ is even. The case
that $n$ is odd can be proved similarly and we omit the details here.

Recalling property (4-$(n-1)$), we have%
\begin{align*}
\deg R_{n-1}Q_{n-2}  &  =\deg(R_{n-1}-tQ_{n-3}Q_{n-1})Q_{n-2}\\
&  =\deg(R_{n-1}-Q_{n-3}Q_{n-1})Q_{n-2}\\
&  =\tfrac{3n(n-1)}{2}+1>\deg G_{n-1}.
\end{align*}
Then the leading term of $\varphi_{n-1}=R_{n-1}Q_{n-2}+\frac{n-1}{2}G_{n-1}$
comes from that of $R_{n-1}Q_{n-2}$, so the coefficient of the leading term of
$\varphi_{n-1}$ is (using (\ref{iv-1})-(\ref{iv-3}) and recalling $n$ even)%
\begin{align}
\mathfrak{r}_{n-1}\mathfrak{q}_{n-2}  &  =\frac{\wp(a)-e_{1}}{2^{(n-1)^{2}%
+1}(e_{2}-e_{1})^{\frac{(n-1)^{2}+1}{2}+1}}\cdot \frac{1}{2^{\frac{n(n-2)}{2}%
}(e_{2}-e_{1})^{\frac{n(n-2)}{4}}}\nonumber \\
&  =\frac{\wp(a)-e_{1}}{2^{\frac{3n^{2}-6n}{2}+2}(e_{2}-e_{1})^{\frac
{3n^{2}-6n}{4}+2}}. \label{iv-5}%
\end{align}
Similarly, the coefficient of the leading term of $\varphi_{n-1}%
-Q_{n-3}Q_{n-2}Q_{n-1}=(R_{n-1}-Q_{n-3}Q_{n-1})Q_{n-2}+\frac{n-1}{2}G_{n-1}$
is%
\[
(\mathfrak{r}_{n-1}-\mathfrak{q}_{n-3}\mathfrak{q}_{n-1})\mathfrak{q}%
_{n-2}=\frac{\wp(a)-e_{2}}{2^{\frac{3n^{2}-6n}{2}+2}(e_{2}-e_{1}%
)^{\frac{3n^{2}-6n}{4}+2}},
\]
and the coefficient of the leading term of $\varphi_{n-1}-tQ_{n-3}%
Q_{n-2}Q_{n-1}$ is (using $t=\frac{e_{3}-e_{1}}{e_{2}-e_{1}}$)%
\[
(\mathfrak{r}_{n-1}-t\mathfrak{q}_{n-3}\mathfrak{q}_{n-1})\mathfrak{q}%
_{n-2}=\frac{\wp(a)-e_{3}}{2^{\frac{3n^{2}-6n}{2}+2}(e_{2}-e_{1}%
)^{\frac{3n^{2}-6n}{4}+2}}.
\]
Consequently, we see from (\ref{II-510}) and $(\wp^{\prime})^{2}=4(\wp
-e_{1})(\wp-e_{2})(\wp-e_{3})$ that the coefficient $\mathfrak{h}_{n-1}$ of
the leading term of $H(\varphi_{n-1};Q_{n-3}Q_{n-2}Q_{n-1})$ is%
\[
\mathfrak{h}_{n-1}=\frac{\wp^{\prime}(a)^{2}}{2^{\frac{9(n^{2}-2n)}{2}%
+8}(e_{2}-e_{1})^{\frac{9(n^{2}-2n)}{4}+6}}.
\]
Therefore, (\ref{II-516}) leads to%
\[
\mathfrak{g}_{n}=\frac{\mathfrak{h}_{n-1}}{\mathfrak{g}_{n-1}\mathfrak{q}%
_{n-3}^{3}}=\frac{\wp^{\prime}(a)}{2^{\frac{3n^{2}}{2}}(e_{2}-e_{1}%
)^{\frac{3n^{2}}{4}+1}},
\]
which proves (\ref{iv-2}) for $n$. Recalling from (\ref{II-530-2}%
)-(\ref{II-530-3}) that%
\[
\deg H(\varphi_{n-1};Q_{n-3}Q_{n-2}Q_{n-1})>\deg G_{n-1}H^{\prime}%
(\varphi_{n-1};Q_{n-3}Q_{n-2}Q_{n-1}),
\]
we see from (\ref{II-517}) that{%
\[
\mathfrak{q}_{n}=\frac{\mathfrak{h}_{n-1}}{\mathfrak{g}_{n-1}^{2}%
\mathfrak{q}_{n-3}^{2}}=\frac{1}{2^{\frac{n(n+2)}{2}}(e_{2}-e_{1}%
)^{\frac{n(n+2)}{4}}},
\]
which proves (\ref{iv-1}) for }$n$. Finally, we see from%
\[
\deg \varphi_{n-1}Q_{n}=2n^{2}+2>\deg Q_{n-3}G_{n}%
\]
and (\ref{II-518}), (\ref{iv-5}) that%
\[
\mathfrak{r}_{n}=\frac{\mathfrak{r}_{n-1}\mathfrak{q}_{n-2}\cdot
\mathfrak{q}_{n}}{\mathfrak{q}_{n-3}\mathfrak{q}_{n-1}}=\frac{\wp(a)-e_{1}%
}{2^{n^{2}}(e_{2}-e_{1})^{\frac{n^{2}}{2}+1}},
\]
which proves (\ref{iv-3}) for $n$.
The proof is complete.
\end{proof}

Now we can prove Theorem \ref{weak}-(1) as stated in the following theorem.
We set $\prod_{k=0}^{-1}\ast=1$ for convenience.

\begin{theorem}[=Theorem \ref{weak}-(1)] \label{asymp-Zn copy(1)}Given any $n\geq1$ and
$(r,s)\in \mathbb{C}^{2}\backslash \frac{1}{2}\mathbb{Z}^{2}$ such that
$\operatorname{Re}s\in(0,\frac{1}{2})$, there holds%
\[
Z_{r,s}^{(n)}(\tau)=\check{Z}^{(n)}(s)+o(1)\text{ as }F_{2}\ni \tau
\rightarrow \infty,
\]
where $\check{Z}^{(n)}(s)$ is a polynomial of $s$ given as follows:

\begin{itemize}
\item[(1)] if $n=2m+1$ with $m\in \mathbb{N}\cup \{0\}$,%
\begin{align*}
\check{Z}^{(n)}(s) =  &  (2\pi)^{2m(m+1)}\left[  \pi i(2s-1)\right]
^{m+1}\prod_{k=0}^{m-1}\left[  (s+k)\left(  s+k+\tfrac{1}{2}\right)  \right. \\
&  \left.  (s-k-1)\left(  s-k-\tfrac{3}{2}\right)  \right]  ^{m-k};
\end{align*}

\item[(2)] if $n=2m$ with $m\in \mathbb{N}$,%
\begin{align*}
\check{Z}^{(n)}(s) =  &  (-1)^{m^{2}}(2\pi)^{2m^{2}}\left[  \pi
i(2s-1)\right]  ^{m}s^{m}(s-1)^{m}\prod_{k=0}^{m-2}\left[  \left(
s+k+\tfrac{1}{2}\right)  \right. \\
&  \left.  (s+k+1)\left(  s-k-\tfrac{3}{2}\right)  (s-k-2)\right]  ^{m-1-k}.
\end{align*}

\end{itemize}

\noindent In particular,
\begin{equation}
\lim_{F_{2}\ni \tau \rightarrow \infty}Z_{r,s}^{(n)}(\tau)\not =0\text{ as long
as }\operatorname{Re}s\in (0,1/2)  \cup(1/2,1). \label{n1on-zero}%
\end{equation}

\end{theorem}

\begin{proof}
Since $e_{2}(\tau)-e_{1}(\tau)=-\pi^{2}+o(1)$ as
$F_{2}\ni \tau \rightarrow \infty$, Theorem \ref{asymp-Zn copy(1)} is a direct consequence of Lemmas
\ref{Q-asymp}-\ref{coefficient-Qn}, Theorem \ref{simple-zn} and (\ref{iv-20}).
\end{proof}

\subsection{The case $s=0$}

In this section, we give the proof of Theorem
\ref{weak}-(3).
When
$F_{2}\ni \tau \rightarrow \infty$, we use notation%
\[
f(\tau)\sim q^{a}\text{ to denote }f(\tau)=Cq^{a}+o(|q|^{a})
\]
with some constant $C\not =0$.

\begin{proof}[Proof of Theorem \ref{weak}-(3)]
By (\ref{iv-39-1}) it suffices for us to prove $\tilde{a}_{n}=a_{n}$ for all
$n$, which holds true already for $n\leq4$. For general $n$, Theorem \ref{modular-zero} shows that we can determine $\tilde{a}_{n}$ by computing the
asymptotics of $Z_{\frac{1}{4},0}^{(n)}(\tau)$ as $F_{2}\ni
\tau \rightarrow \infty$. Therefore in this proof, we always consider
$(r,s)=(\frac{1}{4},0)$. We also write $Q_{n}=Q_{n}(Z_{\frac{1}{4},0}%
(\tau)|\tau)$, $G_{n}=G_{n}(Z_{\frac{1}{4},0}(\tau)|\tau)$, $R_{n}%
=R_{n}(Z_{\frac{1}{4},0}(\tau)|\tau)$, $\varphi_{n}=\varphi_{n}(Z_{\frac{1}%
{4},0}(\tau)|\tau)$, $\lambda^{(n)}=\lambda_{\frac{1}{4},0}^{(n)}(t(\tau))$
and $\mu^{(n)}=\mu_{\frac{1}{4},0}^{(n)}(t(\tau))$ for convenience. We want to
prove for any $n\geq0$ that%
\begin{equation}
G_{n}\sim q^{\frac{n(n+1)}{2}-a_{n}},\text{ }\varphi_{n}\sim q^{\frac
{n(n+1)}{2}-a_{n}}\text{ and }Q_{n}\sim q^{a_{n+1}} \label{iv-45}%
\end{equation}
as $F_{2}\ni \tau \rightarrow \infty$, where
\begin{equation}
a_{2k}=a_{2k+1}=k(k+1)/2\text{ \ for all }k\in \mathbb{N\cup \{}0\mathbb{\}}.
\label{iv-45-1}%
\end{equation}
Once (\ref{iv-45}) is proved, we immediately obtain from Theorem
\ref{simple-zn} that%
\[
Z_{\frac{1}{4},0}^{(n)}(\tau)=\frac{Q_{n-1}}{\mathfrak{q}_{n-1}}\sim q^{a_{n}%
}\text{ \ for all }n\geq1
\]
and so $\tilde{a}_{n}=a_{n}$ for all $n$.

Now we turn to prove (\ref{iv-45}) by induction.\medskip

\textbf{Step 1.} We prove (\ref{iv-45}) for $n=0$.

Recall from (\ref{iv-35-2})-(\ref{iv-35}) that%
\[
\wp^{\prime}\left(  \tfrac{1}{4}|\tau \right)  =-4\pi^{3}+16\pi^{3}q+16\pi
^{3}q^{2}+O(|q|^{3}),
\]%
\[
\wp \left(  \tfrac{1}{4}|\tau \right)  =\frac{5}{3}\pi^{2}+8\pi^{2}q+40\pi
^{2}q^{2}+O(|q|^{3}),
\]%
\[
Q_{0}=Z_{\frac{1}{4},0}(\tau)=\pi+4\pi q+4\pi q^{2}+O(|q|^{3}).
\]
Together with (\ref{iv-50})-(\ref{iv-51}), (\ref{II-46})-(\ref{II-512}) and
(\ref{II-515}), a straightforward computation leads to%
\[
G_{0}=4\pi-32\pi q^{\frac{1}{2}}+O(|q|),
\]%
\[
\varphi_{0}=R_{0}=\pi-8\pi q^{\frac{1}{2}}+O(|q|),
\]%
\begin{equation}
\lambda_{0}=1-8q^{\frac{1}{2}}+O(|q|). \label{iv-53}%
\end{equation}
In particular, (\ref{iv-45}) holds for $n=0$.\medskip

\textbf{Step 2.} We claim that for any $n\geq0$, there hold%
\begin{equation}
R_{n}Q_{n-1}=\frac{1}{4}G_{n}, \label{iv-56}%
\end{equation}%
\begin{align}
&  Q_{n-2}Q_{n-1}Q_{n}\label{iv-65-1}\\
=  &  (-1)^{n}\frac{2n+1}{4}G_{n}\left(  1+8q^{\frac{1}{2}}+32q+96q^{\frac
{3}{2}}+O(|q|^{2})\right)  ,\nonumber
\end{align}%
\begin{equation}
\varphi_{n}=R_{n}Q_{n-1}+\frac{n}{2}G_{n}=\frac{2n+1}{4}G_{n}. \label{iv-57}%
\end{equation}
Remark that (\ref{iv-56})-(\ref{iv-57}) are the key points of this proof.

Indeed, recall Theorem \ref{r=1/4} that%
\begin{equation}
\lambda^{(n)}=\frac{(-1)^{n}}{2n+1}t^{\frac{1}{2}},\  \mu^{(n)}=\frac
{1}{4\lambda^{(n)}}=(-1)^{n}\frac{2n+1}{4}t^{-\frac{1}{2}}. \label{iv-54}%
\end{equation}
It follows from (\ref{t-expansion}) that%
\begin{equation}
\mu^{(n)}=(-1)^{n}\frac{2n+1}{4}\left(  1+8q^{\frac{1}{2}}+32q+96q^{\frac
{3}{2}}+O(|q|^{2})\right)  . \label{iv-55}%
\end{equation}
Since Theorem \ref{thm-II-18} gives%
\[
\lambda^{(n)}=\frac{R_{n}}{Q_{n-2}Q_{n}},\text{ }\mu^{(n)}=\frac
{Q_{n-2}Q_{n-1}Q_{n}}{G_{n}},
\]
it follows from (\ref{iv-54}) and (\ref{iv-55}) that (\ref{iv-56}) and
(\ref{iv-65-1}) hold. Finally, (\ref{II-515}) and (\ref{iv-56}) yield
(\ref{iv-57}).\medskip

\textbf{Step 3.} Assume that (\ref{iv-45}) holds for any $k\leq n-1$ for some
$n\geq1$, we prove (\ref{iv-45}) for $n$. We consider two cases separately.

\textbf{Case 1}. $n=2m$ for some $m\in \mathbb{N}$.

Then (\ref{t-expansion}) and (\ref{iv-65-1})-(\ref{iv-57}) with $n-1$ imply%
\[
\varphi_{n-1}-Q_{n-3}Q_{n-2}Q_{n-1}=\frac{2n-1}{2}G_{n-1}\left(
1+O(|q|^{\frac{1}{2}})\right)  ,
\]%
\[
\varphi_{n-1}-tQ_{n-3}Q_{n-2}Q_{n-1}=\frac{2n-1}{2}G_{n-1}\left(
1+O(|q|^{\frac{1}{2}})\right)  .
\]
It follows from (\ref{II-510}) that%
\[
H(\varphi_{n-1};Q_{n-3}Q_{n-2}Q_{n-1})=\frac{(2n-1)^{3}}{16}G_{n-1}^{3}\left(
1+O(|q|^{\frac{1}{2}})\right)  ,
\]%
\[
H^{\prime}(\varphi_{n-1};Q_{n-3}Q_{n-2}Q_{n-1})=\frac{(2n-1)^{2}}{2}%
G_{n-1}^{2}\left(  1+O(|q|^{\frac{1}{2}})\right)  .
\]
Therefore, we derive from (\ref{II-516})-(\ref{II-517}) and (\ref{iv-45}%
)-(\ref{iv-45-1}) that{\allowdisplaybreaks%
\begin{align*}
G_{n}  &  =\frac{H(\varphi_{n-1};Q_{n-3}Q_{n-2}Q_{n-1})}{G_{n-1}Q_{n-3}^{3}%
}=\frac{\frac{(2n-1)^{3}}{16}G_{n-1}^{2}(  1+O(|q|^{\frac{1}{2}}))
}{Q_{n-3}^{3}}\\
&  \sim \frac{q^{n(n-1)-2a_{n-1}}}{q^{3a_{n-2}}}=q^{\frac{3m^{2}+m}{2}%
}=q^{\frac{n(n+1)}{2}-a_{n}},
\end{align*}
}{\allowdisplaybreaks%
\begin{align*}
Q_{n}  &  =\frac{Q_{n-3}^{3}G_{n}-\frac{n}{2}H^{\prime}(\varphi_{n-1}%
;Q_{n-3}Q_{n-2}Q_{n-1})}{G_{n-1}Q_{n-3}^{2}}\\
&  =\frac{\frac{(2n-1)^{3}}{16}G_{n-1}^{2}(1+O(|q|^{\frac{1}{2}
}))  -\frac{n(2n-1)^{2}}{4}G_{n-1}^{2}(  1+O(|q|^{\frac{1}{2}
}))  }{G_{n-1}Q_{n-3}^{2}}\\
&  =-\frac{(2n+1)(2n-1)^{2}G_{n-1}(  1+O(|q|^{\frac{1}{2}}))
}{16Q_{n-3}^{2}}\\
&  \sim \frac{q^{\frac{n(n-1)}{2}-a_{n-1}}}{q^{2a_{n-2}}}=q^{\frac{m(m+1)}{2}%
}=q^{a_{n+1}},
\end{align*}
and so }$\varphi_{n}\sim q^{\frac{n(n+1)}{2}-a_{n}}${ by (\ref{iv-57}). }This
proves (\ref{iv-45}) for even $n$.

\textbf{Case 2}. $n=2m+1$ for some $m\in \mathbb{N\cup \{}0\mathbb{\}}$.
Differently from Case 1, there are \emph{cancellations} in the following computations.

Again by (\ref{t-expansion}) and (\ref{iv-65-1})-(\ref{iv-57}) with $n-1$, we
have{\allowdisplaybreaks%
\begin{align*}
&  \varphi_{n-1}-Q_{n-3}Q_{n-2}Q_{n-1}\\
=  &  \frac{2n-1}{4}G_{n-1}-\frac{2n-1}{4}G_{n-1}\left(  1+8q^{\frac{1}{2}%
}+32q+96q^{\frac{3}{2}}+O(|q|^{2})\right) \\
=  &  -2(2n-1)q^{\frac{1}{2}}G_{n-1}\left(  1+4q^{\frac{1}{2}}%
+12q+O(|q|^{\frac{3}{2}})\right)  ,
\end{align*}
}{\allowdisplaybreaks%
\begin{align*}
&  \varphi_{n-1}-tQ_{n-3}Q_{n-2}Q_{n-1}\\
=  &  \frac{2n-1}{4}G_{n-1}-\left(  1-16q^{\frac{1}{2}}+128q-704q^{\frac{3}%
{2}}+O(|q|^{2})\right) \\
&  \times \frac{2n-1}{4}G_{n-1}\left(  1+8q^{\frac{1}{2}}+32q+96q^{\frac{3}{2}%
}+O(|q|^{2})\right) \\
=  &  2(2n-1)q^{\frac{1}{2}}G_{n-1}\left(  1-4q^{\frac{1}{2}}+12q+O(|q|^{\frac
{3}{2}})\right)  .
\end{align*}
}Consequently,%
\[
H(\varphi_{n-1};Q_{n-3}Q_{n-2}Q_{n-1})=-(2n-1)^{3}qG_{n-1}^{3}\left(
1+O(|q|)\right)  ,
\]
{\allowdisplaybreaks%
\begin{align*}
&  H^{\prime}(\varphi_{n-1};Q_{n-3}Q_{n-2}Q_{n-1})\\
=  &  \varphi_{n-1}(\varphi_{n-1}-Q_{n-3}Q_{n-2}Q_{n-1}+\varphi_{n-1}%
-tQ_{n-3}Q_{n-2}Q_{n-1})\\
&  +(\varphi_{n-1}-Q_{n-3}Q_{n-2}Q_{n-1})(\varphi_{n-1}-tQ_{n-3}Q_{n-2}%
Q_{n-1})\\
=  &  -8(2n-1)^{2}qG_{n-1}^{2}(1+O(|q|)).
\end{align*}
}Therefore,{\allowdisplaybreaks%
\begin{align*}
G_{n}  &  =\frac{H(\varphi_{n-1};Q_{n-3}Q_{n-2}Q_{n-1})}{G_{n-1}Q_{n-3}^{3}%
}=\frac{-(2n-1)^{3}qG_{n-1}^{2}\left(  1+O(|q|)\right)  }{Q_{n-3}^{3}}\\
&  \sim \frac{q^{1+n(n-1)-2a_{n-1}}}{q^{3a_{n-2}}}=q^{\frac{3m^{2}+5m+2}{2}%
}=q^{\frac{n(n+1)}{2}-a_{n}},
\end{align*}
}{\allowdisplaybreaks%
\begin{align*}
Q_{n}  &  =\frac{Q_{n-3}^{3}G_{n}-\frac{n}{2}H^{\prime}(\varphi_{n-1}%
;Q_{n-3}Q_{n-2}Q_{n-1})}{G_{n-1}Q_{n-3}^{2}}\\
&  =\frac{-(2n-1)^{3}qG_{n-1}^{2}\left(  1+O(|q|)\right)  +4n(2n-1)^{2}%
qG_{n-1}^{2}\left(  1+O(|q|)\right)  }{G_{n-1}Q_{n-3}^{2}}\\
&  =\frac{(2n+1)(2n-1)^{2}qG_{n-1}\left(  1+O(|q|)\right)  }{Q_{n-3}^{2}}\\
&  \sim \frac{q^{1+\frac{n(n-1)}{2}-a_{n-1}}}{q^{2a_{n-2}}}=q^{\frac
{(m+1)(m+2)}{2}}=q^{a_{n+1}},
\end{align*}
and so }$\varphi_{n}\sim q^{\frac{n(n+1)}{2}-a_{n}}${ by (\ref{iv-57}). }This
proves (\ref{iv-45}) for odd $n$.

The proof is complete.
\end{proof}

\begin{remark}
During the proof of Theorem \ref{weak}-(3), we can actually compute explicitly
the coefficient of the leading term $q^{a_{n+1}}$ of $Q_{n}(Z_{\frac{1}{4}%
,0}(\tau)|\tau)$ by induction. We leave the details to the interested reader. In
particular, applying Lemma \ref{coefficient-Qn} and Theorem \ref{simple-zn} one can
obtain the explicit asymptotics
\[
Z_{\frac{1}{4},0}^{(n)}(\tau)=(-1)^{C_{n}}16^{a_{n}}\pi^{\frac{n(n+1)}{2}%
}\prod_{k=1}^{n-1}(2k+1)^{n-k}q^{a_{n}}+O(|q|^{a_{n}+1}),
\]
where
\[
C_{n}=\left \{
\begin{array}
[c]{l}%
\frac{n^{2}-1}{4}\text{ \ if \ }n\text{ is odd,}\\
\frac{n^{2}}{4}\text{ \ if \ }n\text{ is even.}%
\end{array}
\right.
\]
The interesting thing is that \emph{every odd} positive integer will appear in
the coefficient of the leading term $q^{a_{n}}$ as $n$ increases. This
phenomenon is mysterious to us; it indicates that the pre-modular form $Z_{r,s}^{(n)}(\tau)$ should possess
many interesting unknown properties that are worthy to study in a future work.
\end{remark}

\appendix

\section{Proof of Theorem \ref{thm-II-17}}

\label{appendix-A}

\begin{proof}
[Proof of Theorem \ref{thm-II-17}]We prove this theorem by induction. By using
(\ref{II-500-1}) and (\ref{II-501-1}), the proof of (\ref{II-514}%
)-(\ref{II-514-1}) and (\ref{II-515})-(\ref{II-518}) is simple. The difficult
part is to prove the validity of properties (1-$n$)-(4-$n$).\medskip

\textbf{Step 1.} we consider $n=1$.

By (\ref{II-513}), (\ref{II-510}) and (\ref{II-500-1}), we have%
\begin{align}
p_{1}  &  =p_{0}-\frac{1}{2}\left(  \frac{1}{q_{0}}+\frac{1}{q_{0}-1}+\frac
{1}{q_{0}-t}\right) \nonumber \\
&  =Q_{0}\frac{H(R_{0};Q_{0})-\frac{1}{2}G_{0}H^{\prime}(R_{0};Q_{0})}%
{G_{0}H(R_{0};Q_{0})}=\frac{Q_{0}Q_{1}}{G_{1}}, \label{II-521}%
\end{align}
where%
\begin{align}
G_{1}  &  =\frac{H(R_{0};Q_{0})}{G_{0}},\label{II-519}\\
Q_{1}  &  =\frac{H(R_{0};Q_{0})-\frac{1}{2}G_{0}H^{\prime}(R_{0};Q_{0})}%
{G_{0}^{2}}=\frac{G_{1}-\frac{1}{2}H^{\prime}(R_{0};Q_{0})}{G_{0}}.
\label{II-520}%
\end{align}
Since $G_{0}(X)$ is a non-zero constant, by properties (1-$0$), (4-$0$) and
(\ref{II-510}), it is easy to see that $G_{1}$ and $Q_{1}$ are both
polynomials of degree $3$.

If $Q_{0}(X)=0$, then property (3-$0$) gives $R_{0}(X)=\frac{1}{2}%
G_{0}(X)\not =0$, which implies%
\[
H(R_{0}(X);Q_{0}(X))=\frac{1}{8}G_{0}(X)^{3},\text{ \ }H^{\prime}%
(R_{0}(X);Q_{0}(X))=\frac{3}{4}G_{0}(X)^{2}.
\]
From here, it is easy to see that
\begin{equation}
G_{1}(X)=\frac{1}{8}G_{0}(X)^{2}\not =0,\text{ \ }Q_{1}(X)=-\frac{1}{4}%
G_{0}(X)\not =0. \label{II-522}%
\end{equation}

If both $G_{1}(X)=0$ and $Q_{1}(X)=0$, then (\ref{II-519})-(\ref{II-520}) give%
\[
H(R_{0}(X);Q_{0}(X))=H^{\prime}(R_{0}(X);Q_{0}(X))=0,
\]
which implies $R_{0}(X)=Q_{0}(X)=0$, a contradiction with property (2-$0$).
Thus we have proved that any two of $\{Q_{-1},Q_{0},Q_{1},G_{1}\}$ have no
common zeros. Clearly, $G_{1}$ and $G_{0}$ have no common zeros.

By (\ref{II-513}), (\ref{II-501-1}) and (\ref{II-521}) we have%
\[
q_{1}=q_{0}+\frac{1}{p_{1}}=\frac{R_{0}Q_{1}+G_{1}}{Q_{0}Q_{1}}=\frac{R_{1}%
}{Q_{1}},
\]
where%
\begin{equation}
R_{1}=\frac{1}{Q_{0}}(R_{0}Q_{1}+G_{1}). \label{II-523}%
\end{equation}
Recall $Q_{0}(X)=X$. If $Q_{0}(X)=0$, then $R_{0}(X)=\frac{1}{2}G_{0}(X)$ and
(\ref{II-522}) gives $R_{0}(X)Q_{1}(X)+G_{1}(X)=0$. Thus, $R_{1}$ is a
polynomial of degree $3$, namely property (1-$1$) holds.

If $Q_{1}(X)=0$, then $Q_{0}(X)\not =0$ and $G_{1}(X)\not =0$, which implies
$R_{1}(X)=\frac{G_{1}(X)}{Q_{0}(X)}\not =0$. Thus, any two of $\{Q_{-1}%
,Q_{1},R_{1}\}$ have no common zeros, namely property (2-$1$) holds.

Define $\varphi_{1}=R_{1}Q_{0}+\frac{1}{2}G_{1}$ by (\ref{II-515}). Then
property (3-$1$) follows directly from $Q_{-1}=1$ and (\ref{II-523}).

Finally, by property (4-$0$) and%
\[
R_{1}-tQ_{1}=\frac{(R_{0}-tQ_{0})Q_{1}+G_{1}}{Q_{0}},
\]
it is easy to see that property (4-$1$) holds. This completes the proof for
$n=1$.\medskip

\textbf{Step 2.} Assume that this theorem holds for $(p_{m},q_{m})$, where
$1\leq m$ $\leq n-1$ and $n\geq2$, we prove that it also holds for
$(p_{n},q_{n})$.

By (\ref{II-500-1}) and (\ref{II-515}) we have{\allowdisplaybreaks%
\begin{align*}
p_{n}=  &  \frac{Q_{n-3}Q_{n-2}Q_{n-1}}{G_{n-1}}-\frac{n}{2}\Bigg(\frac
{1}{\frac{R_{n-1}}{Q_{n-3}Q_{n-1}}+\frac{(n-1)G_{n-1}}{2Q_{n-3}Q_{n-2}Q_{n-1}%
}-1}\\
&  +\frac{1}{\frac{R_{n-1}}{Q_{n-3}Q_{n-1}}+\frac{(n-1)G_{n-1}}{2Q_{n-3}%
Q_{n-2}Q_{n-1}}}+\frac{1}{\frac{R_{n-1}}{Q_{n-3}Q_{n-1}}+\frac{(n-1)G_{n-1}%
}{2Q_{n-3}Q_{n-2}Q_{n-1}}-t}\Bigg)\\
=  &  \frac{Q_{n-3}Q_{n-2}Q_{n-1}}{G_{n-1}H(\varphi_{n-1};Q_{n-3}%
Q_{n-2}Q_{n-1})}\times \\
&  \Big(H(\varphi_{n-1};Q_{n-3}Q_{n-2}Q_{n-1})-\frac{n}{2}G_{n-1}H^{\prime
}(\varphi_{n-1};Q_{n-3}Q_{n-2}Q_{n-1})\Big)\\
=  &  \frac{Q_{n-2}Q_{n-1}Q_{n}}{G_{n}},
\end{align*}
}where $G_{n}$ and $Q_{n}$ are given by (\ref{II-516}) and (\ref{II-517})
respectively. This proves (\ref{II-514}). By (\ref{II-501-1}) we
have{\allowdisplaybreaks%
\[
q_{n}=\frac{R_{n-1}}{Q_{n-3}Q_{n-1}}+\frac{(n-1)G_{n-1}}{2Q_{n-3}%
Q_{n-2}Q_{n-1}}+\frac{(n+1)G_{n}}{2Q_{n-2}Q_{n-1}Q_{n}}=\frac{R_{n}}%
{Q_{n-2}Q_{n}},
\]
}where $R_{n}$ is given by (\ref{II-518}). This proves (\ref{II-514-1}).

Now we need to prove properties (1-$n$)-(4-$n$). Since the proof is long, we
divide it into several steps.\medskip

\textbf{Step 2-1.} We prove that $G_{n}$ and $Q_{n}$ are polynomials.

Since properties (2-$(n-1)$)-(3-$(n-1)$) show that $G_{n-1}$ and $Q_{n-3}$
have no common zeros and $Q_{n-3}|\varphi_{n-1}$, it suffices to prove%
\begin{equation}
G_{n-1}|H(\varphi_{n-1};Q_{n-3}Q_{n-2}Q_{n-1}),\text{ \  \ }G_{n-1}|H_{n},
\label{II-526}%
\end{equation}
where%
\begin{equation}
H_{n}:=\frac{H(\varphi_{n-1};Q_{n-3}Q_{n-2}Q_{n-1})}{G_{n-1}}-\frac{n}%
{2}H^{\prime}(\varphi_{n-1};Q_{n-3}Q_{n-2}Q_{n-1}). \label{II-526-1}%
\end{equation}

By the formulae (\ref{II-516})-(\ref{II-518}) for $n-1$, we have%
\begin{equation}
G_{n-2}G_{n-1}=\frac{H(\varphi_{n-2};Q_{n-4}Q_{n-3}Q_{n-2})}{Q_{n-4}^{3}%
}=H\Big(\frac{\varphi_{n-2}}{Q_{n-4}};Q_{n-3}Q_{n-2}\Big), \label{II-524-0}%
\end{equation}%
\begin{equation}
G_{n-2}Q_{n-1}=Q_{n-4}G_{n-1}-\frac{n-1}{2}H^{\prime}\Big(\frac{\varphi_{n-2}%
}{Q_{n-4}};Q_{n-3}Q_{n-2}\Big), \label{II-524-1}%
\end{equation}%
\begin{equation}
R_{n-1}Q_{n-2}=\frac{\varphi_{n-2}}{Q_{n-4}}Q_{n-1}+\frac{n}{2}G_{n-1}.
\label{II-524}%
\end{equation}
Notice that $\frac{\varphi_{n-2}}{Q_{n-4}}$ is a polynomial by property
(3-$(n-2)$). Substituting (\ref{II-524}) into (\ref{II-515}) leads to%
\begin{align}
\varphi_{n-1}  &  =\frac{\varphi_{n-2}}{Q_{n-4}}Q_{n-1}+\frac{2n-1}{2}%
G_{n-1}\label{II-525}\\
&  \equiv \frac{\varphi_{n-2}}{Q_{n-4}}Q_{n-1}\text{ \  \ }\operatorname{mod}%
G_{n-1}.\nonumber
\end{align}
Here for polynomials $f_{j}(X)$, the notation $f_{1}\equiv f_{2}%
\operatorname{mod}f_{3}$ means $f_{3}|(f_{1}-f_{2})$. Therefore, we derive
from (\ref{II-510}), (\ref{II-524-0}) and (\ref{II-525})
that{\allowdisplaybreaks%
\begin{align}
&  H(\varphi_{n-1};Q_{n-3}Q_{n-2}Q_{n-1})\label{II-527}\\
\equiv &  \frac{2n-1}{2}G_{n-1}H^{\prime}\Big(\frac{\varphi_{n-2}}{Q_{n-4}%
}Q_{n-1};Q_{n-3}Q_{n-2}Q_{n-1}\Big)\nonumber \\
&  +H\Big(\frac{\varphi_{n-2}}{Q_{n-4}}Q_{n-1};Q_{n-3}Q_{n-2}Q_{n-1}%
\Big)\text{ \  \  \ }\operatorname{mod}G_{n-1}^{2}\nonumber \\
=  &  \left(  \frac{2n-1}{2}H^{\prime}\Big(\frac{\varphi_{n-2}}{Q_{n-4}%
};Q_{n-3}Q_{n-2}\Big)+G_{n-2}Q_{n-1}\right)  Q_{n-1}^{2}G_{n-1},\nonumber
\end{align}
}which proves the first assertion in (\ref{II-526}). Furthermore, by
(\ref{II-526-1}), (\ref{II-527}), (\ref{II-525}) and (\ref{II-524-1}), we
have{\allowdisplaybreaks%
\begin{align*}
H_{n}\equiv &  \left(  \frac{2n-1}{2}H^{\prime}\Big(\frac{\varphi_{n-2}%
}{Q_{n-4}};Q_{n-3}Q_{n-2}\Big)+G_{n-2}Q_{n-1}\right)  Q_{n-1}^{2}\\
&  -\frac{n}{2}H^{\prime}\Big(\frac{\varphi_{n-2}}{Q_{n-4}}Q_{n-1}%
;Q_{n-3}Q_{n-2}Q_{n-1}\Big)\text{ \  \ }\operatorname{mod}G_{n-1}\\
=  &  \left(  \frac{n-1}{2}H^{\prime}\Big(\frac{\varphi_{n-2}}{Q_{n-4}%
};Q_{n-3}Q_{n-2}\Big)+G_{n-2}Q_{n-1}\right)  Q_{n-1}^{2}\\
=  &  Q_{n-4}Q_{n-1}^{2}G_{n-1},
\end{align*}
}namely $G_{n-1}|H_{n}$. This proves (\ref{II-526}).\medskip

\textbf{Step 2-2.} We prove that $R_{n}$ is a polynomial.

Since properties (2-$(n-1)$)-(3-$(n-1)$) show that $Q_{n-3}|\varphi_{n-1}$ and
$Q_{n-3}$ has no common zeros with $Q_{n-1}$, it suffices to prove%
\begin{equation}
Q_{n-1}\left \vert \varphi_{n-1}Q_{n}+\frac{n+1}{2}Q_{n-3}G_{n}.\right.
\label{II-528}%
\end{equation}

Recall from (\ref{II-525}) that%
\[
\varphi_{n-1}\equiv \frac{2n-1}{2}G_{n-1}\text{ }\operatorname{mod}Q_{n-1}.
\]
This, together with the formulae (\ref{II-516})-(\ref{II-517}) and
(\ref{II-510}), implies{\allowdisplaybreaks%
\begin{align*}
&  \left(  \varphi_{n-1}Q_{n}+\frac{n+1}{2}Q_{n-3}G_{n}\right)  Q_{n-3}%
^{2}G_{n-1}\\
\equiv &  \left(  \frac{2n-1}{2}G_{n-1}Q_{n}+\frac{n+1}{2}Q_{n-3}G_{n}\right)
Q_{n-3}^{2}G_{n-1}\text{ }\operatorname{mod}Q_{n-1}\\
=  &  \frac{3n}{2}H(\varphi_{n-1};Q_{n-3}Q_{n-2}Q_{n-1})\\
&  -\frac{n(2n-1)}{4}G_{n-1}H^{\prime}(\varphi_{n-1};Q_{n-3}Q_{n-2}Q_{n-1})\\
\equiv &  \frac{3n}{2}\left(  \frac{2n-1}{2}G_{n-1}\right)  ^{3}%
-\frac{3n(2n-1)}{4}G_{n-1}\left(  \frac{2n-1}{2}G_{n-1}\right)  ^{2}\text{
}\operatorname{mod}Q_{n-1}\\
=  &  0.
\end{align*}
}This proves $Q_{n-1}|(\varphi_{n-1}Q_{n}+\frac{n+1}{2}Q_{n-3}G_{n}%
)Q_{n-3}^{2}G_{n-1}$. Since property (2-$(n-1)$) shows that any two of
$\{Q_{n-1},G_{n-1},Q_{n-3}\}$ have no common zeros, we obtain (\ref{II-528}%
).\medskip

\textbf{Step 2-3.} We prove that $\deg Q_{n}=\frac{(n+1)(n+2)}{2}$, $\deg
G_{n}=\frac{3n(n+1)}{2}$ and $\deg R_{n}=n(n+1)+1$. Therefore, property
(1-$n$) holds.

Recall from properties (1-$(n-1)$) and (4-$(n-1)$) that $\deg G_{n-1}%
=\frac{3n(n-1)}{2}$ and{\allowdisplaybreaks%
\begin{align}
\deg(R_{n-1}-Q_{n-3}Q_{n-1})  &  =\deg(R_{n-1}-tQ_{n-3}Q_{n-1}) \label{II-530}%
\\
&  =\deg R_{n-1}=n(n-1)+1.\nonumber
\end{align}
}Moreover, $\deg Q_{m}=\frac{(m+1)(m+2)}{2}$ for any $-2\leq m\leq n-1$ by our
induction. Hence,%
\[
\deg Q_{n-3}Q_{n-2}Q_{n-1}=\deg R_{n-1}Q_{n-2}=\frac{3n(n-1)}{2}+1>\deg
G_{n-1},
\]
by which and (\ref{II-515}), we have {\allowdisplaybreaks%
\begin{align*}
&  \deg(\varphi_{n-1}-tQ_{n-3}Q_{n-2}Q_{n-1})\\
=  &  \deg \left(  R_{n-1}-tQ_{n-3}Q_{n-1}\right)  Q_{n-2}=\frac{3n(n-1)}{2}+1,
\end{align*}
}and similarly,{\allowdisplaybreaks%
\begin{align}
&  \deg(\varphi_{n-1}-Q_{n-3}Q_{n-2}Q_{n-1})\label{II-530-2}\\
=  &  \deg \varphi_{n-1}=\deg Q_{n-3}Q_{n-2}Q_{n-1}=\frac{3n(n-1)}%
{2}+1,\nonumber
\end{align}
}namely{\allowdisplaybreaks%
\begin{align}
\deg H(\varphi_{n-1};Q_{n-3}Q_{n-2}Q_{n-1})  &  =3\deg \varphi_{n-1}%
,\nonumber \\
\deg H^{\prime}(\varphi_{n-1};Q_{n-3}Q_{n-2}Q_{n-1})  &  \leq2\deg
\varphi_{n-1}. \label{II-530-3}%
\end{align}
}From here and (\ref{II-516})-(\ref{II-518}), it is easy to see that $\deg
G_{n}=\frac{3n(n+1)}{2}$, $\deg Q_{n}=\frac{(n+1)(n+2)}{2}$ and $\deg
R_{n}=n(n+1)+1$.\medskip

\textbf{Step 2-4.} We prove that property (4-$n$) holds.

It is easy to see that{\allowdisplaybreaks%
\begin{align*}
\deg R_{n-1}Q_{n-2}Q_{n}  &  =\deg Q_{n-3}Q_{n-2}Q_{n-1}Q_{n}\\
&  >\deg G_{n-1}Q_{n}=\deg G_{n}Q_{n-3}\text{,}%
\end{align*}
}so we derive from (\ref{II-518}) and (\ref{II-530}) that{\allowdisplaybreaks%
\begin{align*}
&  \deg(R_{n}-tQ_{n-2}Q_{n})\\
=  &  \deg(R_{n-1}Q_{n-2}Q_{n}-tQ_{n-3}Q_{n-2}Q_{n-1}Q_{n})-\deg
Q_{n-3}Q_{n-1}\\
=  &  \deg(R_{n-1}-tQ_{n-3}Q_{n-1})\deg Q_{n-2}Q_{n}-\deg Q_{n-3}Q_{n-1}\\
=  &  \deg R_{n-1}\deg Q_{n-2}Q_{n}-\deg Q_{n-3}Q_{n-1}=\deg R_{n}.
\end{align*}
}Similarly, $\deg(R_{n}-Q_{n-2}Q_{n})=\deg R_{n}$.\medskip

\textbf{Step 2-5.} We prove that $G_{n-1}$ and $G_{n}$ have no common zeros.

Recall from (\ref{II-516}) and (\ref{II-527}) that%
\[
Q_{n-3}^{3}G_{n}\equiv \left(  \tfrac{2n-1}{2}H^{\prime}\Big(\tfrac
{\varphi_{n-2}}{Q_{n-4}};Q_{n-3}Q_{n-2}\Big)+G_{n-2}Q_{n-1}\right)
Q_{n-1}^{2}\operatorname{mod}G_{n-1}.
\]
If both $G_{n}(X)=0$ and $G_{n-1}(X)=0$, then property (2-$(n-1)$) shows that
$G_{n-2}(X)\not =0$, $Q_{n-1}(X)\not =0$ and $Q_{n-3}(X)\not =0$. Thus,%
\begin{equation}
\frac{2n-1}{2}H^{\prime}\Big(\frac{\varphi_{n-2}(X)}{Q_{n-4}(X)}%
;Q_{n-3}(X)Q_{n-2}(X)\Big)+G_{n-2}(X)Q_{n-1}(X)=0. \label{II-531}%
\end{equation}
On the other hand, (\ref{II-524-1}) and $n\geq2$ give%
\[
H^{\prime}\Big(\frac{\varphi_{n-2}(X)}{Q_{n-4}(X)};Q_{n-3}(X)Q_{n-2}%
(X)\Big)=\frac{-2}{n-1}G_{n-2}(X)Q_{n-1}(X)\not =0,
\]
which yields a contradiction with (\ref{II-531}).\medskip

\textbf{Step 2-6.} We prove that any two of $\{Q_{n-2},Q_{n-1},Q_{n},G_{n}\}$
have no common zeros.

If $Q_{n-2}(X)=0$, then property (2-$(n-1)$) gives $Q_{n-1}(X)\not =0$,
$Q_{n-3}(X)\not =0$ and $G_{n-1}(X)\not =0$. By (\ref{II-515}) we have
$\varphi_{n-1}(X)=\frac{n-1}{2}G_{n-1}(X)$, which implies{\allowdisplaybreaks%
\begin{align}
H(\varphi_{n-1}(X);Q_{n-3}(X)Q_{n-2}(X)Q_{n-1}(X))  &  =\frac{(n-1)^{3}}%
{8}G_{n-1}(X)^{3},\label{II-532}\\
H^{\prime}(\varphi_{n-1}(X);Q_{n-3}(X)Q_{n-2}(X)Q_{n-1}(X))  &  =\frac
{3(n-1)^{2}}{4}G_{n-1}(X)^{2}.\nonumber
\end{align}
}From here and (\ref{II-516})-(\ref{II-517}), it is easy to see
that{\allowdisplaybreaks%
\begin{align}
G_{n}(X)  &  =\frac{(n-1)^{3}}{8Q_{n-3}(X)^{3}}G_{n-1}(X)^{2}\not =%
0,\label{II-533}\\
Q_{n}(X)  &  =-\frac{(2n+1)(n-1)^{2}}{8Q_{n-3}(X)^{2}}G_{n-1}(X)\not =%
0.\nonumber
\end{align}
}This proves that $Q_{n-2}$ has no common zeros with any of $\{G_{n}%
,Q_{n},Q_{n-1}\}$.

If $Q_{n-1}(X)=0$, then property (2-$(n-1)$) implies both $Q_{n-3}(X)\not =0$
and $G_{n-1}(X)\not =0$. By (\ref{II-525}) we have $\varphi_{n-1}%
(X)=\frac{2n-1}{2}G_{n-1}(X)$. Similarly as (\ref{II-532})-(\ref{II-533}), we
obtain{\allowdisplaybreaks%
\begin{align*}
G_{n}(X)  &  =\frac{(2n-1)^{3}}{8Q_{n-3}(X)^{3}}G_{n-1}(X)^{2}\not =0,\\
Q_{n}(X)  &  =-\frac{(n+1)(2n-1)^{2}}{8Q_{n-3}(X)^{2}}G_{n-1}(X)\not =0.
\end{align*}
}This proves that $Q_{n-1}$ has no common zeros with any of $\{G_{n},Q_{n}\}$.

If both $Q_{n}(X)=0$ and $G_{n}(X)=0$, then $Q_{n-1}(X)Q_{n-2}(X)\not =0$ and
$G_{n-1}(X)\not =0$. Recalling from property (3-$(n-1)$) that $\frac
{\varphi_{n-1}}{Q_{n-3}}$ is a polynomial, it is easy to see from
(\ref{II-516})-(\ref{II-517}) that{\allowdisplaybreaks%
\begin{align*}
H\Big(\frac{\varphi_{n-1}(X)}{Q_{n-3}(X)};Q_{n-1}(X)Q_{n-2}(X)\Big)  &  =0,\\
H^{\prime}\Big(\frac{\varphi_{n-1}(X)}{Q_{n-3}(X)};Q_{n-1}(X)Q_{n-2}(X)\Big)
&  =0,
\end{align*}
}which implies $Q_{n-1}(X)Q_{n-2}(X)=0$, a contradiction. This proves that
$Q_{n}$ has no common zeros with $G_{n}$.\medskip

\textbf{Step 2-7.} We prove that any two of $\{Q_{n-2},Q_{n},R_{n}\}$ have no
common zeros. Then property (2-$n$) holds.

If $Q_{n}(X)=0$, then $Q_{n-1}(X)\not =0$ and $G_{n}(X)\not =0$. Since
$\frac{\varphi_{n-1}}{Q_{n-3}}$ is a polynomial by property (3-$(n-1)$), we
have%
\[
R_{n}(X)=\frac{n+1}{2}\frac{G_{n}(X)}{Q_{n-1}(X)}\not =0.
\]
This proves that $Q_{n}$ has no common zeros with $R_{n}$.

If $Q_{n-2}(X)=0$, as in Step 2-6 we have $Q_{n}(X)\not =0$, $Q_{n-1}%
(X)\not =0$, $Q_{n-3}(X)\not =0$ and $\varphi_{n-1}(X)=\frac{n-1}{2}%
G_{n-1}(X)\not =0$. Then (\ref{II-516})-(\ref{II-518})
imply{\allowdisplaybreaks%
\begin{align}
&  Q_{n-3}(X)Q_{n-1}(X)R_{n}(X)\label{II-534}\\
=  &  \frac{n-1}{2}G_{n-1}(X)Q_{n}(X)+\frac{n+1}{2}Q_{n-3}(X)G_{n}%
(X)\nonumber \\
=  &  n\frac{H(\frac{n-1}{2}G_{n-1}(X);0)}{G_{n-1}(X)Q_{n-3}(X)^{2}}%
-\frac{n(n-1)}{4Q_{n-3}(X)^{2}}H^{\prime}\Big(\frac{n-1}{2}G_{n-1}%
(X);0\Big)\nonumber \\
=  &  -\frac{n(n-1)^{3}}{16Q_{n-3}(X)^{2}}G_{n-1}(X)^{2}\not =0,\nonumber
\end{align}
}namely $R_{n}(X)\not =0$. This proves that $Q_{n-2}$ has no common zeros with
$R_{n}$.\medskip

\textbf{Step 2-8.} We prove $Q_{n-2}|\varphi_{n}$ and $Q_{n}|(R_{n}%
Q_{n-1}-\frac{n+1}{2}G_{n})$. Then property (3-$n$) holds.

Since $\frac{\varphi_{n-1}}{Q_{n-3}}$ is a polynomial by property (3-$(n-1)$),
by (\ref{II-518}) we have%
\[
R_{n}Q_{n-1}-\frac{n+1}{2}G_{n}=\frac{\varphi_{n-1}}{Q_{n-3}}Q_{n}.
\]
This proves $Q_{n}|(R_{n}Q_{n-1}-\frac{n+1}{2}G_{n})$.

Recall from (\ref{II-515}) that $\varphi_{n-1}\equiv \frac{n-1}{2}G_{n-1}$
$\operatorname{mod}Q_{n-2}$. This, together with $\varphi_{n}=R_{n}%
Q_{n-1}+\frac{n}{2}G_{n}$ and (\ref{II-516})-(\ref{II-518}),
implies{\allowdisplaybreaks%
\begin{align*}
&  \varphi_{n}G_{n-1}Q_{n-3}^{3}\\
=  &  \left(  R_{n}Q_{n-1}+\frac{n}{2}G_{n}\right)  G_{n-1}Q_{n-3}^{3}\\
=  &  \left(  \varphi_{n-1}Q_{n}+\frac{2n+1}{2}Q_{n-3}G_{n}\right)
G_{n-1}Q_{n-3}^{2}\\
\equiv &  \left(  \frac{n-1}{2}G_{n-1}Q_{n}+\frac{2n+1}{2}Q_{n-3}G_{n}\right)
G_{n-1}Q_{n-3}^{2}\operatorname{mod}Q_{n-2}\\
=  &  \frac{3n}{2}H(\varphi_{n-1};Q_{n-3}Q_{n-2}Q_{n-1})\\
&  -\frac{n(n-1)}{4}G_{n-1}H^{\prime}(\varphi_{n-1};Q_{n-3}Q_{n-2}Q_{n-1})\\
\equiv &  \frac{3n}{2}\left(  \frac{n-1}{2}G_{n-1}\right)  ^{3}-\frac
{3n(n-1)}{4}G_{n-1}\left(  \frac{n-1}{2}G_{n-1}\right)  ^{2}\operatorname{mod}%
Q_{n-2}\\
=  &  0.
\end{align*}
}Together with property (2-$(n-1)$) that any two of $\{Q_{n-3},G_{n-1}%
,Q_{n-2}\}$ have no common zeros, we conclude that $Q_{n-2}|\varphi_{n}$.

Therefore, the proof is complete.
\end{proof}

\medskip

{\bf Acknowledgements} The research of Z. Chen was supported NSFC (Grant No. 12071240).

\end{document}